\documentclass[11pt,reqno]{amsart}
\usepackage{amsmath,amsthm,amssymb,mathrsfs,stmaryrd,color}
\usepackage[all]{xy}
\usepackage{url}
\usepackage{geometry}
\usepackage[utf8]{inputenc}
\usepackage[T1]{fontenc}

\usepackage{relsize}
\usepackage[bbgreekl]{mathbbol}
\usepackage{amsfonts}
\usepackage{slashed}
\DeclareSymbolFontAlphabet{\mathbb}{AMSb}
\DeclareSymbolFontAlphabet{\mathbbl}{bbold}
\DeclareMathOperator{\WCart}{WCart}
\DeclareMathOperator*{\Cone}{Cone}
\DeclareMathOperator{\Gal}{Gal}
\DeclareMathOperator{\Rep}{Rep}

\DeclareMathOperator{\Ext}{Ext}
\DeclareMathOperator{\TWCart}{\widetilde{WCart}}

\DeclareMathOperator{\Qcoh}{Qcoh}
\newcommand{\Prism}{{\mathlarger{\mathbbl{\Delta}}}}
\usepackage{enumitem}
\usepackage[colorlinks=true,hyperindex, linkcolor=magenta, pagebackref=false, citecolor=cyan,pdfpagelabels]{hyperref}
\usepackage[capitalize]{cleveref}
\usepackage{mathtools}
\usepackage{tikz-cd}
\setlist[enumerate]{itemsep=2pt,parsep=2pt,before={\parskip=2pt}}

\newcommand{\cosimp}[3]{\xymatrix@1{#1 \ar@<.4ex>[r] \ar@<-.4ex>[r] & {\ }#2 \ar@<0.8ex>[r] \ar[r] \ar@<-.8ex>[r] & {\ } #3 \ar@<1.2ex>[r] \ar@<.4ex>[r] \ar@<-.4ex>[r] \ar@<-1.2ex>[r] & \cdots }}
\newcommand{\fp}{{\mathrm{fp}}}

\newcommand{\adjunction}[4]{\xymatrix@1{#1{\ } \ar@<0.3ex>[r]^{ {\scriptstyle #2}} & {\ } #3 \ar@<0.3ex>[l]^{ {\scriptstyle #4}}}}

\begin{document}

\numberwithin{equation}{section}
\newtheorem{theorem}{Theorem}[section]
\newtheorem*{theorem*}{Theorem}
\newtheorem*{definition*}{Definition}
\newtheorem{proposition}[theorem]{Proposition}
\newtheorem{lemma}[theorem]{Lemma}
\newtheorem{corollary}[theorem]{Corollary}

\theoremstyle{definition}
\newtheorem{definition}[theorem]{Definition}
\newtheorem{question}[theorem]{Question}
\newtheorem{remark}[theorem]{Remark}
\newtheorem{warning}[theorem]{Warning}
\newtheorem{example}[theorem]{Example}
\newtheorem{notation}[theorem]{Notation}
\newtheorem{convention}[theorem]{Convention}
\newtheorem{construction}[theorem]{Construction}
\newtheorem{claim}[theorem]{Claim}
\newtheorem{assumption}[theorem]{Assumption}

\crefname{assumption}{assumption}{assumptions}
\crefname{construction}{construction}{constructions}

%The todo box!
\def\todo#1{\textcolor{red}%
{\footnotesize\newline{\color{red}\fbox{\parbox{\textwidth}{\textbf{todo: } #1}}}\newline}}

%The comment box!
\def\commentbox#1{\textcolor{red}%
{\footnotesize\newline{\color{red}\fbox{\parbox{\textwidth}{\textbf{comment: } #1}}}\newline}}

\newcommand{\dR}{{\mathrm{dR}}}
\newcommand{\qc}{q-\mathrm{crys}}
\newcommand{\Ainf}{{A_{\mathrm{inf}}}}
\newcommand{\Ain}{{\mathbb{A}_{\mathrm{inf}}}}
\newcommand{\Shv}{\mathrm{Shv}}
\newcommand{\Vect}{\mathrm{Vect}}
\newcommand{\et}{\mathrm{\acute{e}t}}
\newcommand{\eh}{\mathrm{\acute{e}h}}
\newcommand{\proet}{\mathrm{pro\acute{e}t}}
\newcommand{\crys}{\mathrm{crys}}
\renewcommand{\inf}{\mathrm{inf}}
\newcommand{\Hom}{\mathrm{Hom}}
\newcommand{\RHom}{\mathrm{RHom}}
\newcommand{\Sch}{\mathrm{Sch}}
\newcommand{\fSch}{\mathrm{fSch}}
\newcommand{\Rig}{\mathrm{Rig}}
\newcommand{\Spf}{\mathrm{Spf}}
\newcommand{\Spa}{\mathrm{Spa}}
\newcommand{\Spec}{\mathrm{Spec}}
\newcommand{\Bl}{\mathrm{Bl}}
\newcommand{\perf}{\mathrm{perf}}
\newcommand{\Perf}{\mathrm{Perf}}
\newcommand{\Pic}{\mathrm{Pic}}
\newcommand{\qsyn}{\mathrm{qsyn}}
\newcommand{\perfd}{\mathrm{perfd}}
\newcommand{\arc}{{\rm arc}}
\newcommand{\conj}{\mathrm{conj}}
\newcommand{\rad}{\mathrm{rad}}
\newcommand{\Id}{\mathrm{Id}}
\newcommand{\coker}{\mathrm{coker}}
\newcommand{\im}{\mathrm{im}}
\newcommand{\Cond}{\mathrm{Cond}}
\newcommand{\CHaus}{\mathrm{CHaus}}
\newcommand{\cg}{\mathrm{cg}}
\newcommand{\topo}{\mathrm{top}}
\newcommand{\Top}{\mathrm{Top}}
\newcommand{\ac}{\mathbb{A}_{\mathrm{crys}}}
\newcommand{\bc}{\mathbb{B}_{\mathrm{crys}}^+}
\newcommand{\oc}{\mathcal{O}_{\Prism}[[\frac{\mathcal{I}_{\Prism}}{p}]]}
\newcommand{\ocn}{\mathcal{O}_{\Prism}[[\frac{\mathcal{I}_{\Prism}}{p}]]/(\mathcal{I}_{\Prism}/p)^n}
\newcommand{\ocm}{\mathcal{O}_{\Prism}[[\frac{\mathcal{I}_{\Prism}}{p}]]/(\mathcal{I}_{\Prism}/p)^m}
\newcommand{\Ab}{\mathrm{Ab}}
\newcommand{\Set}{\mathrm{Set}}
\newcommand{\Pro}{\mathrm{Pro}}
\newcommand{\Ind}{\mathrm{Ind}}
\newcommand{\Sm}{\mathrm{SmRig}}
\newcommand{\HTlog}{\operatorname{HTlog}}
\newcommand{\HT}{\operatorname{HT}}
\newcommand{\Mod}{\mathrm{Mod}}
\newcommand{\MIC}{\mathrm{MIC}}
\newcommand{\nil}{\mathrm{Nil}}
\newcommand{\bdr}{{\mathbb{B}_{\mathrm{dR}}^{+}}}
\newcommand{\an}{{\mathrm{an}}}
\newcommand{\uHom}{\underline{\mathrm{Hom}}}
\newcommand{\Sym}{\operatorname{Sym}}
\newcommand{\Lie}{\operatorname{Lie}}
\newcommand{\fib}{\operatorname{fib}}
\newcommand{\Eq}{\operatorname{Eq}}
\newcommand{\TX}{\widetilde{X^{\Prism}}}
\newcommand{\TXn}{\widetilde{X^{\Prism}_{[n]}}}
\newcommand{\TXL}{\widetilde{\Spf(\mathcal{O}_L)^{\Prism}_{[n]}}}
\newcommand{\TXM}{\widetilde{X^{\Prism}_{[n-1]}}}
\newcommand{\TXV}{\widetilde{X^{\Prism}_{[1]}}}
\newcommand{\TR}{\widetilde{\Spf(R^+)^{\Prism}}}
\newcommand{\TRn}{\widetilde{\Spf(R^+)_n^{\Prism}}}
\setcounter{tocdepth}{1}

\title{On the prismatization of $\mathcal{O}_K$ beyond the Hodge-Tate locus}
\author{Zeyu Liu}
\address{Department of Mathematics, University of California Berkeley, 970 Evans Hall, MC 3840 Berkeley, CA 94720}
\email{zeliu@berkeley.edu}
\begin{abstract}
Let $X=\Spf(\mathcal{O}_K)$. We classify $\Perf((X)_{\Prism}, \mathcal{O}_{\Prism}/\mathcal{I}_{\Prism}^n)
$ when $n\leq 1+\frac{p-1}{e}$ by studying perfect complexes on $X_n^{\Prism}$, which are certain nilpotent thickenings of $X^{\HT}$ inside $X^{\Prism}$, the prismatization of $X$. We describe the category of continuous semilinear representations and their cohomology for $G_K$ with coefficients in $B_{\dR,n}^+$
via rationalization of vector bundles on slight shrinking of $X_n^{\Prism}$. Along the way, we classify $\Perf((X)_{\Prism}, \ocn)$ for all $n$, which should be viewed as integral models for de Rham prismatic crystals studied in \cite{liu2023rham}.
\end{abstract}

\maketitle

\tableofcontents
\section{Introduction}
In this paper, we work with a $p$-adic field $K$. More precisely, let $\mathcal{O}_K$ be a complete discrete valuation ring of mixed characteristic with fraction field $K$ and perfect residue field $k$ of characteristic $p$. 

Recently, based on the pioneering work of Bhatt and Lurie in \cite{bhatt2022absolute} and \cite{bhatt2022prismatization}, Johannes Anschütz, Ben Heuer and Arthur-César Le Bras studied the Hodge-Tate crystals over $\mathcal{O}_K$ in \cite{anschutz2022v} via a stacky approach. Actually, as the category of Hodge-Tate crystals on $(\mathcal{O}_K)_{\Prism}$ is equivalent to the category of vector bundles on the Hodge-Tate stack $\Spf(\mathcal{O}_K)^{\HT}$, it suffices to study the later. Under such a stacky perspective, the unramified case (i.e. $\mathcal{O}_K=W(k)$) was already treated in \cite{bhatt2022absolute}, while the general description was obtained in \cite{anschutz2022v}. More notably, such a stacky approach naturally leads to results for non-abelian coefficients, In particular, \cite{anschutz2022v} obtained a non-abelian version of \cite{gao2023hodge}.
%while \cite{MW21} only works at the abelian level. 

Then %viewing (rational) Hodge-Tate crystals over $\mathcal{O}_K$ as the "torsion" case of de Rham prismatic crystals, 
it is natural to ask whether we could study $\operatorname{Vect}((\mathcal{O}_K)_{\Prism}, \mathcal{O}_{\Prism}/\mathcal{I}_{\Prism}^n)
$ for any $n\geq 1$ via a stacky approach and we would like to give a partial answer in this paper. Actually, we construct $\WCart_{n}$, which are certain closed substacks of the Cartier-Witt stack $\WCart$ and could be viewed as nilpotent thickenings of $\WCart_{1}=\WCart^{\HT}$, satisfying that quasi-coherent complexes on $\WCart_{n}$ parametrizes prismatic crystals on $(\mathbb{Z}_p)_{\Prism}$ with coefficients in $\mathcal{O}_{\Prism}/\mathcal{I}_{\Prism}^n$. Moreover, we end up with a characterization of quasi-coherent complexes on $\WCart_{n}$ for $n\leq p$, generalizing the description of that for $\WCart^{\HT}$ (i.e. $n=1$) in \cite{bhatt2022absolute}.
\begin{theorem}[{\cref{thm.main classification}}]\label{intro.main thm1}
    Assume that $n\leq p$. There exists a  functor 
    \begin{align*}
        \beta_{n}^{+}: \mathcal{D}(\WCart_n)\rightarrow \mathcal{D}(\mathrm{MIC}(\mathbb{Z}_p[[\lambda]]/\lambda^n)) , \qquad \mathscr{E}\mapsto (\rho^{*}(\mathscr{E}),\Theta_{\mathscr{E}})
    \end{align*}
    such that $\beta_{n}^{+}$ is fully faithful. Moreover, the essential image of $\beta_{n}^{+}$ consists of those objects $M\in  \mathcal{D}(\mathrm{MIC}(\mathbb{Z}_p[[\lambda]]/\lambda^n))$ satisfying the following pair of conditions:
    \begin{itemize}
        \item $M$ is $\mathbb{Z}_p$-complete.
        \item The action of $\Theta^p-\Theta$ on the cohomology $\mathrm{H}^*(M\otimes^{\mathbb{L}}\mathbb{F}_p)$ is locally nilpotent.
    \end{itemize}
\end{theorem}
Let us briefly explain notations in this theorem. 

From the construction of $\WCart_n$, there exists a faithfully flat morphism $\rho: \Spf(\mathfrak{S}_n)\to \WCart_n$, where $\mathfrak{S}=\mathbb{Z}_p[[\lambda]]$ is the Breuil-Kisin prism (with $I=(\lambda$)) and $\mathfrak{S}_n=\mathfrak{S}_n/\lambda^n$. When $n\leq p$, we show that there exists a Sen operator $\Theta_{\mathscr{E}}$ on $\rho^{*}(\mathscr{E})$, the pullback of $\mathscr{E}\in \Qcoh(\WCart_n)$ along $\rho$ satisfying certain Leibniz rule (see \cref{lem.leibniz} for the exact statement).  

However, for $n\geq 2$, as $\Theta$ doesn't vanish on the structure sheaf, which is a key difference with the Hodge-Tate case, we need to be a little careful when describing the target of $\beta_{n}^{+}$, see \cref{def.MIC} for the exact definition of $\mathcal{D}(\mathrm{MIC}(\mathbb{Z}_p[[\lambda]]/\lambda^n))$.

\begin{remark}
     Our results are new when $n\geq 2$. Related results for the isogeny category (i.e. with $p$-inverted) of vector bundles on $\WCart_n$ were obtained in \cite{liu2023rham} and \cite{gao2022rham}. First, those results only hold at the abelian level (i.e. for vector bundles) while our results work for non-abelian coefficients (i.e. quasi-coherent complexes). More notably, the methods and results developed there could only describe the isogeny category $\Vect(\WCart_n)[\frac{1}{p}]$ when $n\geq 2$ and it's unclear to see how to refine them to integral levels, while our new results hold in the integral level.
\end{remark}
\begin{remark}
    It turns out that the key step is to construct the Sen operator $\Theta$ on $\rho^{*}(\mathscr{E})$, which requires $n\leq p$. Actually, once this is done, a standard  d\'evissage reduces to the Hodge-Tate case (i.e. $n=1$) studied explicitly in \cite{bhatt2022absolute}. Such a phenomenon always happens when studying prismatic crystals. For example, when we studied de Rham prismatic crystals in \cite{liu2023rham}, the key difficulty was to extract a Sen operator from the stratification data.
\end{remark}
\begin{remark}
    Actually there is a geometric explanation for \cref{intro.main thm1} using deformation theory, pointed to us by Sasha Petrov. Namely, for $1\leq n \leq p$, there is an isomorphism between $\WCart_n$ and $\Sym^{<n}_{\WCart^{\HT}} \mathcal{O}\{1\}$, the relative stack over $\WCart^{\HT}$ formed by the coherent sheaf $\Sym^{<n}(\mathcal{O}\{1\})$, which is the quotient of the symmetric algebra of $\mathcal{O}\{1\}$ by the ideal of elements of degree at least $n$. We refer the reader to \cref{prop.geometric explanation} for details.
\end{remark}
\begin{remark}
    One might wonder what happens when $n>p$. As $\WCart$ should be viewed as the colimit of $\WCart_n$, the difficulty of studying quasi-coherent complexes on $\WCart_n$ approaches that of understanding quasi-coherent complexes on $\WCart$, where we \textbf{shouldn't} expect that such a classification holds. Actually, as suggested by $n\leq p$, if such a theory exists, then the $\mathbb{Z}_p$-linear Sen operator $\Theta$ on $\mathfrak{S}=\mathbb{Z}_p[[\lambda]]$ should send $\lambda^i$ to $i\lambda^i$, leading to the "wrong" cohomology of the structure sheaf. Indeed, via topological methods, in \cite{bhatt2019topological} Bhatt-Morrow-Scholze showed that $$H^{1}((\mathbb{Z}_p)_{\Prism}, \mathcal{O}_{\Prism})=\prod_{n\in \mathbb{N}} \mathbb{Z}_p.$$
\end{remark}

On the other hand, for any $\mathcal{O}_K$ (without unramified assumption), when working with de Rham prismatic crystals instead, in \cite{liu2023rham} we showed for any {\small $\mathcal{E}\in \operatorname{Vect}((\mathcal{O}_{K})_{\Prism}, (\mathcal{O}_{\Prism}[\frac{1}{p}])_{\mathcal{I}}^{\wedge})$,} %
its evaluation at the Breuil-Kisin prism is equipped with a $t$-connection $\Theta$ preserving the $t$-adic filtration (see \cite[Theorem 5.17]{liu2023rham} for details), which implies that for any $\mathcal{E}_m \in \operatorname{Vect}((\mathcal{O}_{K})_{\Prism}, (\mathcal{O}_{\Prism}/{\mathcal{I}}^{m}[1/p])$, its evaluation at the Breuil-Kisin prism is equipped with a $t$-connection $\Theta$. 

In summary, for $m\in \mathbb{N}$, while $\mathcal{F}_m \in \operatorname{Vect}((\mathcal{O}_{K})_{\Prism}, \mathcal{O}_{\Prism}/{\mathcal{I}}^{m})$ might not be realized as a $\mathfrak{S}_m$-module equipped with certain $t$-connection, its rationalization $\mathcal{F}_m[1/p] \in \operatorname{Vect}((\mathcal{O}_{K})_{\Prism}, (\mathcal{O}_{\Prism}/{\mathcal{I}}^{m}[1/p])$
could be realized as a $\mathfrak{S}_m[1/p]$-module equipped with a $t$-connection. Moreover, when $m=1$, such a rationalization process is unnecessary by the work of \cite{bhatt2022absolute,anschutz2022v} via stacky approach as well as the work of \cite{gao2023hodge} using the prismatic site.

Motivated by the discussion above, it is natural to ask whether working with $\mathcal{O}_{\Prism}/{\mathcal{I}}^{m}[1/p]$ is optimal when $m>1$. In other words, could we find a bounded coefficient ring $*$ such that $$\mathcal{O}_{\Prism}/{\mathcal{I}}^{m}\subseteq *\subseteq \mathcal{O}_{\Prism}/{\mathcal{I}}^{m}[1/p]$$
and that $\operatorname{Vect}((\mathcal{O}_{K})_{\Prism}, *)$ (or more greedily, $\Perf((\mathcal{O}_{K})_{\Prism}, *)$ or even $\mathcal{D}((\mathcal{O}_{K})_{\Prism}, *)$ if we make the correct definition) could still be classified using (derived) $*(\mathfrak{S})$-modules with $t$-connections? Moreover, if such a coefficient ring $*$ exists,  we are interested in whether the $p$-radius of $*$ (i.e. the smallest positive integer $k$ such that $k*\subseteq \mathcal{O}_{\Prism}/{\mathcal{I}}^{m}$) depends on $m$ or not. 

When $K=W(k)[\zeta_p][1/p]$, the rationalization of the cyclotomic ring, recently Michel Gros, Bernard Le Stum and Adolfo Quir\'os classified $\operatorname{Vect}((\mathcal{O}_{K})_{\Prism}, \mathcal{O}_{\Prism})$ (hence also $\operatorname{Vect}((\mathcal{O}_{K})_{\Prism}, \mathcal{O}_{\Prism}/{\mathcal{I}}^{m})$ for all $m$) in \cite{gros2023absolute} using absolute $q$-calculus on modules over the $q$-prism instead of the Breuil-Kisin prism. However, for a general $K$, to the best knowledge of the author, the above question is still unknown.

In this paper, we provide a partial answer to this question without any assumption on $K$. In short, we could take $*$ to be $\ocm \subseteq \mathcal{O}_{\Prism}/{\mathcal{I}}^{m}[1/p]$, the ring obtained by adding $\mathcal{I}/p$ to $\mathcal{O}_{\Prism}/{\mathcal{I}}^{m}$ inside $\mathcal{O}_{\Prism}/{\mathcal{I}}^{m}[1/p]$ (in particular, $(\mathcal{I}/p)^m=0$ in *) and this works for all $m\in \mathbb{N}$. Actually, we expect that it should be the ``smallest coefficient ring” in which a Sen operator could still be defined for a general
$p$-adic field $K$ and any $m\in \mathbb{N}$. %(so the $p$-radius of $*$ doesn't need to increase as $m$ increases). 

Similarly as the proof of \cref{intro.main thm1}, we use the stacky approach to hit such a question. Namely, as there is a morphism $\mu$ from $\WCart$ to $[\widehat{\mathbb{A}}^1/\mathbb{G}_m]$ by sending a Catier-Witt divisor $I\rightarrow W(R)$ a generalized Cartier divisor obtained from projection $W(R)\rightarrow R$ (see \cref{remark.second definition} or \cite[Remark 3.1.6]{bhatt2022absolute} for details), for a bounded $p$-adic formal scheme $X$, we could define $\TXn$ to be the base change of $X^{\Prism}\to [\widehat{\mathbb{A}}^1/\mathbb{G}_m]$ along $[\Spf(\mathbb{Z}[[ t/p]]/(t/p)^n)/ \mathbb{G}_m]\to [\widehat{\mathbb{A}}^1/\mathbb{G}_m]$, then it is not hard to see perfect complexes on $\TXn$ parametrize perfect complexes with coefficients in $*$ on $X_{\Prism}$, i.e. 
$$\Perf(\TXn) \stackrel{\simeq}{\longrightarrow} \Perf((\mathcal{O}_{K})_{\Prism}, \ocn).$$ With such a dictionary in hand, the following theorem could help us understand prismatic crystals with coefficients in $*$.

\begin{theorem}[\cref{thmt.main classification}]\label{intro. main theorem 2}
    Let $X=\Spf(\mathcal{O}_{K})$ and $n\in \mathbb{N}$. There exist functors
    \begin{align*}
        &\beta^{+}: \mathcal{D}(\TX) \rightarrow \mathcal{D}(\mathrm{MIC}(\mathfrak{S}[[\frac{E}{p}]]) , \qquad \mathscr{E}\mapsto (\rho^{*}(\mathscr{E}),\Theta_{\mathscr{E}})
        \\ resp. ~&\beta_n^{+}: \mathcal{D}(\TXn) \rightarrow \mathcal{D}(\mathrm{MIC}(\mathfrak{S}[[\frac{E}{p}]]/(\frac{E}{p})^n), \qquad \mathscr{E}\mapsto (\rho^{*}(\mathscr{E}),\Theta_{\mathscr{E}})
    \end{align*}
    such that $\beta^{+}$ (resp. $\beta_n^{+}$) is fully faithful with an essential image consisting of those objects $M\in \mathcal{D}(\mathrm{MIC}(\mathfrak{S}[[\frac{E}{p}]]))$ (resp. $M\in \mathcal{D}(\mathrm{MIC}(\mathfrak{S}[[\frac{E}{p}]]/(\frac{E}{p})^n) $) satisfying the following pair of conditions:
    \begin{itemize}
        \item $M$ is $\mathbb{Z}_p$-complete.
        \item The action of $\Theta^p-(E^{\prime}(u))^{p-1}\Theta$ on the cohomology $\mathrm{H}^*(M\otimes^{\mathbb{L}}k).$ %\footnote{Here the derived tensor product means the derived base change along $\mathfrak{S}[[\frac{E}{p}]])\to \mathfrak{S}[[\frac{E}{p}]])/(\frac{E}{p},u)=k$} is locally nilpotent. 
    \end{itemize}
\end{theorem}
\begin{remark}
    When $n\leq 1+\frac{p-1}{e}$, where $e$ is the degree of the Eisenstein polynomial $E(u)$ (hence $e$ is intrinsic to $\mathcal{O}_K$), the theorem could be strengthened by replacing the left-hand side of $\beta_n^+$ with $\mathcal{D}(X^{\Prism}_n)$ and replacing the right-hand side by $\mathcal{D}(\mathrm{MIC}(\mathfrak{S}/E^n)$, see \cref{rem.third stren} for details. Consequently, we get a slight generalization of \cref{intro.main thm1} (as when $e=1, ~~ 1+\frac{p-1}{1}=p$).
\end{remark}

\subsection{A new perspective of the $p$-adic Riemann-Hilbert correspondence}
As \cite{anschutz2022v} reinterpreted the $p$-adic Simpson correspondence from the perspective via Hodge-Tate stack, \cref{intro. main theorem 2} gives some new perspective on the $p$-adic Riemann-Hilbert correspondence. We explain it in more detail, following \cite[Section 1.2]{anschutz2022v}.

For $R$ an integral perfectoid ring, $\Spf(R)^{\Prism}_n$ is naturally isomorphic to \\$\Spf(\Ainf(R)/I^n)$ and $\TRn$ is naturally identified with $\Spf(\Ainf(R^+)[[\frac{I}{p}]]/(\frac{I}{p})^n)$. Take $R=\mathcal{O}_C$, the natural morphism $\Spf(\mathcal{O}_C)\to X=\Spf(\mathcal{O}_K)$ induces a $G_K$-equivariant morphism 
$$\Spf(\Ainf/I^n) \rightarrow X_n^{\Prism}, \qquad \Spf(\Ainf[[\frac{I}{p}]]/(\frac{I}{p})^n)\to \TXn.$$ 

Then by pullback we get canonical functors
\[
  \Perf(X_n^{\Prism})[\frac{1}{p}] \stackrel{f_n}{\longrightarrow} \Perf(\TXn)[\frac{1}{p}]\stackrel{\alpha_{n,K}}{\longrightarrow} \Perf(\Spa(K)_v, \bdr_{,n}).
\]
In summary, we have the following diagram \[
\begin{tikzcd}
	& \Perf(\TXn)[\frac{1}{p}] \arrow[ld, "\alpha_{n,K}"'] \arrow[rd, "\beta_n"] &                      \\
	\Perf(\Spa(K)_v, \bdr_{,n}) \arrow[rr, dotted] &                                                                    & {\Perf(\mathrm{MIC}(\mathfrak{S}/E^n[1/p]))}
\end{tikzcd}\]
Notice that for any $\mathcal{E}\in \Perf(X^{\HT})[\frac{1}{p}]$, $\beta_1(\mathcal{E})=(\rho^{*}(\mathscr{E})[\frac{1}{p}],\Theta_{\mathscr{E}})$ with $\rho^{*}(\mathscr{E})[\frac{1}{p}]$ a perfect complex over $K$, for any $n\geq  1$, we then let $$M_{\mathcal{E},n}=\rho^{*}(\mathscr{E})[\frac{1}{p}]\otimes_{K} \mathfrak{S}/E^n[1/p]\footnote{Here we utilize the fact that $K$ could be embedded canonically into $\mathfrak{S}/E^n[1/p]$ arguing as \cite[Lemma 2.5]{liu2023rham}. Moreover, the canonical Sen operator on $\mathfrak{S}/E^n[1/p]$ vanishes on $K$.}$$ 

and equip it with a Sen operator $\Theta_{\mathscr{E},n}$ such that
$$\Theta_{\mathscr{E},n}(x\otimes a)=\Theta_{\mathscr{E}}(x)\otimes a+x\otimes \Theta(a).$$

In this way, we obtain a sequence of objects $(M_{\mathcal{E},n},\Theta_{\mathscr{E},n})\in \Perf(\mathrm{MIC}(\mathfrak{S}/E^n[1/p]))$ compatible with $n$. 
Consequently, given the full faithfulness of $\alpha_n$ stated in the next theorem, for those $\mathcal{F}=\alpha_1(\mathcal{E})\in \Perf(\Spa(K)_v, \hat{\mathcal{O}})$ living in the image of $\alpha_1$, $\varprojlim (M_{\mathcal{E},n},\Theta_{\mathscr{E},n})$ determines an object in $\Perf(\mathrm{MIC}(\bdr(\mathfrak{S})))$, hence we get a functor  
$$\alpha_{1,K}(\Perf(X^{\HT})[\frac{1}{p}])\rightarrow \Perf(\mathrm{MIC}(\bdr(\mathfrak{S})))\footnote{As a ring, $\bdr(\mathfrak{S})$ is very simple, for example, there is an isomorphism $K[[t]]\stackrel{\simeq}{\to} \bdr(\mathfrak{S})$ by sending $t$ to $E$, see \cite[Lemma 2.5]{liu2023rham}},$$ 
which deserves to be viewed as a $p$-adic \textbf{Riemann-Hilbert} functor via the Cartier-Witt stack perspective.

The next result is a generalization of \cite[Theorem 1.3]{anschutz2022v} and helps us fully understand $\alpha_n$. 
\begin{theorem}[{\cref{thm.main p adic RH correspondencd,propt.relate with representations}}]\label{intro.main 3}
   For any finite Galois extension $L/K$ the functor
  \[
   \alpha_{n,L}^{*}\colon \Perf(\TXL)[1/p]\to \Perf(\Spa (L)_v, \bdr_{,n})
  \]
  is fully faithful and induces a fully faithful functor
  \[
   \alpha_{n,L/K}^{*} \colon \Perf([\TXL/{\Gal(L/K)}])[1/p] \to \Perf(\Spa(K)_v, \bdr_{,n})
  \]
  on $\Gal(L/K)$-equivariant objects. 
Each $\mathcal{E}\in \Perf(\Spa(K)_v, \bdr_{,n})$ lies in the essential image of $\alpha_{n,K}^{*}$ if and only if it is nearly de Rham. 
Consequently, we get an equivalence
  \[
   2\text{-}\varinjlim\limits_{L/K} \Perf([\TXL/\mathrm{Gal}(L/K)])[1/p] \cong \Perf(\mathrm{Spa}(K)_v, \bdr_{,n}),
   \]
   where $L$ runs over finite Galois extensions  of $K$ contained in $\overline{K}$. 
\end{theorem}

As a byproduct of the above theorem, we obtain the following characterization of (truncated) de Rham prismatic crystals in perfect complexes, answering a conjecture in \cite{liu2023rham}.
\begin{corollary}[{\cref{cort.compare with de Rham}}]
     $\Perf((\mathcal{O}_{K})_{\Prism}, \bdr_{})$ (resp. $\Perf((\mathcal{O}_{K})_{\Prism}, \bdr_{,n})$), the category of (resp. $n$-truncated) de Rham prismatic crystals in perfect complexes is equivalent to the following two categories:
\begin{itemize}
    \item The category of complexes $M\in \Perf(\mathrm{MIC}(\bdr(\mathfrak{S})))$ (resp. \\$\Perf(\mathrm{MIC}(\mathfrak{S}/E^n[1/p]))$) such that $H^\ast(M)$ is finite dimensional over $\bdr(\mathfrak{S})$ (resp. $\mathfrak{S}/E^n[1/p]$) and the action of $\Theta^p-(E^{\prime}(u))^{p-1}\Theta$ on  $H^\ast(M)$ is topologically nilpotent (with respect to the $(p,E)$-adic topology).
    \item The category of (resp. $n$-truncated) nearly de Rham perfect complexes, i.e. perfect complexes $M$ of $B_{\dR}^+$ (resp. $B_{\dR,n}^+$)-modules equipped with a (continuous) semilinear $G_K$ action all of the cohomology groups of $M\otimes^L C$ are nearly Hodge-Tate representations of $G_K$.
\end{itemize}
Moreover, let $n\in \mathbb{N}$, then 
$$\Perf(\TXn)[1/p] \cong \Perf((\mathcal{O}_{K})_{\Prism}, \bdr_{,n}).$$
In other words, every $\bdr_{,n}$ prismatic crystal comes from an $\ocn$-prismatic crystal.
\end{corollary}
\begin{remark}
    When restricted to the abelian level (i.e. restricted to vector bundles), the statements before the "moreover" part were obtained in \cite{gao2022rham} using prismatic site. Our methods are independent of theirs. Moreover, it's hard to see how the methods developed in \cite{liu2023rham} and \cite{gao2022rham} could be refined to find an integral model (i.e. an $\ocn$-prismatic crystal) inside a $n$-truncated de Rham prismatic crystal when $n>1$.
\end{remark}
\begin{remark}
    Recently a $\log$-version of the Cartier-Witt stack has been developed by Olsson, whose quasi-coherent complexes parameterize $\log$ prismatic crystals, we expect all of our results should have a $\log$-version and we will pursue this generality in a subsequent paper \cite{Liu24b}.
\end{remark}
\begin{remark}
We expect all of the above results to hold if we replace $X=\Spf(\mathcal{O}_K)$ with a quasi-syntomic $p$-adic formal scheme over $\Spf(\mathcal{O}_K)$. Actually, for $n=1$, such generalized results for a smooth $p$-adic formal scheme over $\mathcal{O}_K$ are recently obtained in \cite{anschutz2023hodge}. We will work with the relatively smooth setting for all $n$ in the subsequent paper \cite{Liu24a}. 

On the other hand, as a baby example towards the locally complete intersection direction, by applying the techniques developed so far, we have the following results classifying $n$-truncated prismatic crystals over $Y=\Spf(\mathbb{Z}_p/p^m)$ ($m\geq 2$) for $n\leq p$.  
\end{remark}
\begin{theorem}[{\cref{thm.appl. main classification}}]\label{thm.appl.intro}
     Assume that $n\leq p$. The functor 
    \begin{equation*}
        \beta_n^{+}: \mathcal{D}(Y_n^{\Prism}) \rightarrow \mathcal{D}(\mathrm{MIC}(\mathfrak{S}\{\frac{p^m}{\lambda}\}_{\delta}^{\wedge}/\lambda^n), \qquad \mathscr{E}\mapsto (\rho^{*}(\mathscr{E}),\Theta_{\mathscr{E}}),
    \end{equation*}
    is fully faithful. Moreover, its essential image consists of those objects \\$M\in \mathcal{D}(\mathrm{MIC}(\mathfrak{S}\{\frac{p^m}{\lambda}\}_{\delta}^{\wedge}/\lambda^n))$ satisfying the following %pair of 
    condition:
    \begin{itemize}
        %\item $M$ is $\mathbb{Z}_p$-complete.
        \item The action of $\Theta^p-\Theta$ on the cohomology $\mathrm{H}^*(M\otimes^{\mathbb{L}}\mathbb{F}_p)$ is locally nilpotent.
    \end{itemize}
\end{theorem}

\subsection*{Outline}
The paper is organized as follows. In section 2 we define $\WCart_n$ and its relative versions. Section 3 explains the construction of the Sen operator on $\rho^{*}(\mathscr{E})$ for $\mathscr{E} \in \mathcal{D}(\WCart_n)$ and we prove \cref{intro.main thm1}. In section 4 we study (truncated) $\oc$-prismatic crystals and prove \cref{intro. main theorem 2}. Next in section 5 we study the $v$-realization of perfect complexes on $\TXn$ and complete the proof of \cref{intro.main 3}. Finally in section 6, we study the diffracted $n$-truncated prismatization of $Y=\Spf(\mathbb{Z}_p/p^m)$ ($m\geq 2$) and obtain \cref{thm.appl.intro}.

\subsection*{Notations and conventions}
\begin{itemize}
    \item In this paper $\mathcal{O}_K$ is a complete discrete valuation ring of mixed characteristic with fraction field $K$ and perfect residue field $k$ of characteristic $p$. Fix a uniformizer $\pi$ of $\mathcal{O}_K$. $E(u)$ is its Eisenstein polynomial. $e$ is the degree of $E(u)$.
    \item For $X$ a $p$-adic bounded formal scheme, $X^{\Prism}$ (resp. $X^{\HT}$) is the prismatization of $X$ (resp. the Hodge-Tate stack of $X$) defined as $\WCart_X$ (resp. $\WCart_X^{\HT}$) in \cite{bhatt2022absolute} and \cite{bhatt2022prismatization}. But when $X=\mathbb{Z}_p$, we stick to the original notion $\WCart$ and $\WCart^{\HT}$.
\end{itemize}

\subsection*{Acknowledgments}
The influence of the work of Bhatt and Lurie \cite{bhatt2022absolute,bhatt2022prismatization} and that of Anschütz, Heuer and Le Bras \cite{anschutz2022v} on this paper should be obvious to readers, we thank them for their pioneering and wonderful work. The author benefited a lot from the discussions with Johannes Anschütz, Juan Esteban Rodr\'iguez Camargo and Arthur-César Le Bras when he was visiting the trimester program “The Arithmetic of the
Langlands Program” at the Hausdorff Insitute for Mathematics, funded by the
Deutsche Forschungsgemeinschaft under Germany’s Excellence Strategy – EXC2047/1 – 390685813, we are deeply indebted to their help during that period. We thank Frank Calegari, Ana Caraiani, Laurent Fargues,
and Peter Scholze for their efforts in organizing the amazing trimester and for funding the author's visit. Special thanks to Johannes Anschütz and Sasha Petrov for many useful suggestions and discussions as well as feedback on a preliminary version of this paper. The author receives constant help and support from his advisor Kiran Kedlaya throughout the writing of this paper, and we are very grateful for it. During the preparation of the project, the author was partially supported by NSF DMS-2053473 under Professor Kedlaya. This work was part of the author’s Ph.D. thesis.

\section{Nilpotent thickenings of the Hodge-Tate stack}
\subsection{Certain locus inside the Cartier-Witt stack}
Motivated by Bhatt and Lurie's definition of $\WCart^{\mathrm{HT}}$ (see \cite{bhatt2022absolute}), we define certain nilpotent thickenings of $\WCart^{\mathrm{HT}}$ inside $\WCart$: 
\begin{definition}\label{defi.truncated}
    Let $R$ be a $p$-nilpotent commutative ring. Fix a positive integer $n$, We let $\WCart^{\mathrm{HT}}_{n}(R)$ denote the full subcategory of $\WCart(R)$
spanned by those Cartier-Witt divisors $\alpha: I \rightarrow W(R)$ for which the composite map $I^{\otimes n} \xrightarrow{\alpha^{\otimes n}} W(R) \twoheadrightarrow R$
is equal to zero. The construction $R \mapsto \WCart^{\mathrm{HT}}_n(R)$ determines a closed substack of the Cartier-Witt stack $\WCart$.
We denote this closed substack by $\WCart_n$. 
\end{definition}
\begin{remark}
    When $n=1$, $\WCart_1$ coincides with $\WCart^{\mathrm{HT}}$ and we will switch freely between these two notations. In general, $\WCart_{n}$ could be viewed as a infinitesimal thickening of $\WCart^{\mathrm{HT}}$.
\end{remark}

\begin{remark}\label{remark.second definition}
    As discussed in \cite[Remark 3.1.6]{bhatt2022absolute},  given a Cartier-Witt divisor $I\rightarrow W(R)$, its base change along the restriction map $W(R)\rightarrow R$ is a generalized Cartier divisor. Consequently this determines a morphism of stacks $\mu: \WCart \to [\mathbb{A}^1/\mathbb{G}_m]$, which actually factors through the substack $[\widehat{\mathbb{A}}^1/\mathbb{G}_m]$ as the image of $I$ in $R$ is nilpotent. From this point of view, unwinding \cref{defi.truncated}, we see that the diagram 
    $$ \xymatrix@R=50pt@C=50pt{ \WCart_n \ar[d] \ar@{^{(}->}[r] & \WCart \ar[d]^{\mu}\\
[\Spf(\mathbb{Z}[[t]]/t^n)/\mathbb{G}_m] \ar@{^{(}->}[r]^-{} & [\widehat{\mathbb{A}}^1/\mathbb{G}_m] }$$
is a pullback square, which gives an equivalent definition of $\WCart_n$.
    
\end{remark}

\begin{remark}\label{remark.n-th HT-point-of-prismatic-stack}(Relations with prisms)
Let $(A,I)$ be a prism and regard the commutative ring $A$ as equipped with the $(p,I)$-adic topology. By \cite[Construction 3.2.4]{bhatt2022absolute}, there is a morphism  $\rho_{A}: \Spf(A) \rightarrow \WCart$ sending a $(p,I)$-nilpotent $A$-algebra $R$ to the Cartier-Witt divisor $I\otimes_A W(R)\rightarrow W(R)$ obtained via base change from the inclusion $I\rightarrow A$ (Here we implicitly use the fact that the homomorphism $A\rightarrow R$ uniquely lifts to $\delta$-algebra homomorphism $A\rightarrow W(R)$).
Then $\rho_{A}$ carries the formal subscheme $\Spf(A/I^n) \subset \Spf(A)$ to $\WCart^{\mathrm{HT}}_n$,
and therefore restricts to a morphism $\rho_{n,A}: \Spf(A/I^n) \rightarrow \WCart_n$. Moreover,
the diagram
$$ \xymatrix@R=50pt@C=50pt{ \Spf(A/I^n) \ar[d]^-{ \rho_{n,A}} \ar@{^{(}->}[r] &  \ar[d]^-{\rho_{A} } \Spf(A) \\
 \WCart_n \ar@{^{(}->}[r] & \WCart }$$
is a pullback square. When $n$ varies, these diagrams are compatible.
\end{remark}
\begin{remark}[Quasi-coherent complexes on $\WCart_n$]\label{remark.quasicoherent complex}
    Given any prism $(A,I)$, $\rho_{A,n}: \Spf(A/I^n) \rightarrow \WCart_n$ defined in \cref{remark.n-th HT-point-of-prismatic-stack} induces a functor from {\small $\mathcal{D}(\WCart_n)$} to the $p$-complete derived $\infty$-category
    $\widehat{\mathcal{D}}(A/I^n)$. Utilizing the same strategy as in \cite[Proposition 3.3.5]{bhatt2022absolute}, we end in an equivalence of categories
    $$\mathcal{D}(\WCart_{n}) \xrightarrow{\sim} \lim _{(A, I)\in (\mathbb{Z}_p)_{\Prism}} \widehat{\mathcal{D}}(A/I^{n}).$$
    Similar results hold for perfect complexes and vector bundles.
\end{remark}

Finally we justify that $\WCart_{n}$ are the correct objects to study for understanding $n$-truncated prismatic crystals.

\begin{proposition}\label{stacky information for n}
The category of vector bundles on $\Vect(\WCart_{n})$ is equivalent to the category of prismatic crystals on $(\mathbb{Z}_p)_{\Prism}$ with coefficients in ${\mathcal{O}}_{\Prism}/\mathcal{I}_{\Prism}^n$, i.e.
    \[\Vect(\WCart_{n})\cong \operatorname{Vect}((\mathbb{Z}_p)_{\Prism}, {\mathcal{O}}_{\Prism}/\mathcal{I}_{\Prism}^n).\]
Similar results hold if we replace vector bundles with perfect complexes.
\end{proposition}
\begin{proof}
   By \cref{remark.quasicoherent complex}, we need to show that $$\operatorname{Vect}((\mathbb{Z}_p)_{\Prism}, {\mathcal{O}}_{\Prism}/\mathcal{I}_{\Prism}^n) \xrightarrow{\sim} \lim _{(A, I)} \operatorname{Vect}(A/I^{n}) $$
   and that 
   $$\Perf((\mathbb{Z}_p)_{\Prism}, {\mathcal{O}}_{\Prism}/\mathcal{I}_{\Prism}^n) \xrightarrow{\sim} \lim _{(A, I)} \Perf(A/I^{n}), $$
   but this follows from $p$-completely faithfully flat for vector bundles and perfect complexes discussed in \cite[Proposition 2.7]{bhatt2021prismatic}.
\end{proof}

\subsection{Relative case}
Following \cite[Construction 3.7]{bhatt2022prismatization}, we can generalize the previous construction to any bounded $p$-adic formal scheme $X$.
\begin{construction}\label{ConsAbsHT}
Let $X$ be a bounded $p$-adic formal scheme. Form a fiber square
\[ \xymatrix{ X_{n}^{\Prism} \ar[r] \ar[d] & X^{\Prism} \ar[d] \\
\mathrm{WCart}_n \ar[r] & \mathrm{WCart} }\]
%\[ \WCart_X^{\mathrm{HT}} = \WCart_X \times_{\WCart} \WCart^{\mathrm{HT}}.\] 
defining the closed substack $X_{n}^{\Prism}$ inside $X^{\Prism}$, the prismatization of $X$. Given a $p$-nilpotent ring $R$ and a Cartier-Witt divisor $(I \xrightarrow{\alpha} W(R)) \in \WCart_n^{\mathrm{HT}}(R) \subset \WCart(R)$, the map $\alpha^{\otimes n}: I^{\otimes n}\rightarrow W(R)$ factors over $VW(R) \subset W(R)$, so there is an induced map $W(R)/{}^{\mathbb{L}} I^n \to W(R)/VW(R) \simeq R$. On the other hand, $\alpha^{\otimes n}$ can be written as $\alpha^{\otimes n}=\alpha\circ (\Id\otimes \alpha^{\otimes (n-1)})$, hence induces a map $W(R)/{}^{\mathbb{L}} I^n \rightarrow W(R)/{}^{\mathbb{L}} I\rightarrow R/{}^{\mathbb{L}} I \rightarrow R/{}^{\mathbb{L}}\alpha(I) \rightarrow R/\alpha(I)$
Consequently, given a point $( (I \xrightarrow{\alpha} W(R)), \eta:\mathrm{Spec}(\overline{W(R)}) \to X) \in \WCart_{n,X}^{\mathrm{HT}}(R)$, one obtains a map $\overline{\eta}:\mathrm{Spec}(R/\alpha(I)) \to \mathrm{Spec}(\overline{W(R)}) \xrightarrow{\eta} X$. This construction defines a functor  
\[ \pi^{\mathrm{HT}}:X_{n}^{\Prism} \to X\] 
by sending $R$ to $R/\alpha(I)$, which we refer to as {\em the Mod-$I$ Hodge-Tate morphism}.
\end{construction}

\begin{remark}[$W(k)$-structure on $X^{\Prism}_n$\label{rem.w(k) structure} for $X$ over $\mathcal{O}_K$]
Fix a qcqs smooth $p$-adic formal scheme $X$ over $\mathcal{O}_K$. Recall that given a $p$-nilpotent ring $T$, an $T$-valued point of $X^{\Prism}_n$ corresponds to a Cartier-Witt divisor $\alpha: I \rightarrow W(T)$ together with a map $\eta:\mathrm{Spec}(\overline{W(T)}) \to X$ such that the map $\alpha^{\otimes n}: I^{\otimes n}\rightarrow W(T)$ factors over $VW(T) \subseteq W(T)$. In particular, we have an induced map $W(T)/{}^{\mathbb{L}} I^n \to W(T)/VW(T) \simeq T$. As $X$ is over $\mathcal{O}_K$, we have the following diagram 
$$ \xymatrix{ & &W(T)/{}^{\mathbb{L}} I^n \ar[d] \ar[r]&  T\ar[d] \\ W(k) \ar[r] & \mathcal{O}_K \ar[r] & \overline{W(T)}  \ar[r] & \overline{T}
}$$
Here $\overline{T}$ is defined to be the pushout of the right square. In particular, $T\to \overline{T}$ is a nilpotent thickening as the left vertical map is. The composition of all the arrows in the bottom line gives a ring homomorphism $W(k)\to \overline{T}$ which uniquely lifts to a ring homomorphism $W(k)\to T$ as $W(k)$ is $p$-completely \'etale over $\mathbb{Z}_p$. In other words, $T$ is an $W(k)$-algebra. Hence we obtain a structure morphism which we denoted as $\pi: X^{\Prism}_n\to \Spf W(k)$. 

Similarly, suppose $R^{+}$ is an (integral) perfectoid ring corresponding to the prism $(A,I)$ and $X$ is a qcqs smooth $p$-adic formal scheme over $R^{+}$, then the same arguments shows that there is a natural structure morphism $\pi: X^{\Prism}_n \rightarrow A$.
\end{remark}

Next we state a relative version of \cref{remark.n-th HT-point-of-prismatic-stack}.
\begin{construction}[From prisms in $X_{\Prism}$ to $X_{n}^{\Prism}$]\label{construction.relative prism to truncated}
    Let $X$ be a bounded $p$-adic formal scheme. Fix an object $(\operatorname{Spf}(A) \leftarrow \operatorname{Spf}(\bar{A}) \rightarrow X) \in X_{\Prism}$, then similarly to \cref{remark.n-th HT-point-of-prismatic-stack}, the morphism $\rho_{X,A}: \Spf(A) \rightarrow X^{\Prism}$ constructed in \cite[Construction 3.10]{bhatt2022prismatization} induces the following fiber square:
$$ \xymatrix@R=50pt@C=50pt{ \Spf(A/I^n) \ar[d]^-{ \rho_{n,X,A}} \ar@{^{(}->}[r] &  \ar[d]^-{\rho_{X,A} } \Spf(A) \\
 X_n^{\Prism} \ar@{^{(}->}[r] & X^{\Prism} }$$
\end{construction}

\begin{example}[The $n$-truncated prismatization of a perfectoid]
    Let $R$ be a perfectoid ring corresponding to the perfect prism $(A,I)$ via \cite[Theorem 3.10]{BhattScholzepPrismaticCohomology}. In this case $\rho_A: \Spf(A) \rightarrow R^{\Prism}$ is an isomorphism of functors by \cite[Example 3.12]{bhatt2022prismatization}, which implies that 
    \[\rho_{n,A}: \Spf(A/I^n) \rightarrow R_n^{\Prism}\]
    is also an isomorphism by the above pullback square.
\end{example}

Finally we translate prismatic crystals on $X_{\Prism}$ with coefficients in $\mathcal{O}_{\Prism}/\mathcal{I}_{\Prism}^n$ to quasi-coherent complexes on $X_{n}^{\Prism}$.
\begin{proposition}
    Assume that $X$ is a quasi-syntomic $p$-adic formal scheme, then there is an equivalence
    $$\mathcal{D}_{qc}(X_{n}^{\Prism}) \xrightarrow{\sim} \lim _{(A, I)\in X_{\Prism}} \widehat{\mathcal{D}}(A/I^{n})=:\mathcal{D}_{\crys}(X_{\Prism}, \mathcal{O}_{\Prism}/\mathcal{I}_{\Prism}^n) $$
    of symmetric monoidal stable $\infty$-categories.
\end{proposition}
\begin{proof}
    This follows from our definition of $X_{n}^{\Prism}$ and \cite[Proposition 8.13]{bhatt2022prismatization}, \cite[Proposition 8.15]{bhatt2022prismatization}.
\end{proof}

\section{Sen operators on truncated Cartier-Witt stacks}
\subsection{Sen operator}
In this subsection we assume $\mathfrak{S}=\mathbb{Z}_p[[u]]=\mathbb{Z}_p[[\lambda]]$ for $E(u)=\lambda=u-p$. Then $(\mathfrak{S},\lambda)$ defines a transversal object in $(\mathbb{Z}_p)_{\Prism}$, hence $\rho_{n}: \Spf(\mathfrak{S}/\lambda^n) \rightarrow \WCart_n$ defined in \cref{remark.n-th HT-point-of-prismatic-stack} is faithfully flat by \cite[Corollary 3.2.10]{bhatt2022absolute}. Later we will omit the subscript $n$ and just write $\rho$ if no confusion. 

In this section we aim to classify quasi-coherent complexes on $\WCart_n$ for $n$ no larger than $p$ by realizing them as objects in $\widehat{\mathcal{D}}(\mathfrak{S}/\lambda^n)$ equipped with an additional Sen operator satisfying certain nilpotence condition. 

For this purpose, the key step is to construct a Sen operator on $\rho^*\mathscr{E}$ for $\mathscr{E}\in \Qcoh(\WCart_n)$ extending the Sen operator on the Hodge-Tate locus constructed in \cite[Section 3.5]{bhatt2022absolute}, which will be based on the following key lemma:
\begin{lemma}\label{lem.construct b}
    Fix $n\leq p$. For $R=\mathfrak{S}/\lambda^n[\epsilon]/\epsilon^2=\mathbb{Z}_p[[\lambda, \epsilon]]/(\lambda^n,\epsilon^2)$, there exists $b$ in $W(R)^{\times}$ such that the following holds:
    \begin{itemize}
        \item $\Tilde{g}(\lambda)=\Tilde{f}(\lambda)\cdot b$, where $\Tilde{g}$ is the unique $\delta$-ring map such that the following diagram commutes:
    \[\begin{tikzcd}[cramped, sep=large]
& W(R) \arrow[d, "p_0"]\\ 
\mathbb{Z}_p[[\lambda]] \arrow[ru, "\Tilde{g}"] \arrow[r, "\lambda \mapsto (1+\epsilon)\lambda"]& R
\end{tikzcd}\]
and $\Tilde{f}$ is defined similarly with the bottom line in the above diagram replaced with the identity map.
    \end{itemize}
\end{lemma}
\begin{proof}
    First we notice that $\Tilde{g}(\lambda)$ is uniquely characterized by the following two properties:
\begin{itemize}
    \item $p_0(\Tilde{g}(\lambda))=(1+\epsilon)\lambda$.
    \item $\varphi(\Tilde{g}(\lambda))=(\Tilde{g}(\lambda)+p)^p-p$.
\end{itemize}
In the following, for simplicity, we will denote $\Tilde{g}(\lambda)$ as $\Tilde{\lambda}$ and $\Tilde{f}(\lambda)$ simply as $\lambda$. We will show the existence of $b=(b_0,b_1,\ldots)\in W(R)$ ($b_i\in R $) satisfying the desired properties. As $R$ is $p$-torsion free, the ghost map is injective, hence the identity $\Tilde{\lambda}=b\lambda$ is equivalent to that 
\begin{align}\label{equa.ghost identity II}
    \forall m\geq 0, w_m(\Tilde{\lambda})=w_m(b)\cdot w_m(\lambda).
\end{align}
Here $w_m$ denotes the $n$-th ghost map.

Recall that $\varphi(\lambda)=(\lambda+p)^p-p$, hence by induction we have the following equality in $R$:
\begin{align}\label{er1}
    w_m(\lambda)=(w_{0}(\lambda)+p)^{p^m}-p=(\lambda+p)^{p^m}-p=p^{p^m}-p+\sum_{i=1}^{p^m}\binom{p^m}{i}p^{p^m-i}\lambda^i.
\end{align}
Similarly, we have that
\begin{align}\label{er2}
    w_m(\Tilde{\lambda})=(w_{0}(\Tilde{\lambda})+p)^{p^m}-p=((1+\epsilon)\lambda+p)^{p^m}-p=p^{p^m}-p+\sum_{i=1}^{p^m}(1+i\epsilon)\binom{p^m}{i}p^{p^m-i}\lambda^i.
\end{align}

Take $m=0$ in \cref{equa.ghost identity}, we want $b_0$ such that $\lambda b_0=(1+\epsilon)\lambda $, hence we could just take $b_0=1+\epsilon$.

Suppose $m\geq 1$ and we have determined $b_0,\cdots,b_{m-1}$ such that \cref{equa.ghost identity II} holds for non-negative integers no larger than $m-1$. Moreover, we assume that $b_i~ (1\leq i\leq m-1)$ is divisible by $\epsilon$. Then we aim to find $b_m=\sum_{j=0}^{\infty}c_{m,j}\lambda^j\in R$ such that \cref{equa.ghost identity II} holds for $m$. For this, first notice that  
\begin{align}\label{er3}
    w_{m}(b)=\sum_{i=0}^{m}p^ib_{i}^{p^{m-i}}=b_{0}^{p^m}+p^mb_m=(1+\epsilon p^m)+p^mb_m.
\end{align}
Here the second identity holds as $b_i^p=0$ by our assumption that $b_i$ is divisible by $\epsilon$ for $1\leq i \leq m-1$.

Combining \cref{er1},\cref{er2} and \cref{er3}, we see that $w_m(\Tilde{\lambda})=w_m(b)\cdot w_m(\lambda)$ can be reinterpreted as 
{\small
\begin{align}\label{equa.infinity many coefficients}
    p^{p^m}-p+\sum_{i=1}^{p^m}(1+i\epsilon)\binom{p^m}{i}p^{p^m-i}\lambda^i=(p^{p^m}-p+\sum_{i=1}^{p^m}\binom{p^m}{i}p^{p^m-i}\lambda^i)((1+\epsilon p^m)+p^m\sum_{i=0}^{\infty}c_{m,i}\lambda^i).
\end{align}}%
As $n\leq p$, by comparing the coefficients of $\lambda^i$ ($1\leq i\leq p-1$) on both sides, this equality holds if we could pick $c_{m,0}=-\epsilon$ and $c_{m, i}$ for each $1\leq i\leq p-1$ such that
\begin{equation}\label{Compare coefficients}
      (p^{p^m}-p)p^mc_{m,i} = 
           -i\epsilon \binom{p^m}{i}p^{p^m-i}+\sum_{j=1}^{i-1}\binom{p^m}{j}p^{p^m-j}p^m c_{m,i-j} 
  \end{equation}
We claim that there is a unique sequence $\{c_{m,i}=\epsilon d_{m,i}\}_{1\leq i\leq p-1}$ with $d_{m,i}\in \mathbb{Z}_p$ satisfying \cref{Compare coefficients} and that $v_{p}(d_{m,i})=p^m-i-1$.

To prove the claim, we argue by induction on $i$. We aim to define $c_{m,i}$ inductively on $i$ such that \cref{Compare coefficients} is satisfied. The divisibility by $\epsilon$ easily follows from induction and we only need to verify that $v_{p}(d_{m,i})=p^m-i-1$. Take $i=1$, \cref{Compare coefficients} implies that $v_{p}(d_{m,1})=(m+p^m-1)-(m+1)=p^m-2$, the claim holds. Assume that $i\geq 
2$ and that we have shown $v_{p}(d_{m,j})=p^m-j-1$ for all $j\leq i-1$, then we calculate the $p$-valuations for the terms on the right-hand side of \cref{Compare coefficients}: $v_{p}(i\epsilon \binom{p^m}{i}p^{p^m-i})=m+p^m-i$ and for $1\leq j\leq i-1$,
\begin{equation*}
v_{p}(\binom{p^m}{j}p^{p^m-j}p^n c_{m,i-j})=(m-v_{p}(j))+(p^m-j)+m+(p^m-(i-j)-1)>m+p^m-i.
\end{equation*}
Consequently $v_{p}(d_{m,i})=(m+p^m-i)-(m+1)=p^m-i-1$, as desired.

%Now suppose $k\geq 1$ and we have shown the claim for all $i\leq kp^n$. For $kp^n+1\leq i\leq (k+1)p^n$, we define $c_{n,i}$ inductively on $i$ such that the second equation in \cref{Compare coefficients} is satisfied. We only need to verify that $v_{p}(c_{n,i})=(k+1)p^n-i-k-1$. As before, suppose $kp^n+1\leq i$ and that we have shown that $c_{n,j}$ has the desired $p$-valuation for all $j\leq i-1$, then we calculate the $p$-valuations for the terms on the right handside of the second equation in \cref{Compare coefficients}: $v_{p}(p^n c_{n,i-p^n})=n+(kp^n-(i-p^n)-k=n+(k+1)p^n-k-i$,
Our construction of $c_{m,i}$ implies that $b_m=\sum_{j=0}^{\infty}c_{m,j}\lambda^j$ is divisible by $\epsilon$, we win.
\end{proof}
\begin{remark}\label{rem.descend of b to the Hodge-Tate locus}
    \begin{itemize}
        \item We fix $n$ as a priori and our element $b$ might depend on $n$, however, the compatibility with $n$ is clear from our construction, hence we omit $n$ and just write $b$ by abuse of notation. In other words, we could just construct such a $b$ in $W(\mathfrak{S}/E^p[\epsilon]/\epsilon^2)$, then its projection to $W(\mathfrak{S}/E^n[\epsilon]/\epsilon^2)$ for $n\leq p$ automatically satisfies the desired properties. In particular, take $n=1$, one can easily see that $b$ in $ W(\mathfrak{S}/\lambda[\epsilon]/\epsilon^2)=W(\mathbb{Z}_p[\epsilon]/ \epsilon^2)$ is just $[1+\epsilon]$.
        \item The reason that we need to assume $n\leq p$ is implicitly given in the above proof. Actually, if $n\geq p$, then we need to compare the coefficients of $\lambda^{p}$ in \cref{equa.infinity many coefficients}, which leads to no solution of $c_{1,p}$ in $\mathbb{Z}_p$, but only in $\mathbb{Q}_p$.
    \end{itemize}
\end{remark}

%Let $R$ be a commutative ring. Recall that by \cite[Remark 3.1.6]{BL22a}, there is a morphism $\mu: \WCart\rightarrow [\widehat{\mathbf{A}}^1 / \mathbf{G}_m]$
\begin{proposition}\label{prop.key automorphism for small n}
    For $n\leq p$, the element $b$ constructed in \cref{lem.construct b} induces an isomorphism $\gamma_b$ between functors $\rho: \Spf(\mathfrak{S}/\lambda^n) \rightarrow \WCart_n$
    and $\rho\circ \delta: \Spf(\mathfrak{S}/\lambda^n) \rightarrow \WCart_n$ after base change to $\Spec(\mathbb{Z}[\epsilon]/(\epsilon^2))$, i.e. we have the following commutative diagram:
\[\xymatrixcolsep{5pc}\xymatrix{\Spf(\mathfrak{S}/\lambda^n)\times \Spec(\mathbb{Z}[\epsilon]/(\epsilon^2))\ar[d]^{\rho}\ar[r]^{\delta: \lambda \mapsto (1+\epsilon)\lambda}& \Spf(\mathfrak{S}/\lambda^n)\times \Spec(\mathbb{Z}[\epsilon]/(\epsilon^2))\ar@{=>}[dl]^b \ar[d]_{}^{\rho}
\\\WCart_{n}\times \Spec(\mathbb{Z}[\epsilon]/(\epsilon^2))  \ar[r]&\WCart_{n}\times \Spec(\mathbb{Z}[\epsilon]/(\epsilon^2))}\]
\end{proposition}
\begin{proof}
      Given a test $p$-nilpotent $R=\mathfrak{S}/\lambda^n[\epsilon]/\epsilon^2$-algebra $T$ via the structure morphism $h$, we denote the induced morphism $W(R)\to W(T)$ by $\Tilde{h}$. Then $\rho\circ h(T)$ corresponds to the point $((\lambda) \otimes_{\mathfrak{S},\Tilde{h}\circ \Tilde{f}} W(T)\to W(T))$ in $\WCart_{n}(T)$, while $(\rho\circ \delta)\circ  h(T)$ corresponds to the point $((\lambda) \otimes_{\mathfrak{S},\Tilde{h}\circ \Tilde{g}} W(T)\to W(T))$ in $\WCart_{n}(T)$. \cref{lem.construct b}  implies that we have an isomorphism of these two Cartier-Witt divisors given by 
     \[\xymatrix{(\lambda) \otimes_{\mathfrak{S},\Tilde{h}\circ \Tilde{g}} W(T) \ar[d]\ar[r]& W(T) \ar[d]_{}^{\Id}
\\(\lambda) \otimes_{\mathfrak{S},\Tilde{h}\circ \Tilde{f}} W(T)\ar[r]&W(T)}\]
Here the left vertical map is a $W(T)$-linear map sending $(\lambda)\otimes 1$ to $(\lambda)\otimes \Tilde{h}(b)$.
\end{proof}
\begin{warning}\label{war.key warning}
    The above proposition still works if we replace $\WCart_n$ with \\$\Spf (W(k))^{\Prism}_n$ ($n\leq p$) thanks to \cref{rem.w(k) structure}. However, for $\mathcal{O}_K$ ramified, it no longer holds. Actually, to give a point in $\Spf(\mathcal{O}_K)^{\Prism}_n(T)$, we need to specify a Cartier-Witt divisor $\alpha$ together with a map (of derived schemes) $\eta: \Spf(\mathrm{Cone}(\alpha))\to \Spf(\mathcal{O}_K)$. Consequently, to construct an isomorphism of $\rho$ and $\rho\circ \delta$, we need not only the isomorphism $b$ between Cartier-Witt divisors but also a homotopy between $\eta$ and $b\circ \eta$ as there are no canonical $\mathcal{O}_K$-structures on these cones. We will pursue this generality in the next section. Most notably, even working with "small" $n$ between $1$ and $p$, we need to add $\frac{I}{p}$ to construct such a homotopy.
\end{warning}
\begin{remark}[Compatibility with the construction in \cite{bhatt2022absolute} on the Hodge-Tate locus]\label{rem.comapre with BL}
    If we consider the restriction of the isomorphism constructed above to the Hodge-Tate locus (i.e. take $n=1$), then as $\delta=\Id$ on $\Spf (\mathbb{Z}_p[\epsilon]/\epsilon^2)$, we see $b$ descends to an automorphism 
    \[\WCart_{n}\times \Spec(\mathbb{Z}[\epsilon]/(\epsilon^2))\to \WCart_{n}\times \Spec(\mathbb{Z}[\epsilon]/(\epsilon^2))\]
    Per \cref{rem.descend of b to the Hodge-Tate locus}, this automorphism is exactly given by multiplication by $[1+\epsilon]$, hence it coincides with the construction in \cite[Section 3.5]{bhatt2022absolute}.
\end{remark}

Fix $n\leq p$. Now we are ready to construct a Sen operator on $\rho^*\mathscr{E}$ for $\mathscr{E}\in \Qcoh(\WCart_n)$. Based on \cref{prop.key automorphism for small n}, we have an isomorphism $\gamma_b: \rho\circ \delta \stackrel{\simeq}{\longrightarrow} \rho$. Consequently, for $\mathscr{E}\in \Qcoh(\WCart_n)$, we have an isomorphism
\[\gamma_b: \delta^{*}\rho^{*}(\mathscr{E}\otimes \mathbb{Z}[\epsilon]/(\epsilon^2)) \stackrel{\simeq}{\longrightarrow} \rho^{*}(\mathscr{E}\otimes \mathbb{Z}[\epsilon]/(\epsilon^2)).\]
Unwinding the definitions, this could be identified with a $\delta$-linear morphism 
\begin{align}\label{equa.truncated induced morphism}
    \gamma_b: \rho^{*}(\mathscr{E})\rightarrow \rho^{*}(\mathscr{E})\otimes \mathbb{Z}[\epsilon]/(\epsilon^2).
\end{align} 
Moreover, our construction of the unit $b$ in \cref{lem.construct b} implies that the morphism $\gamma_b$ in \cref{equa.truncated induced morphism} reduces to the identity modulo $\epsilon$, hence could be written as $\Id+\epsilon\Theta_{\mathscr{E}}$ for some operator $\Theta_{\mathscr{E}}: \rho^{*}\mathscr{E} \rightarrow \rho^{*}\mathscr{E}$. 

\begin{remark}\label{rem.describe the category}
    $\gamma_b$ gives a $\delta$-linear endomorphism on $\rho^{*}(\mathscr{E})[\epsilon]/\epsilon^2 \in \mathcal{D}(\mathbb{Z}_p[[\lambda,\epsilon]]/(\lambda^n,\epsilon^2))$, hence $(\rho^{*}(\mathscr{E}),\gamma_b)$ lies in the category $\mathcal{C}_1$ defined to be the pullback of the following square:
    \[ \xymatrix{ \mathcal{C}_1 \ar[d] \ar[r] &   \mathcal{D}(\mathbb{Z}_p[[\lambda]]/\lambda^n) \ar[d] \\ 
\operatorname{Fun}(\Delta^{1}, \mathcal{D}(\mathbb{Z}_p[[\lambda,\epsilon]]/(\lambda^n,\epsilon^2))) \ar[r] & \operatorname{Fun}(\Delta^{0}\sqcup \Delta^{0}, \mathcal{D}(\mathbb{Z}_p[[\lambda,\epsilon]]/(\lambda^n,\epsilon^2))) }.\]
Here the arrow in the bottom is the forgetful functor sending $(M\xrightarrow[]{f} N)$ to $(M,N)$ for $M, N\in \mathcal{D}(\mathbb{Z}_p[[\lambda,\epsilon]]/(\lambda^n,\epsilon^2))$ and the right vertical map sends $K \in  \mathcal{D}(\mathbb{Z}_p[[\lambda]]/\lambda^n)$ to $(\delta^{*}(K[\epsilon]/\epsilon^2), K[\epsilon]/\epsilon^2)$.

Similarly, one could define the category $\mathcal{C}_2$ to be the pullback of the square
\[ \xymatrix{ \mathcal{C}_2 \ar[d] \ar[r] &   \mathcal{D}(\mathbb{Z}_p[[\lambda]]/\lambda^n) \ar[d] \\ 
\operatorname{Fun}(\Delta^{1}, \mathcal{D}(\mathbb{Z}_p[[\lambda]]/\lambda^n)) \ar[r] & \operatorname{Fun}(\Delta^{0}\sqcup \Delta^{0}, \mathcal{D}(\mathbb{Z}_p[[\lambda]]/\lambda^n)),}\]
then the fact that that the induced $\delta$-linear morphism  $\gamma_b: \rho^{*}(\mathscr{E})\rightarrow \rho^{*}(\mathscr{E})\otimes \mathbb{Z}[\epsilon]/(\epsilon^2)$ is the identity after modulo $\epsilon$ precisely means that after modulo $\epsilon$, $(\rho^{*}(\mathscr{E}),\gamma_b)$ (viewed as an object in $\mathcal{C}_2$) actually lies in $\mathcal{D}(\mathrm{MIC}(\mathbb{Z}_p[[\lambda]]/\lambda^n))$, defined as follows. 
\end{remark}
%Motivated by the above discussion, we make the following definition:
\begin{definition}\label{def.MIC}
    Define $\mathcal{D}(\mathrm{MIC}(\mathbb{Z}_p[[\lambda]]/\lambda^n))$ to be the pullback of the following diagram
\[ \xymatrix{ \mathcal{D}(\mathrm{MIC}(\mathbb{Z}_p[[\lambda]]/\lambda^n)) \ar[r] \ar[d] & \mathcal{C}_1 \ar[d] \\ 
\mathcal{D}(\mathbb{Z}_p[[\lambda]]/\lambda^n) \ar[r] &  \mathcal{C}_2  .}\]
Here the right vertical map is given by modulo $\epsilon$ and the bottom arrow sends $M\in \mathcal{D}(\mathbb{Z}_p[[\lambda]]/\lambda^n)$ to $(M\xrightarrow[]{\Id} M)$.

Similarly, we define $\Perf(\mathrm{MIC}(\mathbb{Z}_p[[\lambda]]/\lambda^n))$ by replacing $\mathcal{D}(\bullet)$ with $\Perf(\bullet)$ ((including such a modification for $\mathcal{C}_1$ and $\mathcal{C}_2$) in the above diagram.
\end{definition}
\begin{remark}\label{rem.explicit description via sen operator}
    More explicitly, following the discussion before \cref{rem.describe the category}, we see that specifying an object in $\mathcal{D}(\mathrm{MIC}(\mathbb{Z}_p[[\lambda]]/\lambda^n))$ is the same as giving a pair $(M,\Theta_M)$ where $M \in \mathcal{D}(\mathbb{Z}_p[[\lambda]]/\lambda^n)$ and $\Theta_M$ is an operator on $M$ satisfying Leibniz rule thanks to the following lemma.
\end{remark}

\begin{lemma}\label{lem.leibniz}
   Given $(M,\gamma_M) \in \mathcal{D}(\mathrm{MIC}(\mathbb{Z}_p[[\lambda]]/\lambda^n))$ where $M\in \mathcal{D}(\mathbb{Z}_p[[\lambda]]/\lambda^n)$ and $\gamma_M: \delta^{*}(M[\epsilon]/\epsilon^2)\to M[\epsilon]/\epsilon^2$, write $\gamma_M$ as $\Id+\Theta_M$ when restricted to $M$, then we have that $$\Theta_M(ax)=a\Theta_M(x)+\Theta(a)x$$ for $a\in \mathbb{Z}_p[[\lambda]]/\lambda^n$ and $x\in M$, here $\Theta: \mathbb{Z}_p[[\lambda]]/\lambda^n\to \mathbb{Z}_p[[\lambda]]/\lambda^n$ is $\mathbb{Z}_p$-linear and sends $\lambda^i$ to $i\lambda^i$. 
\end{lemma}
\begin{proof}
    By assumption, we have that $$\gamma_M(ax)=\delta(a)\gamma_M(x)=(a+\epsilon\Theta(a))(x+\epsilon\Theta_\mathscr{E}(x))=ax+\epsilon(a\Theta_\mathscr{E}(x)+\Theta(a)x).$$
    This implies the desired result.
\end{proof}

\begin{remark}\label{lem.sen operator on pullback}
    Recall that in \cite{bhatt2022absolute}, the Sen operator for complexes on the Hodge-Tate locus is defined on the complex itself, while for $1<n\leq p$ and $\mathscr{E}\in \Qcoh(\WCart_n)$, our Sen operator is defined on $\rho^{*}(\mathscr{E})$ other than on $\mathscr{E}$ itself, this is due to the fact that our construction of the isomorphism $b$ relies on certain coordinates $u$ in the Breuil-Kisin prism when $n>1$. This is a feature but not a bug as we expect that the Sen operator $\Theta$ for the structure sheaf should act nontrivially on $u$, see next example.
\end{remark}
\begin{example}[Sen operator on the ideal sheaf $\mathcal{I}^k$]\label{example.action on generators}
    Fix $n\leq p$. Let $\mathscr{E}$ be the structure sheaf $\mathcal{O}_{\mathrm{WCart}_{n}}$. Then under the trivialization  $\rho^{*}\mathscr{E}=\mathbb{Z}_p[[\lambda]]/\lambda^n$, $b(\lambda^i)=\delta(\lambda^i)=\lambda^i(1+\epsilon)^i=\lambda^i(1+i\epsilon)$, hence $\Theta_{\mathscr{E}}$ sends $\lambda^i$ to $i\lambda^i$. In general, for $\mathscr{E}=\mathcal{I}^k$, one can verify that under the trivialization  $\rho^{*}\mathscr{E}=\mathbb{Z}_p[[\lambda]]/\lambda^n\cdot (\lambda^k)$, $\Theta_{\mathscr{E}}$ sends $\lambda^i$ to $(i+k)\lambda^i$.
\end{example}

Then our main result in this section is the following description of quasi-coherent complexes on $\WCart_n$ for $n\leq p$:
\begin{theorem}\label{thm.main classification}
Assume that $n\leq p$. The functor 
    \begin{align*}
        \mathcal{D}(\WCart_n)\rightarrow \mathcal{D}(\mathrm{MIC}(\mathbb{Z}_p[[\lambda]]/\lambda^n)) , \qquad \mathscr{E}\mapsto (\rho^{*}(\mathscr{E}),\Theta_{\mathscr{E}})
    \end{align*}
    is fully faithful. Moreover, its essential image consists of those objects $M\in \mathcal{D}(\mathrm{MIC}(\mathbb{Z}_p[[\lambda]]/\lambda^n))$ satisfying the following pair of conditions:
    \begin{itemize}
        \item $M$ is $\mathbb{Z}_p$-complete.
        \item The action of $\Theta^p-\Theta$ on the cohomology $\mathrm{H}^*(M\otimes^{\mathbb{L}}\mathbb{F}_p)$ is locally nilpotent.
    \end{itemize}
\end{theorem}
\begin{remark}
    For $n=1$, this is \cite[Theorem 3.5.8]{bhatt2022absolute}, hence our theorem is a generalization of Bhatt and Lurie's description of quasi-coherent complexes on the Hodge-Tate stack. However, our proof of \cref{thm.main classification} will require \cite[Theorem 3.5.8]{bhatt2022absolute} as an input.
\end{remark}

For the proof of \cref{thm.main classification}, we need several preliminaries.

\begin{proposition}\label{generation}
    For any $n\geq 1$, the $\infty$-category $\mathcal{D}(\WCart_n)$ is generated under shifts and colimits by the invertible sheaves $\mathcal{I}^n$ for $n\in \mathbb{Z}$.
\end{proposition}
\begin{proof}
    The proof is essentially the same as that in \cite[Corollary 3.5.16]{bhatt2022absolute}.
\end{proof}
\begin{proposition}\label{commute with colimits}
    For any $n\geq 1$, the global sections functor $\mathrm{R} \Gamma(\mathrm{WCart}_{n}, \bullet): \mathcal{D}(\mathrm{WCart}_{n}) \rightarrow \widehat{\mathcal{D}}(\mathbb{Z}_p)$ commutes with colimits.
\end{proposition}
\begin{proof}
    We prove the claim by induction on $n$. For $n=1$, the desired result follows from \cite[Corollary 3.5.13]{bhatt2022absolute}. Suppose $n\geq 2$ and that we have shown the claim for up to $n-1$. Then as for any prism $(A,I)$, we have a short exact sequence 
    \begin{equation*}
       0\to I\otimes A/I^{n-1}\to A/I^n\to A/I\to 0.
    \end{equation*}
    Via the dictionary transferring prismatic crystals to quasi-coherent complexes on truncated Cartier-Witt stacks by \cref{stacky information for n}, it implies that the closed embeddings $i_{n-1}: \WCart_{n-1} \hookrightarrow \mathrm{WCart}_{n}$ and $i_{\HT}: \WCart^{\HT} \hookrightarrow \WCart_{n}$ induce the following exact sequence on $\WCart_{n}$:
    \begin{equation}\label{equa.structure exact}
        i_{n-1,*}\mathcal{O}_{\mathrm{WCart}_{n-1}}\{1\}\rightarrow \mathcal{O}_{\mathrm{WCart}_{n}} \rightarrow i_{\HT,*}\mathcal{O}_{\mathrm{WCart}^{\HT}}.
    \end{equation}
For any $\mathcal{F}\in \mathcal{D}(\mathrm{WCart}_{n})$, we hence obtain a fiber sequence 
\[\mathcal{F}\otimes (i_{n-1,*}\mathcal{O}_{\mathrm{WCart}_{n-1}}\{1\})\rightarrow \mathcal{F}\otimes \mathcal{O}_{\mathrm{WCart}_{n}} \rightarrow \mathcal{F}\otimes (i_{\HT,*}\mathcal{O}_{\mathrm{WCart}^{\HT}})\]
by tensoring $\mathcal{F}$ with \cref{equa.structure exact}.

Taking global sections then yields the fiber sequence
\begin{equation*}
    \begin{split}
        \mathrm{R} \Gamma(\mathrm{WCart}_{n}, \mathcal{F}\otimes (i_{n-1,*}\mathcal{O}_{\mathrm{WCart}_{n-1}}\{1\})) \rightarrow &\mathrm{R} \Gamma(\mathrm{WCart}_{n}, \mathcal{F}\otimes \mathcal{O}_{\mathrm{WCart}_{n}}) \\ &\rightarrow \mathrm{R} \Gamma(\mathrm{WCart}_{n}, \mathcal{F}\otimes (i_{\HT,*}\mathcal{O}_{\mathrm{WCart}^{\HT}})).
    \end{split}
\end{equation*}
%\[\mathrm{R} \Gamma(\mathrm{WCart}_{n}, \mathcal{F}\otimes (i_{n-1,*}\mathcal{O}_{\mathrm{WCart}_{n-1}}\{1\})) \rightarrow \mathrm{R} \Gamma(\mathrm{WCart}_{n}, \mathcal{F}\otimes \mathcal{O}_{\mathrm{WCart}_{n}}) \rightarrow \mathrm{R} \Gamma(\mathrm{WCart}_{n}, \mathcal{F}\otimes (i_{\HT,*}\mathcal{O}_{\mathrm{WCart}^{\HT}})).\]
As taking colimits is exact, it suffices to show both $\mathrm{R} \Gamma(\mathrm{WCart}_{n}, \mathcal{F}\otimes (i_{n-1,*}\mathcal{O}_{\mathrm{WCart}_{n-1}}\{1\}))$ and $\mathrm{R} \Gamma(\mathrm{WCart}_{n}, \mathcal{F}\otimes (i_{\HT,*}\mathcal{O}_{\mathrm{WCart}^{\HT}}))$ commute with colimits. However the projection formula implies that
\begin{equation*}
    \begin{split}
        \mathrm{R} \Gamma(\mathrm{WCart}_{n}, \mathcal{F}\otimes (i_{n-1,*}\mathcal{O}_{\mathrm{WCart}_{n-1}}\{1\}))&=\mathrm{R} \Gamma(\mathrm{WCart}_{n}, i_{n-1,*}(i_{n-1}^{*}(\mathcal{F})\otimes \mathcal{O}_{\mathrm{WCart}_{n-1}}\{1\}))\\&=\mathrm{R} \Gamma(\mathrm{WCart}_{n-1}, i_{n-1}^{*}\mathcal{F}\{1\}),
    \end{split}
\end{equation*}
here the last equality holds as $i_{n-1,*}=Ri_{n-1,*}$ since the closed immersion $i_{n-1}$ is an affine morphism. Consequently, the functor $\mathrm{R} \Gamma(\mathrm{WCart}_{n}, \bullet\otimes (i_{n-1,*}\mathcal{O}_{\mathrm{WCart}_{n-1}}))$ commutes with colimits as both $i_{n-1,*}$ and $\mathrm{R} \Gamma(\mathrm{WCart}_{n-1}, \bullet)$ commute with colomits (by induction). Similarly, one can prove that the functor $\mathrm{R} \Gamma(\mathrm{WCart}_{n}, \bullet\otimes (i_{\HT,*}\mathcal{O}_{\mathrm{WCart}^{\HT}}))$ also commutes with colimits. We are done.
\end{proof}

Recall that for $\mathcal{E}\in \mathcal{D}(\WCart_n)$, the global section of $\mathcal{E}$ is defined as 
\begin{equation*}
    \mathrm{R} \Gamma(\mathrm{WCart}_{n}, \mathcal{E}):=\lim_{f: \Spec(R)\to \mathrm{WCart}_{n}} f^{*}\mathcal{E}
\end{equation*}

In particular, the cover $\rho: \Spf(\mathfrak{S}/\lambda^n) \rightarrow \WCart_n$ induces a natural morphism $$\mathrm{R} \Gamma(\mathrm{WCart}_{n}, \mathcal{E})\to \rho^{*}\mathcal{E}.$$

Utilizing our construction of the Sen operator and results from \cite{bhatt2022absolute}, we can understand $\mathrm{R} \Gamma(\mathrm{WCart}_{n}, \mathcal{E})$ quite well through this morphism:
\begin{proposition}\label{sen calculate cohomology}
    Let $n\leq p$, then for any $\mathcal{E}\in \mathcal{D}(\WCart_n)$, the natural morphism $\mathrm{R} \Gamma(\mathrm{WCart}_{n}, \mathcal{E})\to \rho^{*}\mathcal{E}$ induces a canonical identification $$
    \mathrm{R} \Gamma(\mathrm{WCart}_{n}, \mathcal{E})=\fib(\rho^*\mathcal{E}\stackrel{\Theta_{\mathscr{E}}}{\longrightarrow} \rho^*\mathcal{E}).$$
\end{proposition}
    %By the functoriality of $\Theta$ and \cref{generation}, \cref{commute with colimits}, both sides commute with taking colimits, hence it suffices to show that the identification in the proposition holds for $\mathcal{E}=\mathcal{I}^n$. For such $\mathcal{E}$, we invoke the calculation in \cite{Liu22} to calculate $\mathrm{R} \Gamma(\mathrm{WCart}_{n}, \mathcal{E})$.
\begin{proof}
By \cref{prop.key automorphism for small n}, we have an isomorphism $b: \rho\circ \delta \stackrel{\simeq}{\longrightarrow} \rho$ as functors  $\Spf(\mathfrak{S}/\lambda^n [\epsilon]/\epsilon^2) \rightarrow \WCart_n$. Then the definition of $ \mathrm{R} \Gamma(\mathrm{WCart}_{n}, \mathcal{E})$ implies that the natural morphism $\mathrm{R} \Gamma(\mathrm{WCart}_{n}, \mathcal{E})\to \rho^{*}\mathcal{E}$ factors through the equalizer of $$\rho^*\mathcal{E}\stackrel{\Id\otimes 1}{\longrightarrow} 
 \rho^*\mathcal{E}\otimes \mathbb{Z}[\epsilon]/(\epsilon^2)$$
and 
$$\rho^*\mathcal{E}\stackrel{b}{\longrightarrow} 
 \rho^*\mathcal{E}\otimes \mathbb{Z}[\epsilon]/(\epsilon^2),$$
where $b=\Id+\epsilon \Theta_{\mathcal{E}}$ is defined in \cref{equa.truncated induced morphism}. This produces a canonical morphism 
$$
    \mathrm{R} \Gamma(\mathrm{WCart}_{n}, \mathcal{E})\to \fib(\rho^*\mathcal{E}\stackrel{\Theta_{\mathscr{E}}}{\longrightarrow} \rho^*\mathcal{E}).$$
To see that it is actually an identification, using standard  d\'evissage (the trick we used in the proof of \cref{commute with colimits}), by induction on $n$, this reduces to $n=1$, which follows from \cite[Proposition 3.5.11]{bhatt2022absolute} thanks to \cref{rem.comapre with BL}.

\end{proof}
%\begin{lemma}For $k\in \mathbb{Z}$, $\mathrm{R} \Gamma(\mathrm{WCart}_{n}, \mathcal{I}^k)=(\mathbb{Z}_p[[\lambda]]/\lambda^n\stackrel{\lambda^i\mapsto (i-k)\lambda^i}{\longrightarrow} \mathbb{Z}_p[[\lambda]]/\lambda^n)$.\end{lemma}
%\begin{proof}$\mathcal{I}^k(\mathfrak{S})=(\lambda^k)\cdot \mathbb{Z}_p[[\lambda]]/\lambda^n$ could be identified with $e\cdot \mathbb{Z}_p[[\lambda]]/\lambda^n$ (here $e$ corresponds to $\lambda^k$), hence $$\varepsilon (e)=e\cdot (\frac{\eta}{\lambda})^k=e\cdot (1-X_1)^k.$$\end{proof}

With all of the above ingredients in hand, finally we are ready to prove \cref{thm.main classification}.
\begin{proof}[Proof of \cref{thm.main classification}]
    The functor is well defined thanks to \cref{lem.leibniz}. Then we follow the proof of \cite[Theorem 3.5.8]{bhatt2022absolute}. 
    For the full faithfulness, let $\mathcal{E}$ and $\mathcal{F}$ be quasi-coherent complexes on $\WCart_n$ and we want to show that the natural map 
    \begin{equation*}
        \Hom_{\mathcal{D}(\WCart_n)}(\mathscr{E}, \mathscr{F}) \rightarrow \Hom_{\mathcal{D}(\mathrm{MIC}(\mathbb{Z}_p[[\lambda]]/\lambda^n))} (\rho^{*}(\mathscr{E}), \rho^{*}(\mathscr{F}))
    \end{equation*}
    is a homotopy equivalence. Thanks to \cref{generation}, we could reduce to the case that $\mathscr{E}=\mathcal{I}^k$ for some $k\in \mathbb{Z}$. Replacing $\mathscr{F}$ by the twist $\mathscr{F}(-k)$, we could further assume that $k=0$. Then the desired result follows from \cref{sen calculate cohomology}. 

    To check that the action of $\Theta^p-\Theta$ on the cohomology $\mathrm{H}^*(\rho^{*}(\mathscr{E})\otimes^{\mathbb{L}}\mathbb{F}_p)$ is locally nilpotent for $\mathscr{E}\in \mathcal{D}(\WCart_n)$, again thanks to \cref{generation}, we might assume $\mathscr{E}=\mathcal{I}^k$ for some $k\in \mathbb{Z}$. Then by \cref{example.action on generators}, $\eta^{*}(\mathcal{I}^k)=e\cdot \mathbb{Z}_p[[\lambda]]/\lambda^n$ (here we identity $e$ with $\lambda^k$) and under this trivialization  %$\rho^{*}\mathscr{E}=\mathbb{Z}_p[[\lambda]]/\lambda^n\cdot (\lambda^k)$, 
    $\Theta_{\mathscr{E}}$ sends $\lambda^i$ to $(i+k)\lambda^i$. As for any integer $j$, $j^p\equiv j \mod p$, we conclude that $\Theta^p-\Theta$ acts nilpotently on each $\lambda^i$, hence so on $\rho^{*}(\mathcal{I}^k)$.

    Let $\mathcal{C}\subseteq \mathcal{D}(\mathrm{MIC}(\mathbb{Z}_p[[\lambda]]/\lambda^n))$ be the full subcategory spanned by objects satisfying two conditions listed in \cref{thm.main classification}. As the source $\mathcal{D}(\WCart_n)$ is generated under shifts and colimits by the invertible sheaves $\mathcal{I}^n$ for $n\in \mathbb{Z}$ by \cref{generation}, to complete the proof it suffices to show that $\mathcal{C}$ is also generated under shifts and colimits by $\{\rho^*\mathcal{I}^k\}$ ($k\in \mathbb{Z}$). In other words, we need to show that for every nonzero object $M\in \mathcal{C}$, $M$ admits a nonzero morphism from $\rho^*\mathcal{I}^k[m]$ for some $m,k \in \mathbb{Z}$. Replacing $M$ by $M\otimes \mathbb{F}_p$, we may assume that there exists some cohomology group $\mathrm{H}^{-m}(M)$ containing a nonzero element killed by $\Theta^p-\Theta=\prod_{0\leq n<p} (\Theta-n)$ (this could be done by iterating the action of $\Theta^p-\Theta$ and then use the nilpotence assumption). Furthermore, we could assume this element is actually killed by $\Theta-k$ for a single integer $k$. It then follows that there exists a non-zero morphism from $\rho^*\mathcal{I}^k[m]$ to $M$ which is nonzero in degree $m$.
\end{proof}

We end this section with the following geometric characterization of $\WCart_n$, observed by Sasha Petrov. It would essentially lead to \cref{thm.main classification} with some extra work, although we don't pursue this further in this paper.
\begin{proposition}\label{prop.geometric explanation}
    There is a unique isomorphism between $\WCart_n$ and $\Sym^{<n}_{\WCart^{\HT}} \mathcal{O}\{1\}$ as stacks over $\Spf(\mathbb{Z}_p)$, here the later is the relative stack over $\WCart^{\HT}$ formed by the coherent sheaf $\Sym^{<n}(\mathcal{O}\{1\})$, the quotient of the symmetric algebra of $\mathcal{O}\{1\}$ by the ideal of elements of degree at least $n$.
\end{proposition}
\begin{proof}
    For any $n$, unwinding the definition of $\WCart_{n}$ via \cref{remark.second definition}, we see that the ideal sheaf corresponding to the closed embedding $\WCart_{n} \hookrightarrow \mathrm{WCart}_{n+1}$ is supported on $\WCart^{\HT}$ and is isomorphic to $\mathcal{O}\{n\}$. Moreover, $\mathrm{WCart}_{n+1}$ is a square-zero thickening of $\WCart_{n}$ living over the $n$-th order neighbourhood of $B\mathbb{G}_m$ inside $[\widehat{\mathbb{A}}^1/\mathbb{G}_m]$. On the other hand, such square-zero thickenings are classified by $$\Ext^1_{\WCart_{n}}(\mathbb{L}_{\WCart_{n}/[\Spf(\mathbb{Z}[[t]]/t^n)/\mathbb{G}_m]}, \mathcal{O}\{n\})=\Ext^1_{\WCart^{\HT}}(\mathbb{L}_{\WCart^{\HT}/B\mathbb{G}_m}, \mathcal{O}\{n\}),$$
    where the identity comes from the adjunction of the closed embedding above and base change for the cotangent complex. 

    Now it suffices to show that $\RHom_{\WCart^{\HT}}(\mathbb{L}_{\WCart^{\HT}/B\mathbb{G}_m}, \mathcal{O}\{n\})=0$ for $1\leq n\leq p-1$, which will also give the uniqueness of the above isomorphism. By considering $\WCart^{\HT} \to B\mathbb{G}_m \to \Spf(\mathbb{Z}_p)$, we have the following exact triangle of cotangent complexes
    $$\mathbb{L}_{B\mathbb{G}_m/\mathbb{Z}_p}\otimes_{\mathcal{O}_{B\mathbb{G}_m}} \mathcal{O}_{\WCart^{\HT}} \to \mathbb{L}_{\WCart^{\HT}/\mathbb{Z}_p} \to \mathbb{L}_{\WCart^{\HT}/\mathbb{G}_m}.$$

    We claim that both $\RHom_{\WCart^{\HT}}(\mathbb{L}_{B\mathbb{G}_m/\mathbb{Z}_p}\otimes_{\mathcal{O}_{B\mathbb{G}_m}} \mathcal{O}_{\WCart^{\HT}}, \mathcal{O}\{n\})$ and $\RHom_{\WCart^{\HT}}(\mathbb{L}_{\WCart^{\HT}/\mathbb{Z}_p}, \mathcal{O}\{n\})$ are zero when $1\leq n\leq p-1$, leading to the desired vanishing of $\RHom_{\WCart^{\HT}}(\mathbb{L}_{\WCart^{\HT}/B\mathbb{G}_m}, \mathcal{O}\{n\})$. We will prove $\RHom_{\WCart^{\HT}}(\mathbb{L}_{\WCart^{\HT}/\mathbb{Z}_p}, \mathcal{O}\{n\})=0$ for $1\leq n\leq p-1$ in detail, and the vanishing of the former term will follow for a similar reason. For this purpose, first recall that for any flat group $\mathbb{Z}_p$-scheme $G$, the cotangent complex $\mathbb{L}_{BG/\mathbb{Z}_p}$ is identified with $e^*\mathbb{L}_{G/\mathbb{Z}_p}[-1]$ equipped with the adjoint action of G, where $e:\Spf (\mathbb{Z}_p) \to G$ is the unit section. Applying the above discussion to $\WCart^{\HT}$, which is isomorphic to the classifying stack of $\mathbb{G}_m^{\sharp}$ by \cite[Theorem 3.4.13]{bhatt2022absolute}, we see that $\mathbb{L}_{\WCart^{\HT}/\mathbb{Z}_p}\cong \mathbb{L}_{B\mathbb{G}_m^{\sharp}}=e^*\mathbb{L}_{\mathbb{G}_m^{\sharp}/\mathbb{Z}_p}[-1]$ is a $\mathbb{Z}_p$-module equipped with the trivial adjoint action of $\mathbb{G}_m^{\sharp}$ as $\mathbb{G}_m^{\sharp}$ is a commutative group scheme. Consequently we conclude that $\mathbb{L}_{\WCart^{\HT}/\mathbb{Z}_p}$ (viewed as an object in $\mathcal{D}(\WCart^{\HT})$) comes from pullback of a certain object in $\mathcal{D}(\mathbb{Z}_p)$ via the structure morphism $f: \WCart^{\HT}$ to $\Spf(\mathbb{Z}_p)$. 
    
    However, for any $\mathscr{E}=f^*M \in \mathcal{D}(\WCart^{\HT})$ (here $M\in \mathcal{D}(\mathbb{Z}_p)$), it just corresponds to the pair $(M, 0)$ as an object in $\mathcal{D}(\mathbb{Z}_p[\Theta])$ under the full faithful embedding from $\mathcal{D}(\WCart^{\HT})$ to $\mathcal{D}(\mathbb{Z}_p[\Theta])$ given in \cite[Theorem 3.5.8]{bhatt2022absolute}. Hence for $1\leq n\leq p-1$
    $$\RHom_{\mathcal{D}(\WCart^{\HT})}(\mathscr{E}, \mathcal{O}\{n\})=\RHom_{\mathcal{D}(\mathbb{Z}_p[\Theta])}((M,0), (\mathbb{Z}_p,n))=0.$$
    Here the first identity comes from \cite[Theorem 3.5.8]{bhatt2022absolute} and the second equality is due to \cite[Lemma 6.1]{petrov2023non} as multiplication by $n$ is invertible.

    Now we win by applying the above discussion to $\mathscr{E}=\mathbb{L}_{\WCart^{\HT}/\mathbb{Z}_p}$.
\end{proof}
\begin{remark}
    It is interesting to study the extension class in $\Ext^1_{\WCart^{\HT}}(\mathbb{L}_{\WCart^{\HT}/B\mathbb{G}_m}, \mathcal{O}\{p\})$ corresponding to the closed immersion $\WCart_{p} \hookrightarrow \mathrm{WCart}_{p+1}$, which should be a non-zero element. We plan to delve into this further in the near future. 
\end{remark}
\section{(Truncated)-\texorpdfstring{$\oc$}\quad -prismatic crystals}
In this section we study prismatic crystals with coefficients in (truncated) $\oc$ for arbitrary $p$-adic field $K$. The motivation is that $\oc \subseteq \bdr$ should be the "smallest coefficient ring" in which a Sen operator could still be defined for a general $p$-adic field $K$. 

Again, under the philosophy that prismatic crystals with coefficients derived from ${\mathcal{O}}_{\Prism}$ should correspond to quasi-coherent complexes on certain restricted locus of the Cartier-Witt stack, we introduce the following definition:
\begin{definition}
    \[ \TWCart:=\WCart \times_{[\widehat{\mathbf{A}}^1 / \mathbf{G}_m]} [\Spf(\mathbb{Z}[[ \frac{t}{p}]])/ \mathbf{G}_m].\]  
Here the structure morphism $\mu:
\WCart \rightarrow [\widehat{\mathbf{A}}^1 / \mathbf{G}_m]$ is reviewed in \cref{remark.second definition}.

Moreover, let $\TWCart_{[n]}$ be the closed substack of $\TWCart$ determined by $(\frac{t}{p})^n=0$.
\end{definition}
\begin{example}\label{compare two versions}
    When $n=1$, we see that $\TWCart_{[1]}=\WCart^{\HT}$.
\end{example}
\begin{definition}
    Let $X$ be a bounded $p$-adic formal scheme. $\TX$ is defined to be the fiber square of the following diagram
\[ \xymatrix{ X_{n}^{\Prism} \ar[r] \ar[d] & X^{\Prism} \ar[d] \\
\TWCart \ar[r] & \mathrm{WCart} }\]
Similarly we get the definition of $\TXn$ by replacing $\TWCart$ in the bottom left corner with $\TWCart_{[n]}$.
%\[ \WCart_X^{\mathrm{HT}} = \WCart_X \times_{\WCart} \WCart^{\mathrm{HT}}.\] 
\end{definition}

\subsection{Construction of Sen operator}
In this subsection, we work with a general $p$-adic field $K$. More precisely, let $\mathcal{O}_K$ be a complete discrete valuation ring of mixed characteristic with fraction field $K$ and perfect residue field $k$ of characteristic $p$. Fix a uniformizer $\pi$ of $\mathcal{O}_K$ and Let $(\mathfrak{S}=W(k)[[u]],E(u))$ be the corresponding Breuil-Kisin prism. Denote $X=\Spf(\mathcal{O}_K)$. Then $\rho: \Spf(\mathfrak{S}[[\frac{E}{p}]]) \rightarrow \TX$ is a faithfully flat cover as it is the base change of the faithfully flat cover $\Spf(\mathfrak{S})\to X^{\Prism}$.

Similarly as before, our construction of a Sen operator on $\rho^*\mathscr{E}$ for $\mathscr{E}\in \Qcoh(\TX)$ will be based on the following several key lemmas.
\begin{lemma}\label{lemt.extend delta}
The $W(k)[\epsilon]/\epsilon^2$-linear homomorphism $\delta: \mathfrak{S}[\epsilon]/\epsilon^2 \to  \mathfrak{S}[\epsilon]/\epsilon^2$ sending $u$ to $u+\epsilon E(u)$ extends uniquely to a ring homomorphism $ \mathfrak{S}[[\frac{E}{p}]][\epsilon]/\epsilon^2 \to \mathfrak{S}[[\frac{E}{p}]][\epsilon]/\epsilon^2$, which will still be denoted as $\delta$ by abuse of notation.
\end{lemma}
\begin{proof}
    If we define $\delta (\frac{E}{p})=\frac{E(u)(1+\epsilon E^{\prime}(u))}{p}$, then $\delta$ first extends to a ring homomorphism $\mathfrak{S}[\frac{E}{p}][\epsilon]/\epsilon^2 \to \mathfrak{S}[[\frac{E}{p}]][\epsilon]/\epsilon^2$, as $\delta (\frac{E}{p})$ is topologically nilpotent in the target, this further extends to 
    \[\delta: \mathfrak{S}[[\frac{E}{p}]][\epsilon]/\epsilon^2 \to \mathfrak{S}[[\frac{E}{p}]][\epsilon]/\epsilon^2.\]
    The uniqueness can be checked easily.
\end{proof}

\begin{lemma}\label{lemT.construct b}
    For $R=\mathfrak{S}[[\frac{E}{p}]][\epsilon]/\epsilon^2$, there exists a unique $b$ in $W(R)^{\times}$ such that the following holds:
    \begin{itemize}
        \item $\Tilde{g}(\lambda)=\Tilde{f}(\lambda)\cdot b$, where $\Tilde{g}$ is the unique $\delta$-ring map such that the following diagram commutes:
    \[\begin{tikzcd}[cramped, sep=large]
& W(R) \arrow[d, "p_0"]\\ 
\mathfrak{S} \arrow[ru, "\Tilde{g}"] \arrow[r, "\delta: u \mapsto u+\epsilon E(u)"]& R
\end{tikzcd}\]
and $\Tilde{f}$ is defined similarly with the bottom line in the above diagram replaced with the canonical embedding $\mathfrak{S}\hookrightarrow R$.
    \end{itemize}
\end{lemma}
\begin{proof}
    Notice that the $W(k)$-linear algebra morphism $\Tilde{g}$ is uniquely characterized by the following two properties:
\begin{itemize}
    \item $p_0(\Tilde{g}(u))=u+\epsilon E(u)$.
    \item $\varphi(\Tilde{g}(u))=\Tilde{g}(\varphi(u)$.
\end{itemize}

Now we wish to construct $b=(b_0,b_1,\ldots)$ such that $\Tilde{g}(\lambda)=\Tilde{f}(\lambda)\cdot b$. As $R$ is $p$-torsion free, the ghost map is injective, hence this identity is equivalent to that 
\begin{align}\label{equa.ghost identity}
    \forall n\geq 0, w_n(\Tilde{g}(\lambda))=w_n(b)\cdot w_n(\Tilde{f}(\lambda)).
\end{align}
Here $w_n$ denotes the $n$-th ghost map.

We first make the three terms showing up in \cref{equa.ghost identity} more explicit. Notice that 
\begin{equation*}
    w_n(\Tilde{f}(\lambda))=w_0(\varphi^n(\Tilde{f}(\lambda)))=w_0(\Tilde{f}(\varphi^n(\lambda)))=\varphi^n(E(u))=E(u)^{p^n}+ph_n(u),
\end{equation*}
where $h_n(u)\in \mathfrak{S}$ is defined to be $\varphi^n(E(u))-E(u)^{p^n}$. Here the second equality follows as $\Tilde{f}$ commutes with $\varphi$. Similarly, one could calculate that
\begin{equation*}
\begin{split}
    w_n(\Tilde{g}((\lambda))=g(\varphi^n(E(u)))&=E(u+\epsilon E(u))^{p^n}+ph_n(u+\epsilon E(u))
    \\&=(E(u)(1+\epsilon E^{\prime}(u)))^{p^n}+p(h_n(u)+\epsilon h_n^{\prime}(u)E(u))
    \\&=E(u)^{p^n}(1+p^n\epsilon E^{\prime}(u))+ph_n(u)+p\epsilon h_n^{\prime}(u)E(u).
\end{split}
\end{equation*}
Take $n=0$ in \cref{equa.ghost identity}, we want $b_0$ such that $E(u)(1+\epsilon E^{\prime}(u))=b_0\cdot E(u)$, hence $b_0=1+\epsilon E^{\prime}(u)$ as $R$ is $E(u)$-torsion free.

Suppose $n\geq 1$ and we have determined $b_0,\cdots,b_{n-1}$ such that \cref{equa.ghost identity} holds for non-negative integers no larger than $n-1$. Moreover, we assume that $b_i~ (1\leq i\leq n-1)$ is divisible by $\epsilon$. Then we claim that there exists a unique $b_n\in R$ such that \cref{equa.ghost identity} holds for $n$. For this, first notice that 
\begin{equation*}
     w_{n}(b)=\sum_{i=0}^{n}p^ib_{i}^{p^{n-i}}=b_{0}^{p^n}+p^nb_n=(1+\epsilon p^n E^{\prime}(u))+p^nb_n.
\end{equation*}
Here the second identity holds as $b_i^p=0$ by our assumption that $b_i$ is divisible by $\epsilon$ for $1\leq i \leq n-1$. 

Combining all of the previous calculations together, we see that $w_n(\Tilde{g}(\lambda))=w_n(b)\cdot w_n(\Tilde{f}(\lambda))$ if and only if 
\begin{equation*}
    \begin{split}
    & E(u)^{p^n}(1+p^n\epsilon E^{\prime}(u))+ph_n(u)+p\epsilon h_n^{\prime}(u)E(u)=(E(u)^{p^n}+ph_n(u))(1+\epsilon p^n E^{\prime}(u)+p^nb_n)
        \\ & \Longleftrightarrow  ~p^{n+1}b_n(h_n(u)+\frac{E(u)^{p^n}}{p})=p\epsilon h_n^{\prime}(u)E(u)-p^{n+1}\epsilon h_n(u)E^{\prime}(u)
    \end{split}
\end{equation*}
As $h_n(u)$ is a unit in $\mathfrak{S}$ by the proof of \cref{lem.units sn}, we see that $h_n(u)+\frac{E(u)^{p^n}}{p}\in R^{\times}$ since $\frac{E(u)^{p^n}}{p}\in R^{\times}$ is topologically nilpotent in $R$. Utilizing the $p$-torsion freeness of $R$, the above equation has a unique solution 
\begin{equation*}
    b_n=\epsilon (h_n(u)+\frac{E(u)^{p^n}}{p})^{-1}(t_n(u)E(u)-h_n(u)E^{\prime}(u)),
\end{equation*}
where $t_n \in \mathfrak{S}$ is defined as in \cref{lem.units sn}. In particular, $b_n$ is divisible by $\epsilon$. Then inductively we construct the unique $b$ satisfying the desired properties.
\end{proof}
\begin{remark}\label{rem.first strenthern}
    For $k\leq p$, if we take $R_k=\mathfrak{S}/E^k[\epsilon]/\epsilon^2$ instead, then we could run the above proof to show that there exists $b$ in $W(R_k)^{\times}$ satisfying similar properties (but such a $b$ might not be unique). For example, we could take $b=(b_0, b_1, \cdots)$ with $b_0=1+\epsilon E^{\prime}(u)$ and $b_i=\epsilon (h_n(u))^{-1}(t_n(u)E(u)-h_n(u)E^{\prime}(u))$.
\end{remark}

The following lemma is used in the above proof:
\begin{lemma}\label{lem.units sn}
    Keep notations as in the above lemma. In $T=\mathfrak{S}[[\frac{E}{p}]]$, for all $n\geq 1$, let $s_n =h_n(u)+\frac{E(u)^{p^n}}{p}$, then $s_n\in T^{\times}$ and there exists $t_n(u)\in T$ such that
    \begin{equation*}
        h_n^{\prime}(u)=p^n t_n(u).
    \end{equation*}
\end{lemma}
\begin{proof}
    In general, given an oriented prism $(A,d)$, by induction on $n$, one could easily show that there exists $v_n \in A^{\times}$ such that $\varphi^n(d)=d^{p^n}+pv_n$ (The base case $n=1$ just follows from the definition of a prism). Applying to our case, we see that $h_n(u)\in \mathfrak{S}^{\times}$, then $s_n=h_n(u)+\frac{E(u)^{p^n}}{p}$ is still a unit in $T$ as $\frac{E(u)^{p^n}}{p}$ is a topologically nilpotent in $T$.
     
    Next we show $t_n$ already exists in $\mathfrak{S}$. Suppose the Eisenstein polynomial $E(u)=\sum_{i=0}^{e}a_iu^i$ with $a_e=1$ and $p|a_i$ for $0\leq i\leq e-1$. Then $\varphi^n(E(u))=u^{p^ne}+\varphi^n(a_0)+k_n(u)$ with
    \begin{equation*}
        k_n(u)=\sum_{i=1}^{e-1}\varphi^n(a_i)u^{p^ni},
    \end{equation*}
     from which we can see that $p^{n+1}|k_n^{\prime}(u)$. Assume that $k_n^{\prime}(u)=p^{n+1}r_n(u)$

    On the other hand, $ (E(u)^{p^n})^{\prime}=p^nE(u)^{p^n-1}E^{\prime}(u)$ and that by the definition of Eisenstein polynomial, the coefficients of $u^i$ in $E(u)^{p^n-1}E^{\prime}(u)$ are all divisible by $p$ except possibly for the top degree $i_0=e(p^n-1)+e-1)=p^n e-1$. In summary, $(E(u)^{p^n})^{\prime}=p^n u^{p^n e-1}+p^{n+1}R_n(u)$ for some polynomial $R_n(u)\in W(k)[u]$. Consequently,
    \begin{equation*}
        (p h_n(u))^{\prime}=(\varphi^n(E(u))-E(u)^{p^n})^{\prime}=p^{n+1}(r_n(u)-R_n(u)),
\end{equation*}
     hence we could take $t_n(u)=r_n(u)-R_n(u)$ satisfying that $h_n^{\prime}(u)=p^n t_n(u)$ (This is the unique choice as $R$ is $p$-torsion free).
\end{proof}

Next we state a result which will be used together with \cref{lemT.construct b} to construct the desired Sen operator.
\begin{lemma}\label{lemt.construct homotopy}
    Keep notations as in \cref{lemT.construct b}, there exists a unique $c$ in $W(R)$ such that $$\Tilde{g}(u)-\Tilde{f}(u)=\Tilde{f}(\lambda)\cdot c.$$
\end{lemma}
\begin{proof}
    We wish to construct $c=(c_0,c_1,\ldots)$ such that $\Tilde{g}(u)-\Tilde{f}(u)=\Tilde{f}(\lambda)\cdot c$. As $R$ is $p$-torsion free, the ghost map is injective, hence this identity is equivalent to that 
\begin{align}\label{equa for c.ghost identity}
    \forall n\geq 0, w_n(\Tilde{g}(u))-w_n(\Tilde{f}(u))=w_n(c)\cdot w_n(\Tilde{f}(\lambda)),
\end{align}
where $w_n$ denotes the $n$-th ghost map.
Notice that $$w_n(\Tilde{f}(u))=w_0(\varphi^n(\Tilde{f}(u)))=w_0(\Tilde{f}(\varphi^n(u)))=u^{p^n}$$
and that 
\[w_n(\Tilde{g}(u))=w_0(\varphi^n(\Tilde{g}(u)))=w_0(\Tilde{g}(\varphi^n(u)))=(u+\epsilon E(u))^{p^n}=u^{p^n}+p^nu^{p^n-1}\epsilon E(u)\]
Take $n=0$ in \cref{equa for c.ghost identity}, we want $c_0$ such that $\epsilon E(u)=c_0\cdot E(u)$, hence $c_0=\epsilon$ as $R$ is $E(u)$-torsion free.

Suppose $n\geq 1$ and we have determined $c_0,\cdots,c_{n-1}$ such that \cref{equa for c.ghost identity} holds for non-negative integers no larger than $n-1$. Moreover, we assume that $c_i~ (0\leq i\leq n-1)$ is divisible by $\epsilon$. Then we claim that there exists a unique $c_n\in R$ such that \cref{equa for c.ghost identity} holds for $n$ and that $c_n$ is divisible by $\epsilon$ as well. For this, first notice that 
\begin{equation*}
     w_{n}(c)=\sum_{i=0}^{n}p^ic_{i}^{p^{n-i}}=p^nc_n.
\end{equation*}
Here the second identity holds as $c_i^p=0$ by our assumption that $c_i$ is divisible by $\epsilon$ for $0\leq i \leq n-1$.

Combining all of the previous calculations with the calculation of $w_n(\Tilde{f}(\lambda))$ in \cref{lemT.construct b}, we see that \cref{equa for c.ghost identity} holds for $n$ if and only if 
\begin{equation*}
    p^nu^{p^n-1}\epsilon E(u)=p^nc_n(E(u)^{p^n}+ph_n(u))=p^{n+1}c_n(h_n(u)+\frac{E(u)^{p^n}}{p})
\end{equation*}
As $R$ is $p$-torsion free and $h_n(u)+\frac{E(u)^{p^n}}{p}\in R^{\times}$, the above equation has a unique solution
\begin{equation*}
    c_n=\epsilon(h_n(u)+\frac{E(u)^{p^n}}{p})^{-1} u^{p^n-1} \frac{E(u)}{p}.
\end{equation*}
In particular, $c_n$ is divisible by $\epsilon$. Then inductively we construct the unique $c$ satisfying the desired properties.
\end{proof}
\begin{remark}\label{rem.second strenthern}
    For $k\leq 1+\frac{p-1}{e}$, if we take $R_k=\mathfrak{S}/E^k[\epsilon]/\epsilon^2$ instead, then we could run the above proof to show that there exists $c$ in $W(R_k)$ satisfying similar properties (but such a $c$ might not be unique). Actually, in this case, $u^{p-1}E(u) \equiv u^{p-1+e}\mod p$. But as $E(u)^k=0$ in $R_k$, we see that $u^{ek} \in pR_k$, hence $u^{p-1+e}\in pR_k$ as our choice of $k$ guarantees that $p-1+e\geq ek$. Consequently we could keep finding $c_n$ such that \cref{equa for c.ghost identity} holds.
\end{remark}

\begin{proposition}\label{propt.key automorphism of functors}
    The elements $b$ and $c$ constructed in \cref{lemT.construct b} and \cref{lemt.construct homotopy} together induce an isomorphism $\gamma_{b,c}$ between functors $\rho: \Spf(\mathfrak{S}[[\frac{E}{p}]]) \rightarrow \TX$
    and $\rho\circ \delta: \Spf(\mathfrak{S}[[\frac{E}{p}]]) \rightarrow \TX$ after base change to $\Spec(\mathbb{Z}[\epsilon]/(\epsilon^2))$, i.e. we have the following commutative diagram:
\[\xymatrixcolsep{5pc}\xymatrix{\Spf(\mathfrak{S}[[\frac{E}{p}]])\times \Spec(\mathbb{Z}[\epsilon]/(\epsilon^2))\ar[d]^{\rho}\ar[r]^{\delta}& \Spf(\mathfrak{S}[[\frac{E}{p}]])\times \Spec(\mathbb{Z}[\epsilon]/(\epsilon^2))\ar@{=>}[dl]^{\gamma_{b,c}} \ar[d]_{}^{\rho}
\\\TX\times \Spec(\mathbb{Z}[\epsilon]/(\epsilon^2))  \ar[r]&\TX\times \Spec(\mathbb{Z}[\epsilon]/(\epsilon^2))}\]
\end{proposition}
\begin{proof}
      Let $R$ be $\mathfrak{S}[[\frac{E}{p}]][\epsilon]/\epsilon^2$. Given a test $(p,\frac{E}{p})$-nilpotent $R$-algebra $T$ via the structure morphism $h: R\to T$, we denote the induced morphism $W(R)\to W(T)$ by $\Tilde{h}$. Then $\rho\circ h(T)$ corresponds to the point $$(\alpha:(E) \otimes_{\mathfrak{S},\Tilde{h}\circ \Tilde{f}} W(T)\to W(T),\eta: \Cone((E)\to \mathfrak{S})\xrightarrow{\Tilde{h}\circ \Tilde{f}} \Cone(\alpha))$$
      in $\TX(T)$, while $(\rho\circ \delta)\circ  h(T)$ corresponds to the point $$(\alpha^{\prime}:(E) \otimes_{\mathfrak{S},\Tilde{h}\circ \Tilde{g}} W(T)\to W(T), \eta^{\prime}: \Cone((E)\to \mathfrak{S}) \xrightarrow{\Tilde{h}\circ \Tilde{g}} \Cone(\alpha^{\prime})).$$ 
      
      We need to specify an isomorphism $\gamma_{b}: \alpha^{\prime}\xrightarrow{\simeq} \alpha$ as well as a homotopy $\gamma_{c}$ between $\gamma_{b}\circ \eta^{\prime}$ and $\eta$ which are both functorial in $T$.
      
      Utilizing \cref{lemT.construct b},  we construct the desired isomorphism $\gamma_{b}$ as follows: 
     \[\xymatrixcolsep{5pc}\xymatrix{(E) \otimes_{\mathfrak{S},\Tilde{h}\circ \Tilde{g}} W(T) \ar[d]^{(E)\otimes x\mapsto (E)\otimes \Tilde{h}(b)x}\ar[r]^{\iota}& W(T) \ar[d]_{}^{\Id}
\\(E) \otimes_{\mathfrak{S},\Tilde{h}\circ \Tilde{f}} W(T)\ar[r]^{\iota}&W(T)}\]
Here the left vertical map is $W(T)$-linear and the commutativity of the diagram follows from \cref{lemT.construct b}.

Then we draw a diagram illustrating $\gamma_{b}\circ \eta^{\prime}$ and $\eta$ (as maps of quasi-ideals):
\[\xymatrixcolsep{5pc}\xymatrix{(E) \ar@<-.5ex>[d]_{\gamma_{b}\circ \eta^{\prime}} \ar@<.5ex>[d]^{\eta}\ar[r]^{\iota}& \mathfrak{S} \ar[d]_{\gamma_{b}\circ \eta^{\prime}}
 \ar@<-.5ex>[d]^{\eta}
\\(E) \otimes_{\mathfrak{S},\Tilde{h}\circ \Tilde{f}} W(T)\ar[r]^{\iota}&W(T)}\]
Here the two left vertical maps are given by $\gamma_{b}\circ \eta^{\prime}: x\cdot (E)\mapsto (E)\otimes \Tilde{h}(b)\Tilde{h}(\Tilde{g}(x))$ and $\eta: x\cdot (E)\mapsto (E)\otimes \Tilde{h}\circ \Tilde{f}(x)$, the two right vertical maps are given by $\gamma_{b}\circ \eta^{\prime}$ and $\eta=\Tilde{h}\circ \Tilde{f}$.

To construct the desired homotopy, we need to specify a map $\gamma_c: \mathfrak{S} \to (E) \otimes_{\mathfrak{S},\Tilde{h}\circ \Tilde{f}} W(T)$ such that $\iota\circ \gamma_c=\gamma_{b}\circ \eta^{\prime}-\eta$ and that $\gamma_c \circ \iota =\gamma_{b}\circ \eta^{\prime}-\eta$. Without loss of generality, we assume $T=R$. In this case, $\iota: (E) \otimes_{\mathfrak{S}, \Tilde{f}} W(R) \to W(R)$ is injective as $W(R)$ is $E$-torsion free by the uniqueness in \cref{lemT.construct b} (otherwise if there exists a nontrivial $E$-torsion $q$, $b+q\neq b$ is another objects in $W(R)$ satisfying $(b+q)\cdot E=E$, a contradiction with the uniqueness of $b$). Consequently, it suffices to construct $\gamma_c$ and check that $\iota\circ \gamma_c=\gamma_{b}\circ \eta^{\prime}-\eta$.

Inspired by \cref{lemt.construct homotopy}, we just define $\gamma_c(u)$ to be $E\otimes_{\mathfrak{S}, \Tilde{f}} c$. For a general $s(u)\in \mathfrak{S}$, denote $k_s(u,v) \in W(k)[[u,v]]$ to be the unique power series such that $s(u)-s(v)=(u-v)\cdot k_s(u,v)$. Then 
\[(\gamma_{b}\circ \eta^{\prime}-\eta)(s(u))=\Tilde{g}(s(u))-\Tilde{f}(s(u))=s(\Tilde{g}(u))-s(\Tilde{f}(u))=(\Tilde{g}(u)-\Tilde{f}(u))\cdot k_s=\Tilde{f}(E)c\cdot k_s\]
for $k_s=k_s(\Tilde{g}(u), \Tilde{f}(u))\in W(R)$. Here the last quality follows from our construction of $c$ in \cref{lemt.construct homotopy}.

Hence if we define 
\[\gamma_c(s(u))=E\otimes_{\mathfrak{S}, \Tilde{f}} c k_s,\]

Then the desired identity $\iota\circ \gamma_c=\gamma_{b}\circ \eta^{\prime}-\eta$ follows.

Finally it is clear that $\gamma_b$ and $\gamma_c$ are all constructed via a base change from $W(R)$ to $W(T)$, hence they are all natural in $T$. We win.
\end{proof}

When restricted to the locus where $(\frac{t}{p})^n=0$ inside $[\Spf(\mathbb{Z}[[ \frac{t}{p}]])/ \mathbf{G}_m]$, we obtain the following truncated version of \cref{propt.key automorphism of functors}.
\begin{corollary}\label{cort.truncated}
    The $\gamma_{b,c}$ constructed in \cref{propt.key automorphism of functors} induces an isomorphism between functors $\rho: \Spf(\mathfrak{S}[[\frac{E}{p}]]/(\frac{E}{p})^n) \rightarrow \TXn$
    and $\rho\circ \delta: \Spf(\mathfrak{S}[[\frac{E}{p}]]) \rightarrow \TXn$ after base change to $\Spec(\mathbb{Z}[\epsilon]/(\epsilon^2))$, i.e. we have the following commutative diagram:
\[\xymatrixcolsep{3pc}\xymatrix{\Spf(\mathfrak{S}[[\frac{E}{p}]]/(\frac{E}{p})^n)\times \Spec(\mathbb{Z}[\epsilon]/(\epsilon^2))\ar[d]^{\rho}\ar[r]^{\delta}& \Spf(\mathfrak{S}[[\frac{E}{p}]]/(\frac{E}{p})^n)\times \Spec(\mathbb{Z}[\epsilon]/(\epsilon^2))\ar@{=>}[dl]^{\gamma_{b,c}} \ar[d]_{}^{\rho}
\\\TXn\times \Spec(\mathbb{Z}[\epsilon]/(\epsilon^2))  \ar[r]&\TXn\times \Spec(\mathbb{Z}[\epsilon]/(\epsilon^2))}\]
\end{corollary}
\begin{remark}[Compatibility with the construction in \cite{anschutz2022v} on the Hodge-Tate locus]\label{remt.compatibility with LB}
If we consider the restriction of the isomorphism constructed above to the Hodge-Tate locus (i.e. take $n=1$), then as $\delta=\Id$ on $\Spf (\mathcal{O}_K[\epsilon]/\epsilon^2)$, we see $\gamma_{b,c}$ descends to an automorphism 
    \[\Spf(\mathcal{O}_K)^{\HT}\times \Spec(\mathbb{Z}[\epsilon]/(\epsilon^2))\to \Spf(\mathcal{O}_K)^{\HT}\times \Spec(\mathbb{Z}[\epsilon]/(\epsilon^2))\]
    As $\Spf(\mathcal{O}_K)^{\HT}$ is the classifying stack of $G_{\pi}$ by \cite[Proposition 9.5]{bhatt2022prismatization}, where $G_{\pi}$ is calculated in \cite[Example 9.6]{bhatt2022prismatization}. More explicitly, 
    \[G_\pi=\{(t,a)\in \mathbb{G}_m^\sharp\ltimes \mathbb{G}_a^\sharp\ |\ t-1=E^{\prime}(\pi)\cdot a\}\]
Under this identification, $\gamma_{b,c}$ corresponds to an element in $G_\pi(\mathcal{O}_K[\epsilon]/\epsilon^2)$, we claim this element is precisely $(1+E^{\prime}(\pi)\epsilon,\epsilon)$. %\footnote{ Here all necessary higher divided powers of $\varepsilon, 1+e\varepsilon$ are defined to be $0$.}. 
To see this, unwinding the construction of $\gamma_{b,c}$ from $b$ and $c$, we just need to verify that the image of $b$ in $\mathbb{G}_m^\sharp(\mathcal{O}_K[\epsilon]/\epsilon^2)$ is precisely $1+e\epsilon$ and the image of $c$ in $\mathbb{G}_a^\sharp(\mathcal{O}_K[\epsilon]/\epsilon^2)$ is precisely $\epsilon$. As $\mathcal{O}_K[\epsilon]/\epsilon^2$ is $p$-torsion free, it suffices to check that $b_0=1+E^{\prime}(\pi)\epsilon$ and $c_0=\epsilon$ in $\mathcal{O}_K[\epsilon]/\epsilon^2$, which are both clear from our construction of $b$ and $c$ in \cref{lemT.construct b} and \cref{lemt.construct homotopy}. 
\end{remark}

Let $n\in \mathbb{N}$. Now we are ready to construct a Sen operator on $\rho^*\mathscr{E}$ for $\mathscr{E}\in \Qcoh(\TX)$ (resp. $ \Qcoh(\TXn)$). Based on \cref{propt.key automorphism of functors} (resp. \cref{cort.truncated}), we have an isomorphism $\gamma_{b,c}: \rho\circ \delta \stackrel{\simeq}{\longrightarrow} \rho$. Consequently, for $\mathscr{E}\in \Qcoh(\TX)$ (resp. $ \Qcoh(\TXn)$), we have an isomorphism
\[\gamma_{b,c}: \delta^{*}\rho^{*}(\mathscr{E}\otimes \mathbb{Z}[\epsilon]/(\epsilon^2)) \stackrel{\simeq}{\longrightarrow} \rho^{*}(\mathscr{E}\otimes \mathbb{Z}[\epsilon]/(\epsilon^2)).\]
Unwinding the definitions, this could be identified with a $\delta$-linear morphism 
\begin{align}\label{equat.truncated induced morphism}
    \gamma_{b,c}: \rho^{*}(\mathscr{E})\rightarrow \rho^{*}(\mathscr{E})\otimes \mathbb{Z}[\epsilon]/(\epsilon^2).
\end{align} 
Moreover, our definition of the element $b$ and $c$ in \cref{lemT.construct b} and \cref{lemt.construct homotopy} implies that $\gamma_{b,c}$ in \cref{equat.truncated induced morphism} reduces to the identity modulo $\epsilon$, hence could be written as $\Id+\epsilon\Theta_{\mathscr{E}}$ for some operator $\Theta_{\mathscr{E}}: \rho^{*}\mathscr{E} \rightarrow \rho^{*}\mathscr{E}$. 

\begin{proposition}\label{propt. sen calculate cohomology}
    Let $n\in \mathbb{N}$, then for any 
  $\mathscr{E}\in \Qcoh(\TX)$ (resp. $ \Qcoh(\TXn)$), the natural morphism $\mathrm{R} \Gamma(\TX, \mathcal{E})\to \rho^{*}\mathcal{E}$ (resp. $\mathrm{R} \Gamma(\TXn, \mathcal{E})\to \rho^{*}\mathcal{E}$) induces a canonical identification $$
    \mathrm{R} \Gamma(\TX, \mathcal{E})=\fib(\rho^*\mathcal{E}\stackrel{\Theta_{\mathscr{E}}}{\longrightarrow} \rho^*\mathcal{E})\quad \text{resp.}\quad  \mathrm{R} \Gamma(\TXn, \mathcal{E})=\fib(\rho^*\mathcal{E}\stackrel{\Theta_{\mathscr{E}}}{\longrightarrow} \rho^*\mathcal{E}).$$
\end{proposition}
\begin{proof}
Arguing as in \cref{sen calculate cohomology}, we see that the natural morphism $\mathrm{R} \Gamma(\TX, \mathcal{E})\to \rho^{*}\mathcal{E}$ (resp. $\mathrm{R} \Gamma(\TXn, \mathcal{E})\to \rho^{*}\mathcal{E}$) factors through the fiber of $\Theta_{\mathscr{E}}$.

For $n\in \mathbb{N}$, to see this induces an identification of $\mathrm{R} \Gamma(\TXn, \mathcal{E})\to \rho^{*}\mathcal{E}$ with $\fib(\rho^*\mathcal{E}\stackrel{\Theta_{\mathscr{E}}}{\longrightarrow} \rho^*\mathcal{E})$, using standard  d\'evissage (the trick we used in the proof of \cref{commute with colimits}), by induction on $n$, it reduces to $n=1$, which follows from \cite[Proposition 2.7]{anschutz2022v} thanks to \cref{remt.compatibility with LB}.

Finally for $\mathscr{E}\in \Qcoh(\TX)$, as taking global sections commutes with limits, by writing $\mathscr{E}$ as the inverse limit of $\mathscr{E}_n$ for $\mathscr{E}_n$ the restriction of $\mathscr{E}$ to $\TXn$, we see that 
\[\mathrm{R} \Gamma(\TX, \mathcal{E})=\varprojlim \mathrm{R} \Gamma(\TXn, \mathcal{E}_n)=\varprojlim \fib(\rho^*\mathcal{E}_n\stackrel{\Theta_{\mathscr{E}}}{\longrightarrow} \rho^*\mathcal{E}_n)=\fib(\rho^*\mathcal{E}\stackrel{\Theta_{\mathscr{E}}}{\longrightarrow} \rho^*\mathcal{E}).\]
Here the second equality follows from the above paragraph and the last equality holds as finite limits commute with limits.
\end{proof}

As a byproduct of \cref{propt. sen calculate cohomology}, we conclude that
\begin{corollary}
    The global sections functor $$\mathrm{R} \Gamma(\TX, \bullet): \mathcal{D}(\TX) \rightarrow \widehat{\mathcal{D}}(\mathbb{Z}_p) \quad \text{resp.}\quad \mathrm{R} \Gamma(\TXn, \bullet): \mathcal{D}(\TXn) \rightarrow \widehat{\mathcal{D}}(\mathbb{Z}_p)$$
    commutes with colimits.
\end{corollary}
\begin{remark}
    In contrary, $\mathrm{R} \Gamma(\WCart, \bullet)$ doesn't commute with colimits by \cite{bhatt2022absolute}.
\end{remark}

\begin{definition}\label{def.I/p sheaf}
    We define $(\frac{\mathcal{I}}{p})^k$ to be the invertible sheaf on $\TX$ by pulling back the invertible sheaf generated by $(\frac{t}{p})^k$ on $[\Spf(\mathbb{Z}[[ \frac{t}{p}]])/ \mathbf{G}_m]$.
\end{definition}

\begin{example}[Sen operator on the ideal sheaf $(\frac{\mathcal{I}}{p})^k$]\label{examplet.action on generators}
    Let $\mathscr{E}$ be the structure sheaf $\mathcal{O}_{\TX}$. Then under the trivialization  $\rho^{*}\mathscr{E}=\mathfrak{S}[[\frac{E}{p}]]$, $\gamma_{b,c}(u)=\delta(u)=u+\epsilon E(u)$, hence $\Theta_{\mathscr{E}}$ sends $f(u)\in \mathfrak{S}[[\frac{E}{p}]]$ to $f^{\prime }(u)E(u)$. In general, for $\mathscr{E}=(\frac{\mathcal{I}}{p})^k$, one can verify that under the trivialization  $\rho^{*}\mathscr{E}=\mathfrak{S}[[\frac{E}{p}]]\cdot (\frac{E}{p})^k$, $\Theta_{\mathscr{E}}$ sends $f(u)\in \mathfrak{S}[[\frac{E}{p}]]$ to $kf(u)E^{\prime}(u)+f^{\prime }(u)E(u)$.
\end{example}

\begin{proposition}\label{propt.generation}
    Let $n\in \mathbb{N}$. The $\infty$-category $\mathcal{D}(\TX)$ (resp. $\mathcal{D}(\TXn)$ is generated under shifts and colimits by the invertible sheaves $(\frac{\mathcal{I}}{p})^k$ for $k\in \mathbb{Z}$.
\end{proposition}
\begin{proof}
    Arguing as in \cite[Corollary 3.5.16]{bhatt2022absolute}, this could be reduced to $n=1$, where the results follow from \cite[Proposition 2.9]{anschutz2022v} as on the Hodge-Tate locus $\frac{\mathcal{I}}{p}\simeq \mathcal{I}$.
\end{proof}

\begin{theorem}\label{thmt.main classification}
Let $n\in \mathbb{N}$. The functor 
    \begin{align*}
        &\beta^{+}: \mathcal{D}(\TX) \rightarrow \mathcal{D}(\mathrm{MIC}(\mathfrak{S}[[\frac{E}{p}]]), \qquad \mathscr{E}\mapsto (\rho^{*}(\mathscr{E}),\Theta_{\mathscr{E}})
        \\ resp. ~&\beta_n^{+}: \mathcal{D}(\TXn) \rightarrow \mathcal{D}(\mathrm{MIC}(\mathfrak{S}[[\frac{E}{p}]]/(\frac{E}{p})^n), \qquad \mathscr{E}\mapsto (\rho^{*}(\mathscr{E}),\Theta_{\mathscr{E}})
    \end{align*}
    is fully faithful \footnote{Here $\mathcal{D}(\mathrm{MIC}(\mathfrak{S}[[\frac{E}{p}]])$ is defined similarly as \cref{def.MIC}}. Moreover, its essential image consists of those objects $M\in \mathcal{D}(\mathrm{MIC}(\mathfrak{S}[[\frac{E}{p}]]))$ (resp. $M\in \mathcal{D}(\mathrm{MIC}(\mathfrak{S}[[\frac{E}{p}]]/(\frac{E}{p})^n) $) satisfying the following pair of conditions:
    \begin{itemize}
        \item $M$ is $\mathbb{Z}_p$-complete.
        \item The action of $\Theta^p-(E^{\prime}(u))^{p-1}\Theta$ on the cohomology $\mathrm{H}^*(M\otimes^{\mathbb{L}}k)$ \footnote{Here the derived tensor product means the derived base change along $\mathfrak{S}[[\frac{E}{p}]])\to \mathfrak{S}[[\frac{E}{p}]])/(\frac{E}{p},u)=k$} is locally nilpotent. 
    \end{itemize}
\end{theorem}
\begin{proof}%[Proof of \cref{thmt.main classification}]
    Given \cref{propt. sen calculate cohomology} and \cref{propt.generation}, the functor is well-defined and fully faithful using the same argument in the proof of \cref{thm.main classification}. 

    To check that the action of $\Theta^p-(E^{\prime}(u))^{p-1}\Theta$ on the cohomology $\mathrm{H}^*(\rho^{*}(\mathscr{E})\otimes^{\mathbb{L}}k)$ is locally nilpotent for $\mathscr{E}\in \mathcal{D}(\TX)$ (resp. $\mathscr{E}\in \mathcal{D}(\TXn)$), again thanks to \cref{propt.generation}, we might assume $\mathscr{E}=(\frac{\mathcal{I}}{p})^k$ for some $k\in \mathbb{Z}$. Then by \cref{examplet.action on generators}, after base change to $k$ the action of $\Theta^p-(E^{\prime}(u))^{p-1}\Theta$ is given by
    \[(kE^{\prime}(\pi))^{p}-kE^{\prime}(\pi)(E^{\prime}(\pi))^{p-1}=(k^p-k)=(k^p-k)(E^{\prime}(\pi))^{p},\] 
    which already vanishes as for any integer $j$, $j^p\equiv j \mod p$. 

    Let $\mathcal{C}\subseteq \mathcal{D}(\mathrm{MIC}(\mathfrak{S}[[\frac{E}{p}]]))$ (resp. $M\in \mathcal{D}(\mathrm{MIC}(\mathfrak{S}[[\frac{E}{p}]]/(\frac{E}{p})^n) $) be the full subcategory spanned by objects satisfying two conditions listed in \cref{thmt.main classification}. As the source $\mathcal{D}(\WCart_n)$ is generated under shifts and colimits by the invertible sheaves $\mathcal{I}^n$ for $n\in \mathbb{Z}$ by \cref{propt.generation}, to complete the proof it suffices to show that $\mathcal{C}$ is also generated under shifts and colimits by $\{\rho^*(\frac{\mathcal{I}}{p})^k\}$ ($k\in \mathbb{Z}$). In other words, we need to show that for every nonzero object $M\in \mathcal{C}$, $M$ admits a nonzero morphism from $\rho^*(\frac{\mathcal{I}}{p})^k[m]$ for some $m,k \in \mathbb{Z}$. Replacing $M$ by $M\otimes k$ (the derived Nakayama guarantees that $M\otimes k$ detects whether $M$ is zero or not as $M$ is assumed to be $p$-complete), we may assume that there exists some cohomology group $\mathrm{H}^{-m}(M)$ containing a nonzero element killed by $\Theta^p-(E^{\prime}(\pi))^{p-1}\Theta=\prod_{0\leq i<p} (\Theta-E^{\prime}(\pi)i)$ (this could be done by iterating the action of $\Theta^p-(E^{\prime}(\pi))^{p-1}\Theta$ and then use the nilpotence assumption). Furthermore, we could assume this element is actually killed by $\Theta-k$ for a single integer $k$. It then follows that there exists a non-zero morphism from $\rho^*(\frac{\mathcal{I}}{p})^k[m]$ to $M$ which is nonzero in degree $m$.
\end{proof}
%\begin{proposition}
%There is an equivalence of categories $$\mathcal{D}(\TWCart_{[n]}) \rightarrow \lim _{(A, I)} \widehat{\mathcal{D}}(A/I^{p^n}[\frac{I^p}{p}])$$ Here $A/I^{p^n}[\frac{I^p}{p}]$ denote the sub-$A/I^{p^n}$ algebra generated by $1$ and $\frac{I^p}{p}$ inside $A/I^{p^n}[\frac{1}{p}]$.
%\end{proposition}

\begin{corollary}\label{corot.perfect version of the remain theorem}
Let $n\in \mathbb{N}$. The functor $\beta^+_n$ from \cref{thmt.main classification} restricts to a fully faithful functor
  \[
    \beta^+_n\colon \Perf(\TXn) \to \Perf(\mathrm{MIC}(\mathfrak{S}[[\frac{E}{p}]]/(\frac{E}{p})^n)%\footnote{Here $\Perf(\mathrm{MIC}(\mathfrak{S}[[\frac{E}{p}]]/(\frac{E}{p})^n)$ is defined to be the $\infty$-category of perfect complexes $M$ over $\mathfrak{S}[[\frac{E}{p}]]/(\frac{E}{p})^n$ admitting a $W(k)$-linear operator $\Theta$ from $M$ to itself satisfying the Leibniz rule. In particular, when $n=1$, as $\Theta$ vanishes on the $\mathfrak{S}/E$, this coincides with $\Perf(\mathcal{O}_K[\Theta_{\pi}])$  defined in \cite{anschutz2022v} by the proof of \cite[Lemma 2.13]{anschutz2022v}.}
  \]
  whose essential image consists of $p$-adically complete perfect complexes $M$ over \\$\mathfrak{S}[[\frac{E}{p}]]/(\frac{E}{p})^n$ admitting a $W(k)$-linear operator $\Theta$ from $M$ to itself satisfying the Leibniz rule such that $\Theta^p-(E^{\prime}(u))^{p-1}\Theta$ is nilpotent on $H^\ast(k\otimes^L M)$.    
\end{corollary}
\begin{proof}
    This follows from \cref{thmt.main classification} directly.
\end{proof}
\begin{remark}\label{rem.third stren}
    When $n\leq 1+\frac{p-1}{e}$, \cref{thmt.main classification} still holds if we replace the left-hand side of $\beta_n^+$ with $\mathcal{D}(X^{\Prism}_n)$ and replace the right-hand side by $\mathcal{D}(\mathrm{MIC}(\mathfrak{S}/E^n)$. Actually, by \cref{rem.first strenthern,rem.second strenthern} we could construct $\gamma_{b,c}$, hence the Sen operator and run the above proof similarly.
\end{remark}

\subsection{An analytic variant}
Let $n\in \mathbb{N}$. Later when passing from perfect complexes on $\TXn$ to perfect complexes on $\Spa(K)_v$  with coefficients in $\bdr_n$, the functor naturally factors through the isogeny category of the source, hence it is better to give a more explicit characterization of $\Perf(\TXn)[1/p]$. 
 Actually, by formally inverting $p$ on the source, the functor in \cref{thmt.main classification} induces an "analytic" functor
\[
\beta_n:\Perf(\TXn)[1/p]\to \Perf(\mathrm{MIC}(\mathfrak{S}/E^n[1/p]))\footnote{Here $\Perf(\mathrm{MIC}(\mathfrak{S}/E^n[1/p]))$ is defined similarly as \cref{def.MIC}. Moreover, the topology is given by the $(E,p)$-adic topology, or just the $p$-adic topology as $E$ is nilpotent.}.%to be the $\infty$-category of perfect complexes $M$ over $\mathfrak{S}/E^n[1/p]$ admitting a $W(k)[1/p]$-linear operator $\Theta$ from $M$ to itself satisfying the Leibniz rule. 
 \]
\begin{corollary}\label{cor.analytic-variant}
 The functor $\beta_n$ is fully faithful. Its essential image consists of complexes $M\in \Perf(\mathrm{MIC}(\mathfrak{S}/E^n[1/p]))$ such that $H^\ast(M)$ is finite dimensional over $\mathfrak{S}/E^n[1/p]$ and the action of the operator $\Theta^p-(E^{\prime}(u))^{p-1}\Theta$ on  $H^\ast(M)$ is topologically nilpotent.
\end{corollary}
\begin{proof}
    Let $\mathscr{E}\in \Perf(\TXn)[1/p]$, then $\beta(\mathscr{E})=\beta^{+}(\mathscr{E})[1/p]=(\rho^{\ast}\mathscr{E}[1/p], \Theta_{\mathscr{E}})$. By \cref{corot.perfect version of the remain theorem}, $\rho^{\ast}\mathscr{E} \in \Perf(\mathrm{MIC}(\mathfrak{S}[[\frac{E}{p}]]/(\frac{E}{p})^n)$, in particular, it is a perfect \\$\mathfrak{S}[[\frac{E}{p}]]/(\frac{E}{p})^n$-complex, hence $\beta$ is well defined. The full faithfulness follows from \cref{corot.perfect version of the remain theorem} as perfect complexes are compact.
    
    To see that $\Theta^p-(E^{\prime}(u))^{p-1}\Theta$ acts on the cohomology of $M=\rho^{\ast}\mathscr{E}[1/p]$ topologically nilpotently, we do induction on $n$. The base case $n=1$ is due to \cite[Corollary 2.15]{anschutz2022v}. For general $n\geq 2$, the canonical fiber sequence 
    $$\mathscr{E}\otimes i_{[n-1],*}(\mathcal{O}_{\TXM}\{1\})\to \mathscr{E} \to  \mathscr{E}\otimes i_{[1],*}(\mathcal{O}_{\TXV})$$
    and the projection formula implies the existence of a fiber sequence 
    $$M_1 \to \rho^{\ast}\mathscr{E}[1/p] \to M_2$$
    {\small compatible with $\Theta$ for some $M_1\in \beta_{n-1}(\Perf(\TXM)[1/p])$ and $M_2\in \beta_{n-1}(\Perf(\TXV)[1/p])$.}
    
    Utilizing the induced long exact sequence and by induction, we see that the action of the operator $\Theta^p-(E^{\prime}(u))^{p-1}\Theta$ on $H^\ast(\rho^{\ast}\mathscr{E}[1/p])$ is topologically nilpotent.

    Next we verify the description of the essential image. By induction on the amplitude and via considering cones, we reduce to the case that $M$ is concentrated on degree $0$. Now $M$ is just a finite projective $\mathfrak{S}/E^n[1/p]$-module equipped with an operator $\Theta: M\to M$ satisfying the required nilpotence condition and we wish to construct a $\mathfrak{S}[[\frac{E}{p}]]/(\frac{E}{p})^n$-lattice $M_0$ inside $M$ stable under $\Theta$ first.

    For this purpose, we do induction $n$. The base case $n=1$ is \cite[Corollary 2.15]{anschutz2022v}. For $n\geq 2$, notice that there is a short exact sequence
    \[0\to (E/p)\otimes M/(E/p)^{n-1} \xrightarrow[]{f_1} M \xrightarrow[]{f_2} M/(E/p)M\to 0.\]
    After trivialization, the source could be identified with $(M/(E/p)^{n-1},\Theta_M+E\cdot E^{\prime}\Id)$, whose Sen operator still satisfies the desired nilpotence. By induction, we could find a 
    (full rank) $\mathfrak{S}[[\frac{E}{p}]]/(\frac{E}{p})^{n-1}$-lattice $M_1$ inside $(E/p)\otimes M/(E/p)^{n-1}$ and a (full rank) $\mathfrak{S}[[\frac{E}{p}]]/(\frac{E}{p})^{}$-lattice $M_2$ inside $M/(E/p)M$ which are both stable under their Sen operators. Fix a $\mathfrak{S}/E$-basis $\{\bar{e}_1,\cdots,\bar{e}_r\}$ of $M_1$ and pick $e_i\in M$ as the lift of $\bar{e}_i$. Let $M^{\prime}$ be the $\mathfrak{S}[[\frac{E}{p}]]/(\frac{E}{p})^n$-lattice generated by $\{e_1,\cdots,e_r\}$ inside $M$. As $M_2$ is stable under $\Theta$, $\bar{d}_i:= \theta(\bar{e}_i)$ is inside $M_2$. Notice that $M^{\prime} \xrightarrow{f_2} M $ is surjective by our construction, hence we could find lift $d_i$ of $\bar{d}_i$ inside $M^{\prime}$. Let $c_i:= \theta (e_i)-d_i \in M$, then we have that $f_2(c_i)=0$, which implies that $c_i$ lives in the image of $f_1$. Enlarging $M_1$ when necessary, without loss of generality we could assume $f_1(M_1)$ contains all of the $c_i$ for $1\leq i\leq r$. 
    
    Take $M_0=f_1(M_1)+M^{\prime}$ and we claim that $M_0$ is a $\Theta$-stable $\mathfrak{S}[[\frac{E}{p}]]/(\frac{E}{p})^n$-lattice inside $M$  such that $M_0[1/p]=M$. Actually, as $f_1(M_1)$ is $\Theta$-stable, it suffices to notice that for $1\leq i\leq n$, $\Theta(e_i)=c_i+d_i\in M_0$ as both $c_i$ and $d_i$ are in $M_0$ by construction. 

    On the other hand, for any $m\in M$, to see that $m\in M_0[1/p]$, it suffices to show that there exists $j>0$ such that $p^jm\in M_0$. First we can pick $l>0$ such that $p^lf_2(m)\in M_2$ as $M_2[1/p]=f_2(M)$ by our construction of $M_2$. As $M_0\to M_2$ is subjective, we can pick $v\in M_0$ such that $f_2(v)=p^lf_2(m)$. Consequently, $f_2(p^lm-v)=p^lf_2(m)-f_2(v)=0$, which implies that $p^lm-v$ lives in the image of $f_1$. Suppose $p^lm-v=f_1(x)$ and $t>0$ satisfies that $p^tx\in M_1$ (such a $t$ exists by our choice of $M_1$). This implies that $p^{t+l}m=p^t(p^lm-v)+p^tv=p^tf_1(x)+p^tv=f_1(p^tx)+p^tv \in M_0$. We win by taking $j$ to be $t+l$.

    Finally to show that $M_0$ lies in the essential image of $\beta_n^{+}$ in \cref{corot.perfect version of the remain theorem}, we need to check that $\Theta^p-(E^{\prime}(u))^{p-1}\Theta$ is nilpotent on $k\otimes M)$. This property holds for any $\Theta$-stable lattice $N_0$ inside $M$ as topological nilpotence just means usual nilpotence after modulo $(E/p,p)$.
\end{proof}
\begin{remark}
    One should compare this with \cite[Remark 2.28]{liu2023rham}.
\end{remark}

%\subsection{Pro-\'etale realization and comparison with Galois cohomology}

\section{Quasi-coherent complexes on \texorpdfstring{$\TXn$}\quad ~ and \texorpdfstring{$\bdr/\xi^m$-}\quad semilinear Galois representations}
In this section we relate perfect complexes on $\TXn$ (recall that $X=\Spf(\mathcal{O}_K)$) with perfect complexes in $\bdr/\xi^m$-modules on the v-site of $\Spa(K)$ following \cite[Section 4]{anschutz2022v}. 

Let 
\[
  \Spa(K):=\Spa(K,\mathcal{O}_K)
\]
be the diamond associated to $K$, cf.\ \cite[Definition 15.5]{scholze2017etale}, and its $v$-site
\[
  \Spa(K)_v,
\]
cf.\ \cite[Definition 14.1]{scholze2017etale}.
If $\Spa(R,R^+)\in \Spa(K)_v$ is an affinoid perfectoid space over $K$\footnote{Here and in the following we identify the v-site of $K$, which consists of perfectoid spaces $S$ in characteristic $p$ and an untilt $S^\sharp$ over $\Spa(K,\mathcal{O}_K)$, with the site of perfectoid spaces over $\Spa(K,\mathcal{O}_K)$.}, then
\begin{align*}
    &\bdr: \Spa(R,R^+) \mapsto (W((R^{+})^\flat))[\frac{1}{p}])^{\wedge}_{I}
    \\resp. ~&\bdr_{,m}: \Spa(R,R^+) \mapsto W((R^{+})^\flat))/I^m[\frac{1}{p}]
\end{align*}
defines the de Rham period sheaf (resp. $m$-trucated de Rham period sheaf) on $\Spa(K)_v$.

Arguing as the last sentence in \cref{rem.w(k) structure}, we see that the natural morphism $\Spf(R^+)^{\Prism}\rightarrow \Spf(\Ainf(R^+))$ induces a natural morphism 
\[\TRn \to \Spf(\Ainf(R^+)[[\frac{I}{p}]]/(\frac{I}{p})^n).\]
Moreover, it forms a fiber square
\[ \xymatrix{ \TRn \ar[r] \ar[d] & \Spf(R^+)^{\Prism} \ar[d] \\
\Spf(\Ainf(R^+)[[\frac{I}{p}]]/(\frac{I}{p})^n) \ar[r] & \Spf(\Ainf(R^+)) }\]
As the right vertical morphism is an isomorphism by \cite[Example 3.12]{bhatt2022prismatization}, so is the left vertical morphism.

Consequently, the structure morphism $f\colon \Spf(R^+)\to X=\Spf(\mathcal{O}_K)$ induces a natural map $$\Tilde{f}_n: \Spf(\Ainf(R^+)[[\frac{I}{p}]]/(\frac{I}{p})^n)\rightarrow \TXn,$$
from which we get a symmetric monoidal, exact functor
\[
  \alpha_n^{+*}\colon \Perf(\TXn)\to \Perf(\Spa(K)_v, \Ain[[\frac{I}{p}]]/(\frac{I}{p})^n)\to \Perf(\Spa(K)_v, \bdr_{,n}).
\]
Indeed, by definition
\[
  \Perf(\TXn):=\varprojlim\limits_{\Spec(S)\to \TXn} \Perf(S),
\]
where the limit is taken over the category of all (discrete) rings $S$ with a morphism $\Spec(S)\to \TXn$. 
Using the maps $\widetilde{f} \colon \Spf(R^+)=\varinjlim\limits_{n} \Spec(R^+/p^n)\to \mathrm{Spf}(\mathcal{O}_K)^{\rm HT}$, the construction of $\alpha_n^{+*}$ can now be stated as
\[
  (\mathcal{E}_S)_{\Spec(S)\to \mathrm{Spf}(\mathcal{O}_K)^{\rm HT}}\mapsto ((R\varprojlim\limits_{n} \mathcal{E}_{\Spec(R^+/p^n)\to \mathrm{Spf}(\mathcal{O}_K)^{\rm HT}})[1/p])_{\Spa(R,R^+)\to \Spa(K,\mathcal{O}_K)}.
\]
As $p$ is invertible on the target of  $\alpha_n^{+*}$, it induces a functor
\[
  \alpha_n^{*} \colon \Perf(\TXn)[1/p]\to \Perf(\Spa(K)_v, \bdr_{,n}).
\]

\begin{theorem}\label{thm.full faithfulness of pullback}
The functor $\alpha_n^{*} \colon \Perf(\TXn)[1/p]\to \Perf(\Spa(K)_v, \bdr_{,n})$ is fully faithful.
\end{theorem}
\begin{proof}
    We prove by induction on $n$. For $n=1$, this is \cite[Theorem 4.2]{anschutz2022v} thanks to \cref{remt.compatibility with LB}. For $n\geq 2$, let $\mathscr{F}, \mathscr{G}\in \Perf(\TXn)$, then tensoring $\mathscr{G}$ with the fiber sequence 
    \begin{equation*}\label{equat.structure exact}
        i_{[n-1],*}(\mathcal{O}_{\TXM}\{1\}) \rightarrow \mathcal{O}_{\TXn} \rightarrow i_{[1],*}(\mathcal{O}_{\TXV})
    \end{equation*}
    induces a fiber sequence 
    $$\mathscr{G}\otimes i_{[n-1],*}(\mathcal{O}_{\TXM}\{1\})\to \mathscr{G} \to  \mathscr{G}\otimes i_{[1],*}\mathcal{O}_{\TXV}.$$
    Applying the projection formula, we have that
    \begin{align*}
        &\Hom_{\Perf(\TXn)}(\mathcal{F}, \mathscr{G}\otimes i_{[n-1],*}(\mathcal{O}_{\TXM}\{1\}))[1/p]
        \\=&\Hom_{\Perf(\TXM)}(i^{[n-1],*}\mathcal{F}, i^{[n-1],*} \mathscr{G}\otimes \mathcal{O}_{\TXM}\{1\})[1/p]
        \\=&\Hom_{\Perf(\Spa(K)_v, \bdr_{,n-1})}(\alpha_{n-1}^{*}i^{[n-1],*}\mathcal{F},\alpha_{n-1}^{*}(i^{[n-1],*} \mathscr{G}\{1\}))
        \\=&\Hom_{\Perf(\Spa(K)_v, \bdr_{,n})}(\alpha_{n}^{*}\mathcal{F},\alpha_{n}^{*}(\mathscr{G})\otimes i_{[n-1],*}(\bdr_{n-1}\{1\})).
    \end{align*}
    Similarly, 
    \begin{align*}
        \Hom_{\Perf(\TXn)}(\mathcal{F}, \mathscr{G}\otimes i_{[1],*}\mathcal{O}_{\TXV})[1/p]
        =\Hom_{\Perf(\Spa(K)_v, \bdr_{,n})}(\alpha_{n}^{*}\mathcal{F},\alpha_{n}^{*}\mathscr{G}\otimes i_{[1],*}(\bdr_{,1}))
    \end{align*}
    Then the desired result follows by induction using the fiber sequence $$\alpha_{n}^{*}(\mathscr{G})\otimes i_{[n-1],*}(\bdr_{n-1}\{1\}) \to \alpha_{n}^{*}(\mathscr{G})\to \alpha_{n}^{*}\mathscr{G}\otimes i_{[1],*}(\bdr_{,1})$$
\end{proof}
\begin{remark}\label{remt.factor through beta}
    \begin{itemize}
        \item The above construction and theorem works if we replace \\$\Perf(\TXn)[1/p]$ with $\Perf(X^{\Prism}_n)[1/p]$.
        \item Clearly $\alpha_n^{*}$ factors through $\beta_n$ in \cref{cor.analytic-variant}.
    \end{itemize}
\end{remark}

 We need several preliminaries to describe the essential image of $\alpha_n^{*}$. 
 \begin{proposition}\label{propt.relate with representations}
     The category of perfect complexes with coefficients in $\bdr_{,m}$ on $\Spa(K)_v$ is equivalent to the category of continuous semilinear representations of $G_K$ on
perfect complexes of $B_{\dR,n}^+$-modules, here $B_{\dR,n}^+:=\bdr_{,m}(C)$. 
 \end{proposition}
 \begin{proof}
     For $m=1$, this is \cite[Theorem 2.1]{anschutz2021fourier}. For general $n$, we need to check the $v$-descent of perfect complexes with coefficients in $\bdr_{,m}$ on perfectoid spaces, which could be reduced to \cite[Theorem 2.1]{anschutz2021fourier} by a standard d\'evissage.
 \end{proof}

 We need the following definition of nearly de Rham representations motivated from \cite[Notation 1.8]{liu2023rham}:
 \begin{definition}\label{not.nearly}
     $M\in \Perf(\bdr_{,n}(C))$ is called \textit{nearly de Rham} if all of the cohomology groups of $M\otimes_{\bdr_{,n}(C)}^{\mathbb{L}} C$ are nearly Hodge-Tate, i.e. all of the Sen weights of $H^{i}(M\otimes_{\bdr_{,n}(C)}^{\mathbb{L}} C)$ are in the subset $\mathbb{Z}+E^{\prime}(\pi)^{-1}\mathfrak{m}_{\mathcal{O}_{\overline{K}}}$.
 \end{definition}
 \begin{proposition}\label{prop.nearly de Rham image}
     Let $M$ be a perfect complex of $\bdr_{,n}(C)$-modules equipped with a (continuous) semilinear $G_K$ action. Then under the equivalence in \cref{propt.relate with representations}, $M$ lies in the essential image of $\alpha_n^{*}$ if and only if $M$ is nearly de Rham 
     in the sense of Notation \ref{not.nearly}.
\end{proposition}
\begin{proof}
    By induction on the amplitude and via considering cones, we reduce to the case that $M$ is concentrated on degree $0$. In other words, we need to show that if $M\in \Rep_{B_{\dR,n}^+}^{\fp}(G_K)$ is nearly de Rham (i.e. $M/E$) is nearly Hodge-Tate, then $M$ lives in the essential image of $\alpha_n^{*}$. We do this by induction on $n$ imitating Fontaine's proof in \cite[Theorem 3.6]{fontaine2004arithmetique}. When $n=1$, this is \cite[Lemma 4.6]{anschutz2022v}. For $n\geq 2$, suppose we have shown results up to $n-1$. Then for a nearly de Rham $M\in \Rep_{B_{\dR,n}^+}^{\fp}(G_K)$, we have the following short exact sequence 
    \[0\to E^{n-1}M\to M\to M/E^{n-1}\to 0.\]
    Without loss of generality, by induction we assume $E^{n-1}M=\alpha_{1}^{*}(Y)$ and that $M/E^{n-1}=\alpha_{n-1}^{*}(Z)$ for some $Y\in \Vect(\TXV)[1/p]$ and $Z\in \Vect(\TXM)[1/p]$. We wish to construct an extension $\mathscr{G}$ of $Z$ by $Y$ in $\Vect(\TXn)[1/p]$ such that $\alpha_{n}^{*}(\mathscr{G})=M$.

    For this purpose, first notice that we could find $\mathscr{F} \in \Vect(\TXn)[1/p]$ such that $i_{[n-1]}^*\mathscr{F}\cong Z$ and that $\alpha_{n}^{*}(\mathscr{F})\cong M$ as a finite projective $B_{\dR,n}^+$-module (but might not be $G_K$-equiavriant). Actually, by \cref{cor.analytic-variant}, it suffices to find an object $F$ in $\Vect(\mathrm{MIC}(\mathfrak{S}/E^{n}[1/p]))$  lifting $\beta_{n-1}(Z)=(\rho_{n-1}^*Z[1/p],\Theta_Z) \in \Vect(\mathrm{MIC}(\mathfrak{S}/E^{n-1}[1/p]))$ satisfying that $F\subseteq M$ is a full rank $\mathfrak{S}/E^{n}[1/p]$-lattice (then $F\otimes_{\mathfrak{S}/E^{n}[1/p]} B_{\dR,n}^+=M$ as a $B_{\dR,n}^+$-module). This could be done as follows: 
    first we pick a $\mathfrak{S}/E^{n-1}[1/p]$-basis of $\rho_{n-1}^*Z[1/p]$, as \cref{remt.factor through beta} implies that $$\rho_{n-1}^*Z[1/p]\otimes_{\mathfrak{S}/E^{n-1}[1/p]} B_{\dR,n-1}^+=\alpha_{n-1}^{*}(Z)=M/E^{n-1},$$
    such a basis also forms a basis of $M/E^{n-1}$, which could be lifted to a basis $\{e_i\}\in M$. Then we consider the $\mathfrak{S}/E^{n}[1/p]$-module generated by $\{e_i\}$ inside $M$, denoted as $N_0$. Equip $N_0$ with any Sen operator $\Theta_{N_0}$ lifting that on $N_0\otimes_{\mathfrak{S}/E^{n}[1/p]}\mathfrak{S}/E^{n-1}[1/p]=\rho_{n-1}^*Z[1/p]$, then $F=N_0$ satisfies the desired properties.

    Under \cref{cor.analytic-variant}, such a pair $(N_0,\Theta_{N_0})$ determines an object in \\$\Vect(\TXn)[1/p]$, denoted as $\mathscr{F}$. Then $\alpha_{n}^{*}(\mathscr{F})\cong M$ (as a finite projective $B_{\dR,n}^+$-module, but might not be $G_K$-equaivariant) actually induces $G_K$-equaivariant isomrphisms $$\alpha_{n}^{*}(\mathscr{F})/E^{n-1}\cong M/E^{n-1},\quad E^{n-1} \alpha_{n}^{*}(\mathscr{F})\cong E^{n-1}M.$$ 
    
    Now as the $G_K$-structure on $M/E^{n-1}$ is determined by $Z$, or equivalently, by $\beta_{n-1}(Z)=(\rho_{n-1}^*Z[1/p],\Theta_Z)$ and that the $G_K$-structure on $\alpha_{n}^{*}(\mathscr{F})/E^{n-1}$ is determined by $i_{[n-1]}^*(\mathcal{F})\cong Z$, the desired $G_K$-equivariance follows.  

    Consequently, $[\alpha_{n}^{*}(\mathscr{F})]$ determines a class in $\Ext_{G_K}^1(M/E^{n-1},E^{n-1}M)$, the extension group in $\Rep_{B_{\dR,n}^+}^{\mathrm{fg}}(G_K)$ classifying $G_K$-equaivariant extensions of $E^{n-1}M$ by $M/E^{n-1}$. Moreover, the isomorphism $\alpha_{n}^{*}(\mathscr{F})\cong M$ implies the vanishing of $[\alpha_{n}^{*}(\mathscr{F})]-[M]$ in $\Ext_{B_{\dR,n}^+}^1(M/E^{n-1},E^{n-1}M)$, the extension group classifying extensions of $E^{n-1}M$ by $M/E^{n-1}$ as finite generated $B_{\dR,n}^+$-modules, hence $$[\alpha_{n}^{*}(\mathscr{F})]-[M] \in \Ext_{0,G_K}^1(M/E^{n-1},E^{n-1}M),$$ the subgroup of $\Ext_{G_K}^1(M/E^{n-1},E^{n-1}M)$ consisting of those extensions split viewed as extensions of $B_{\dR,n}^+$-modules (but the splitting might not respect the $G_K$-action). 
    
    Similarly, one could define $\Ext_{\Prism}^1(Z,Y)$ as the extension group in $\Qcoh(\TXn)[1/p]$ classifying extensions of $Z$ by $Y$, both viewed as in the category of the isogeny category of (non-derived) quasi-coherent sheaves on $\TXn$ via pushforwards and $\Ext_{\Mod}^1(Z,Y)$ to be the extension group in $\Mod(\mathfrak{S}/E^n[1/p])$ classifying extensions of $\rho_{n-1}^*Z[1/p]$ by $\rho_{1}^*Z[1/p]$,
    both viewed as $\mathfrak{S}/E^n[1/p]$-modules. Finally, we define $\Ext_{0, \Prism}^1(Z,Y)$ as the subgroup classifying those extensions split after pullback along $\rho_n^*$. Then we have the following diagram
    \[\xymatrixcolsep{0.8pc}\xymatrix{\Ext_{0, \Prism}^1(Z,Y)\ar[r] \ar[d]& \Ext_{\Prism}^1(Z,Y)\ar[r]\ar[d]&\Ext_{\Mod}^1(Z,Y)\ar[d]\\\Ext_{0, G_K}^1(M/E^{n-1},E^{n-1}M)\ar[r]& \Ext_{G_K}^1(M/E^{n-1},E^{n-1}M)\ar[r]&\Ext_{B_{\dR,n}^+}^1(M/E^{n-1},E^{n-1}M)}\]
    which is exact in the middle of each row. Also, the first arrow in each row is injective.

    Then we claim that the first vertical map is bijective. Actually, \cref{cor.analytic-variant} implies that $\Ext_{0, \Prism}^1(Z,Y)=\Ext_{0, \Prism}^1(i_1^*Z,Y)$ essentially due to the fact that $$\hom_{\Mod(\mathfrak{S}/E^n[1/p])}(\rho_{n-1}^*Z[1/p],\rho_{1}^*Y[1/p])=\hom_{\Mod(\mathfrak{S}/E[1/p])}(\rho_{n-1}^*Z[1/p]/E,\rho_{1}^*Y[1/p]).$$

    Similarly, $\Ext_{0, G_K}^1(M/E^{n-1},E^{n-1}M)=\Ext_{0, G_K}^1(M/E,E^{n-1}M)$. But now as the statements hold for $n=1$, we see that $\Ext_{0, \Prism}^1(i_1^*Z,Y)=\Ext_{0, G_K}^1(M/E,E^{n-1}M)$, hence the first vertical map is bijective.
    
    Combining the bijection of the first vertical map and that $[\alpha_{n}^{*}(\mathscr{F})]-[M] \in \Ext_{0,G_K}^1(M/E^{n-1},E^{n-1}M)$, we see that $M$ lies in the essential image of $\alpha_{n}^{*}$, we win.
\end{proof}

As a by-product of \cref{propt.relate with representations} and the full faithfulness of the functor passing from de Rham prismatic crystals to nearly de Rham representations obtained in \cite{liu2023rham}, we obtain the following classification of (truncated) de Rham prismatic crystals in perfect complexes.
\begin{corollary}\label{cort.compare with de Rham}
Let $n\in \mathbb{N}$. Then the base change functor $\ocn \to \bdr_{,n}$ induces an equivalence of categories 
    $$\Perf(\TXn)[1/p] \stackrel{\sim}{\longrightarrow} \Perf((\mathcal{O}_{K})_{\Prism}, \bdr_{,n}).$$
In particular, we get the following two equivalent descriptions of de Rham prismatic crystals in perfect complexes:
\begin{itemize}
    \item The category of complexes $M\in \Perf(\mathrm{MIC}(\mathfrak{S}/E^n[1/p]))$ such that $H^\ast(M)$ is finite dimensional over $\mathfrak{S}/E^n[1/p]$ and the action of $\Theta^p-(E^{\prime}(u))^{p-1}\Theta$ on  $H^\ast(M)$ is topologically nilpotent.
    \item The category of $n$-truncated nearly de Rham perfect complexes, i.e. perfect complexes $M$ of $B_{\dR,n}^+$-modules equipped with a (continuous) semilinear $G_K$ action all of the cohomology groups of $M\otimes^L_{B_{\dR,n}^+} C$ are nearly Hodge-Tate representations of $G_K$. 
\end{itemize}
When passing to the inverse limit, we see that $\Perf((\mathcal{O}_{K})_{\Prism}, \bdr_{})$ is equivalent to the following two categories:
\begin{itemize}
    \item The category of complexes $M\in \Perf(\mathrm{MIC}(\bdr(\mathfrak{S})))$ such that $H^\ast(M)$ is finite dimensional over $\bdr(\mathfrak{S})$ and the action of $\Theta^p-(E^{\prime}(u))^{p-1}\Theta$ on  $H^\ast(M)$ is topologically nilpotent (with respect to the $(p,E)$-adic topology).
    \item The category of nearly de Rham perfect complexes, i.e. perfect complexes $M$ of $B_{\dR}^+$-modules equipped with a (continuous) semilinear $G_K$ action all of the cohomology groups of $M\otimes^L_{B_{\dR}^+} C$ are nearly Hodge-Tate representations of $G_K$.
\end{itemize}
\end{corollary}
\begin{proof}
 It suffices to prove the equivalence in the first sentence, then the desired results follow from \cref{cor.analytic-variant,thm.full faithfulness of pullback,prop.nearly de Rham image}. By \cref{remt.factor through beta}, we have the following commutative diagram 
 \[\begin{tikzcd}[cramped, sep=large]
\Perf(\TXn)[1/p] \arrow[rd, "\alpha_{n}^{*}"] \arrow[r, ""] & \Perf((\mathcal{O}_{K})_{\Prism}, \bdr_{,n}) \arrow[d, ""]\\
& \{n\text{-truncated nearly de Rham perfect complexes}\}
\end{tikzcd}\]
Notice that the usual
truncations equip the target of the right vertical map with a t-structure whose heart is the usual category of continuous
semilinear representations of $G_K$ on finite dimensional $B_{\dR,n}^+$-vector spaces satisfying the nearly de Rham condition. Moreover, $\Perf((\mathcal{O}_{K})_{\Prism}, \bdr_{,n})$ is also equipped with a $t$-structure whose heart is just $\Vect((\mathcal{O}_{K})_{\Prism}, \bdr_{,n})$ (for a detailed study of $t$-structures on prismatic crystals, one could see \cite[Section 2]{guo2023frobenius}). When restricted to the heart, the right vertical map is fully faithful by a slightly variant version of \cite[Theorem 5.12]{liu2023rham} (the statement there is for de Rham prismatic crystals, but its proof still works for truncated de Rham prismatic crystals), hence so is the right vertical map.

As we have shown $\alpha_{n}^{*}$ is an equivalence in \cref{thm.full faithfulness of pullback,prop.nearly de Rham image}, we see the horizontal morphism is also an equivalence of categories.
\end{proof}

 Finally we combine all of the ingredients in this section to get a description of $\Perf(\Spa(K)_v, \bdr_{,n})$, generalizing \cite[Theorem 4.9]{anschutz2022v}.
 \begin{theorem}\label{thm.main p adic RH correspondencd}
      Let $n\in \mathbb{N}$. For any finite Galois extension $L/K$ the functor
  \[
   \alpha_{n,L}^{*}\colon \Perf(\TXL)[1/p]\to \Perf(\Spa (L)_v, \bdr_{,n})
  \]
  is fully faithful and induces a fully faithful functor
  \[
   \alpha_{n,L/K}^{*} \colon \Perf([\TXL/{\Gal(L/K)}])[1/p] \to \Perf(\Spa(K)_v, \bdr_{,n})
  \]
  on $\Gal(L/K)$-equivariant objects. 
Each $\mathcal{E}\in \Perf(\Spa(K)_v, \bdr_{,n})$ lies in the essential image of some $\alpha_{n,L/K}^{*}$. 
Consequently, we get an equivalence
  \[
   2\text{-}\varinjlim\limits_{L/K} \Perf([\TXL/\mathrm{Gal}(L/K)])[1/p] \cong \Perf(\mathrm{Spa}(K)_v, \bdr_{,n}),
   \]
   where $L$ runs over finite Galois extensions  of $K$ contained in $\overline{K}$. 

   %Moreover, when $n\leq 1+\frac{p-1}{e}$, all of the above results holds if we replace 
 \end{theorem}
 \begin{proof}
     $\alpha_{n,L}^{*}$ is fully faithful by applying \cref{thm.full faithfulness of pullback} to the $p$-adic field $L$. Passing to $\Gal(L/K)$-invariant objects induces the functor $\alpha_{n,L/K}^{*}$, which is again fully faithful by finite \'etale descent and the full faithfulness of $\alpha_{n,L}^{*}$. For any $V\in \Rep_{B_{\dR,n}^+}^{\fp}(G_K)$ (i.e. $V$ is a finite projective $B_{\dR,n}^+$-module equipped with a continuous semilinear $G_K$-action), $V|_{G_L}$ is nearly de Rham (as a representation of $G_L$) for some large enough finite Galois extension $L$ over $K$, hence it lies in the essential image of $\alpha_{n,L}^{*}$ by \cref{prop.nearly de Rham image}. Moreover, as $V|_{G_L}$ is $\Gal(L/K)$-equivariant and $\alpha_L$ is fully faithful, it actually lives in the essential image of $\alpha_{n,L/K}^{*}$. 
 \end{proof}

\section{Applications: certain truncated prismatic crystals on $\mathbb{Z}_p/p^m$}
\subsection{Classification of certain truncated prismatic crystals over $\mathbb{Z}_p/p^m$}
In this section, we classify $\Perf((\mathbb{Z}_p/p^m)_{\Prism}, \mathcal{O}_{\Prism}/\mathcal{I}_{\Prism}^n)
$ for $n\leq p$ via methods developed in the previous sections.   
\begin{construction}[The diffracted $n$-truncated Cartier-Witt stack]\label{const.appl relative}
Let $X$ be a bounded $p$-adic formal scheme. We have the structure map $X_{n}^{\Prism}\to \mathrm{WCart}_n$ defined in Construction \ref{ConsAbsHT}. Form a fiber square
\[ \xymatrix{ X_{n}^{\slashed{D}} \ar[r] \ar[d] &  X_{n}^{\Prism} \ar[d] \\
\Spf(\mathbb{Z}_p[[\lambda]]/\lambda^n) \ar^{\rho_n}[r] & \mathrm{WCart}_n }\]
%\[ \WCart_X^{\mathrm{HT}} = \WCart_X \times_{\WCart} \WCart^{\mathrm{HT}}.\] 
where $\rho_n$ in the bottom is defined in Section 3. We call $X_{n}^{\slashed{D}}$ the \textit{diffracted $n$-truncated prismatization of} $X$. More explicitly, given any $p$-nilpotent $\mathbb{Z}_p[[\lambda]]/\lambda^n$-algebra $S$,
\[X_{n}^{\slashed{D}}(S)=\mathrm{Map}(\Spec(W(S)/^{\mathbb{L}} \lambda), X),\]
where the mapping space is taken in the category of animated commutative rings. By abuse of notation, we will still denote the top arrow as $\rho_n: X_{n}^{\slashed{D}} \to X_{n}^{\Prism}$ or just $\rho$ if there is no ambiguity. 
\end{construction}
\begin{remark}\label{rem.appl. compatibility when n=1}
    When $n=1$, our definition of $X_{1}^{\slashed{D}}$ is not exactly the same as $X^{\slashed{D}}$ defined in \cite[Construction 3.8]{bhatt2022prismatization} as they use $\Tilde{\rho}: \Spf(\mathbb{Z}_p) \to \mathrm{WCart}^{\HT}$ (sending a test algebra $S$ to the Cartier-Witt divisor $W(S)\xrightarrow{\cdot V(1)} W(S)$, denoted as $\eta$ in \cite{bhatt2022absolute}) instead of $\rho_1$ in the bottom line, but we will see that this slight variance doesn't matter for the purpose of studying quasi-coherent sheaves on $X_{n}^{\slashed{D}}$, see the next subsection for details.
\end{remark}

For simplicity, from now on we fix $Y=\Spf(\mathbb{Z}_p/p^m)$ with $m\geq 2$ and $\mathfrak{S}=\mathbb{Z}_p[[\lambda]]$ in this section.
\begin{lemma}\label{lem.appl description of points}
The functor sending a $p$-nilpotent $\mathbb{Z}_p[[\lambda]]/\lambda^n$-algebra $S$ to $Y_{n}^{\slashed{D}}(S)$ is represented by $\Spf(\mathfrak{S}\{\frac{p^m}{\lambda}\}_{\delta}^{\wedge}/\lambda^n)$.
\end{lemma}
\begin{proof}
    For a $p$-nilpotent test algebra $S$ over $\mathbb{Z}_p[[\lambda]]/\lambda^n$ with structure morphism $f: \mathbb{Z}_p[[\lambda]]/\lambda^n \to S$, by definition we have that 
    $$Y_{n}^{\slashed{D}}(S)=\mathrm{Map}(\mathbb{Z}_p/p^m, W(S)/{}^{\mathbb{L}} \lambda),$$
    where the mapping space is calculated in $p$-complete animated rings. As the animated rings $\mathbb{Z}_p/p^m$ and $W(S)/{}^{\mathbb{L}} \lambda$ are obtained from $\mathbb{Z}_p$ and $W(S)$ by freely setting $p^m$ and $\lambda$ to be zero respectively and that $\mathbb{Z}_p$ is the initial object in $p$-complete animated rings, the above then simplifies to
    \begin{equation*}
        Y_{n}^{\slashed{D}}(S)=\{x\in W(S)|~p^m=\lambda x\}.
    \end{equation*}
    Given any $x\in Y_{n}^{\slashed{D}}(S)$, the unique $\delta$-ring map $\Tilde{f}: \mathbb{Z}_p[[\lambda]] \to W(S)$ lifting the structure map $f: \mathbb{Z}_p[[\lambda]]\to S$ extends uniquely to a $\delta$-ring map $\Tilde{f}_x: \mathfrak{S}\{\frac{p^m}{\lambda}\}_{\delta}^{\wedge} \to W(S)$ by sending $\delta^{i}(\frac{p^m}{\lambda})$ to $\delta^{i}(x)$ ($i\geq 0$) due to the universal property of $\mathfrak{S}\{\frac{p^m}{\lambda}\}_{\delta}^{\wedge}$. Consider the composition of the projection $W(S)\to S$ and $\Tilde{f}_x$, we obtain a ring homomorphism $f_x: \mathfrak{S}\{\frac{p^m}{\lambda}\}_{\delta}^{\wedge} \to S$, which further factors through (by abuse of notation) $f_x: \mathfrak{S}\{\frac{p^m}{\lambda}\}_{\delta}^{\wedge}/\lambda^n \to S$ as $\lambda^n=0$ in $S$. Consequently we get $f_x\in \Spf(\mathfrak{S}\{\frac{p^m}{\lambda}\}_{\delta}^{\wedge}/\lambda^n)(S)$.

    Conversely, given a morphism $f: \mathfrak{S}\{\frac{p^m}{\lambda}\}_{\delta}^{\wedge}/\lambda^n \to S$, by precomposing it with $\mathfrak{S}\{\frac{p^m}{\lambda}\}_{\delta}^{\wedge} \to \mathfrak{S}\{\frac{p^m}{\lambda}\}_{\delta}^{\wedge}/\lambda^n$, we obtain  a morphism $\mathfrak{S}\{\frac{p^m}{\lambda}\}_{\delta}^{\wedge} \to S$, which will still be denoted as $f$ by abuse of notation. Then by the universal property of Witt rings, $f$ uniquely lifts to a $\delta$-ring morphism $\Tilde{f}: \mathfrak{S}\{\frac{p^m}{\lambda}\}_{\delta}^{\wedge} \to W(S)$. Then $x_{f}:=\Tilde{f}(\frac{p^m}{\lambda})$ determines an element in $W(S)$ satisfying that $p^mx_f=\lambda$, hence a point in $Y_{n}^{\slashed{D}}(S)$.

    To see $x_{f_x}=x$, we just need to notice that given $x\in  Y_{n}^{\slashed{D}}(S)$ , the $\Tilde{f}_x: \mathfrak{S}\{\frac{p^m}{\lambda}\}_{\delta}^{\wedge} \to W(S)$ constructed above is precisely the $\delta$-ring morphism lifting $f_x: \mathfrak{S}\{\frac{p^m}{\lambda}\}_{\delta}^{\wedge} \to S$ by construction.

    Finally, for the purpose of showing that $f_{x_f}=f$, it suffices to observe that given $f: \mathfrak{S}\{\frac{p^m}{\lambda}\}_{\delta}^{\wedge}/\lambda^n \to S$, $\Tilde{f}=\Tilde{f}_{x_f}$ by our construction. Then we are done.
\end{proof}

\begin{lemma}\label{lemt.appl extend eta}
Assume that $(\mathfrak{S}, E)$ is a Breuil-Kisin prism ($E$ is an Eisenstein polynomial). Then the $W(k)$-linear homomorphism $\eta: \mathfrak{S} \to  \mathfrak{S}[\epsilon]/\epsilon^2$ sending $u$ to $u+\epsilon E(u)$ extends uniquely to a $\delta$-ring homomorphism $ \mathfrak{S}\{\frac{p^m}{E}\}_{\delta}^{\wedge} \to \mathfrak{S}\{\frac{p^m}{E}\}_{\delta}^{\wedge}[\epsilon]/\epsilon^2$, which will still be denoted as $\eta$ by abuse of notation. Moreover, this further induces an $\eta: \mathfrak{S}\{\frac{p^m}{E}\}_{\delta}^{\wedge}/E^n \to \mathfrak{S}\{\frac{p^m}{E}\}_{\delta}^{\wedge}/E^n[\epsilon]/\epsilon^2$ after modulo $E^n$ for any $n\in \mathbb{N}$.
\end{lemma}
\begin{proof}
It's not hard to see that $\mathfrak{S}\{\frac{p^m}{E}\}_{\delta}^{\wedge}[\epsilon]/\epsilon^2$ could be promoted to a $\delta$-ring by requiring that $\delta(\epsilon)=0$. Moreover, as $\eta(E)=E(u+\epsilon E(u))=E(u)(1+\epsilon E^{\prime}(u))$, we have the following holds in $\mathfrak{S}\{\frac{p^m}{E}\}_{\delta}^{\wedge}[\epsilon]/\epsilon^2$:
\begin{equation*}
    p^m=\eta(E)(1-\epsilon E^{\prime}(u))\cdot \frac{p^m}{E}.
\end{equation*}
Hence $\eta$ extends to a unique $\delta$-ring homomorphism sending $\frac{p^m}{E}$ to $(1-\epsilon E^{\prime}(u))\frac{p^m}{E}$ by the universal property of the source. One could verify this is the unique $\delta$-ring homomorphism extending that on $\mathfrak{S}$. For the moreover part, just notice that $\eta$ preserves the $E$-adic filtration as it sends $E$ to $E(u)(1+\epsilon E^{\prime}(u))$.
\end{proof}
\begin{remark}\label{rem.appl. eta on the structure sheaf}
    Notice that $\mathfrak{S}\{\frac{p^m}{E}\}_{\delta}^{\wedge}$ is the $p$-adic completion of $\mathfrak{S}[\delta^{i}(T)_{i\geq 0}]/(\delta^{i}(Et-p^m)_{i\geq 0})$, under this explicit interpretation, one can show that $\eta(u)=E(u)$ and
    $$
        \eta(t)=t(1-\epsilon E^{\prime}(u)),\qquad \qquad~\eta(\delta^i(t))=\delta^i(t)+(-1)^{i-1}(\prod_{j=0}^{i-1}\delta^{j}(t))^{p-1}tE^{\prime}(u)\epsilon, ~~~i\geq 1.
    $$
\end{remark}

Recall that we could identify $Y_{n}^{\slashed{D}}(S)$ with $\Spf(\mathfrak{S}\{\frac{p^m}{\lambda}\}_{\delta}^{\wedge}/\lambda^n)$ thanks to \cref{lem.appl description of points}. Hence Construction \ref{const.appl relative} gives us $\rho: \Spf(\mathfrak{S}\{\frac{p^m}{\lambda}\}_{\delta}^{\wedge}/\lambda^n) \to Y_n^{\Prism}$. Next we could apply the trick used in previous sections to construct the Sen operator on $\rho^*\mathscr{E}$ for $\mathscr{E}\in Y_n^{\Prism}$ when $n\leq p$. 
\begin{proposition}\label{appl.prop.key automorphism for small n}
    For $n\leq p$, the element $b$ constructed in \cref{lem.construct b} induces an isomorphism $b$ between functors $\rho: \Spf(\mathfrak{S}\{\frac{p^m}{\lambda}\}_{\delta}^{\wedge}/\lambda^n) \to Y_n^{\Prism}$
    and $\rho\circ \eta: \rho: \Spf(\mathfrak{S}\{\frac{p^m}{\lambda}\}_{\delta}^{\wedge}/\lambda^n) \to Y_n^{\Prism}$ after base change to $\Spec(\mathbb{Z}[\epsilon]/(\epsilon^2))$, i.e. we have the following commutative diagram:
\[\xymatrixcolsep{5pc}\xymatrix{\Spf(\mathfrak{S}\{\frac{p^m}{\lambda}\}_{\delta}^{\wedge}/\lambda^n)\times \Spec(\mathbb{Z}[\epsilon]/(\epsilon^2))\ar[d]^{\rho}\ar[r]^{\eta: \lambda \mapsto (1+\epsilon)\lambda}& \Spf(\mathfrak{S}\{\frac{p^m}{\lambda}\}_{\delta}^{\wedge}/\lambda^n)\times \Spec(\mathbb{Z}[\epsilon]/(\epsilon^2))\ar@{=>}[dl]^b \ar[d]_{}^{\rho}
\\Y_n^{\Prism}\times \Spec(\mathbb{Z}[\epsilon]/(\epsilon^2))  \ar[r]&Y_n^{\Prism}\times \Spec(\mathbb{Z}[\epsilon]/(\epsilon^2))}\]
\end{proposition}
\begin{proof}
   By abuse of notation, we regard $b$ as an element in $W(\mathfrak{S}\{\frac{p^m}{\lambda}\}_{\delta}^{\wedge}/\lambda^n)$ via the structure morphism $W(\mathfrak{S}/\lambda^n [\epsilon]/\epsilon^2) \to W(\mathfrak{S}\{\frac{p^m}{\lambda}\}_{\delta}^{\wedge}/\lambda^n[\epsilon]/\epsilon^2)$. Then the proof of \cref{prop.key automorphism for small n} still works by replacing $\delta$ used there with $\eta$ constructed in \cref{lemt.appl extend eta}.
\end{proof}

Then following the discussion before \cref{rem.describe the category} with \cref{appl.prop.key automorphism for small n} as the input replacing \cref{prop.key automorphism for small n}, we obtain the following result.
\begin{corollary}\label{cor.appl.produce sen operator}
    Fix $n\leq p$. The pullback along $\rho_n: \Spf(\mathfrak{S}\{\frac{p^m}{\lambda}\}_{\delta}^{\wedge}/\lambda^n) \to Y_n^{\Prism}$ induces a functor 
    \begin{equation*}
        \beta_n^{+}: \mathcal{D}(Y_n^{\Prism}) \rightarrow \mathcal{D}(\mathrm{MIC}(\mathfrak{S}\{\frac{p^m}{\lambda}\}_{\delta}^{\wedge}/\lambda^n)), \qquad \mathscr{E}\mapsto (\rho^{*}(\mathscr{E}),\Theta_{\mathscr{E}}).
    \end{equation*}
    Here $\mathcal{D}(\mathrm{MIC}(\mathfrak{S}\{\frac{p^m}{\lambda}\}_{\delta}^{\wedge}/\lambda^n)$ is defined similar to \cref{def.MIC} with $\delta$ there replaced with $\eta$.
\end{corollary}
\begin{example}\label{exam.appl.Sen on the structure sheaf}
    We would like to write the Sen operator $\Theta$ on the structure sheaf explicitly as this is needed in describing objects in $\mathcal{D}(\mathrm{MIC}(\mathfrak{S}\{\frac{p^m}{\lambda}\}_{\delta}^{\wedge}/\lambda^n)$ (as in \cref{rem.explicit description via sen operator} and \cref{lem.leibniz}). By \cref{rem.appl. eta on the structure sheaf}, under the identification of $\mathfrak{S}\{\frac{p^m}{\lambda}\}_{\delta}^{\wedge}/\lambda^n$ with $\mathfrak{S}[\delta^{i}(T)_{i\geq 0}]/(\lambda^n, \delta^{i}(\lambda t-p^m)_{i\geq 0})$, 
    $$\Theta(\lambda^k)=k\lambda^k,\qquad\qquad
    \Theta(\delta^{i}(t))=(-1)^{i-1}(\prod_{j=0}^{i-1}\delta^{j}(t))^{p-1}t,~~~~i\geq 0.$$
\end{example}

Finally we state the main result in this section, the classification of $n$-truncated prismatic crystals on $\Spf(\mathbb{Z}_p/p^m)_{\Prism}$ for $n\leq p$:
\begin{theorem}\label{thm.appl. main classification}
    Assume that $n\leq p$. The functor 
    \begin{equation*}
        \beta_n^{+}: \mathcal{D}(Y_n^{\Prism}) \rightarrow \mathcal{D}(\mathrm{MIC}(\mathfrak{S}\{\frac{p^m}{\lambda}\}_{\delta}^{\wedge}/\lambda^n), \qquad \mathscr{E}\mapsto (\rho^{*}(\mathscr{E}),\Theta_{\mathscr{E}}),
    \end{equation*}
    is fully faithful. Moreover, its essential image consists of those objects \\$M\in \mathcal{D}(\mathrm{MIC}(\mathfrak{S}\{\frac{p^m}{\lambda}\}_{\delta}^{\wedge}/\lambda^n))$ satisfying the following %pair of 
    conditions:
    \begin{itemize}
        %\item $M$ is $\mathbb{Z}_p$-complete.
        \item The action of $\Theta^p-\Theta$ on the cohomology $\mathrm{H}^*(M\otimes^{\mathbb{L}}\mathbb{F}_p)$ is locally nilpotent.
    \end{itemize}
\end{theorem}
\begin{proof}
    Given \cref{cor.appl.produce sen operator}, the strategy proving \cref{thm.main classification} still works once we show that the theorem holds for $n=1$, which is due to the next proposition. Notice that for any $M\in \mathcal{D}(\mathrm{MIC}(\mathfrak{S}\{\frac{p^m}{\lambda}\}_{\delta}^{\wedge}/\lambda^n))$, the underlying complex $M\in  \mathcal{D}(\mathfrak{S}\{\frac{p^m}{\lambda}\}_{\delta}^{\wedge}/\lambda^n)$ is already $p$-complete as $p^{nm}=0$ in $\mathfrak{S}\{\frac{p^m}{\lambda}\}_{\delta}^{\wedge}/\lambda^n$, hence we do not need to write this requirement separately as in \cref{thm.main classification}.
\end{proof}

The following proposition is used in the proof of the above theorem.
\begin{proposition}\label{prop.appl. base n=1 case}
    \cref{thm.appl. main classification} holds for $n=1$.  
\end{proposition}
\begin{proof}
As $\rho_1: \Spf(\mathbb{Z}_p) \to \WCart^{\HT}$ is a covering with automorphism group $\mathbb{G}_m^\sharp$ (see \cite[Theorem 3.4.13] {bhatt2022absolute} and \cite[Example 9.6]{bhatt2022prismatization}), then Construction \ref{const.appl relative} and \\ \cref{lem.appl description of points} implies that $$Y^{HT}=Y_1^{\Prism}=Y_{1}^{\slashed{D}}/\mathbb{G}_{m,\mathbb{Z}_p/p^m}^\sharp=\Spf(\mathfrak{S}\{\frac{p^m}{\lambda}\}_{\delta}^{\wedge}/\lambda)/\mathbb{G}_{m,\mathbb{Z}_p/p^m}^\sharp.$$
Moreover, unwinding the identification of $Y_{1}^{\slashed{D}}(S)$ with $\Spf(\mathfrak{S}\{\frac{p^m}{\lambda}\}_{\delta}^{\wedge}/\lambda)$ in \cref{lem.appl description of points}, we see that the $\mathbb{G}_{m,\mathbb{Z}_p/p^m}^\sharp$-action on $\mathfrak{S}\{\frac{p^m}{\lambda}\}_{\delta}^{\wedge}/\lambda$ is given (hence is also determined) by the usual scaling action on $\frac{p^m}{\lambda}$.

We have the following pullback diagram
\[\xymatrixcolsep{5pc} \xymatrix{ \Spf(\mathfrak{S}\{\frac{p^m}{\lambda}\}_{\delta}^{\wedge}/\lambda) \ar[r]^{\rho_1^{\prime}} \ar^{\pi^{\prime}}[d] &  Y^{HT}=\Spf(\mathfrak{S}\{\frac{p^m}{\lambda}\}_{\delta}^{\wedge}/\lambda)/\mathbb{G}_{m,\mathbb{Z}_p/p^m}^\sharp \ar^{\pi}[d] \\
Y=\Spf(\mathbb{Z}_p/p^m) \ar[r]^{\rho_1} & Y/\mathbb{G}_{m,\mathbb{Z}_p/p^m}^\sharp=\Spf(\mathbb{Z}_p/p^m)/\mathbb{G}_{m,\mathbb{Z}_p/p^m}^\sharp }\]

Consequently, we have that 
\begin{equation*}
    \mathcal{D}(Y^{\HT})=\Mod_{\pi_*\mathcal{O}}(\mathcal{D}(Y/\mathbb{G}_{m,\mathbb{Z}_p/p^m}^\sharp))=\Mod_{\rho_1^{*}\pi_*\mathcal{O}}(\mathcal{D}_{\nil}(\MIC(\mathbb{Z}_p/p^m)))%=\Mod_{\pi^{\prime}_*\rho_1^{\prime *}\mathcal{O}}(\mathcal{D}_{\nil}(\MIC(\mathbb{Z}_p/p^m)))
\end{equation*}

Here the first identity holds as $\pi$ is affine. The second equality follows as the proof of \cite[Theorem 3.5.8]{bhatt2022absolute} implies that the pullback along $\rho_1$ induces a fully faithful functor $$\mathcal{D}(Y/\mathbb{G}_{m,\mathbb{Z}_p/p^m}^\sharp) \to \mathcal{D}(\MIC(\mathbb{Z}_p/p^m))$$
with essential image consisting of those $M$ such that $\Theta_M: M\to M$ satisfies the nilpotence condition stated in \cref{thm.appl. main classification}, which is denoted as $\mathcal{D}_{\nil}(\MIC(\mathbb{Z}_p/p^m))$ for simplicity.

On the other hand, as $\rho_1^{*}\pi_*\mathcal{O}=\pi^{\prime}_*\rho_1^{\prime *}\mathcal{O}$, we end in the following:
\[\mathcal{D}(Y^{\HT})=\Mod_{\pi^{\prime}_*\rho_1^{\prime *}\mathcal{O}}(\mathcal{D}_{\nil}(\MIC(\mathbb{Z}_p/p^m)))\]

But the right-hand side is exactly the category of objects stated in \cref{thm.appl. main classification}, combining the fact that $$\pi^{\prime}_*\rho_1^{\prime *}\mathcal{O}=(\mathfrak{S}\{\frac{p^m}{\lambda}\}_{\delta}^{\wedge}/\lambda, \Theta)\in \mathcal{D}(\mathrm{MIC}(\mathbb{Z}_p/p^m))$$ for $\Theta$ described in \cref{exam.appl.Sen on the structure sheaf} and that the Sen operator on $M\otimes N$ is given by $1\otimes \Theta_N+\Theta_{M}\otimes 1$ for $M,N \in \mathcal{D}(\mathrm{MIC}(\mathbb{Z}_p/p^m))$, which already lies in the definition of $\mathcal{D}(\mathrm{MIC}(\mathfrak{S}\{\frac{p^m}{\lambda}\}_{\delta}^{\wedge}/\lambda))$.
\end{proof}

\subsection{Remark on compatibility with Petrov's result when $n=1$}
As pointed out in \cref{rem.appl. compatibility when n=1}, the diffracted Hodge-Tate stack $Y^{\slashed{D}}$ constructed in \cite{bhatt2022absolute} and \cite{bhatt2022prismatization} is slightly different than our $Y_{1}^{\slashed{D}}$. Actually, Alexander Petrov calculated $Y^{\slashed{D}}$ explicitly in \cite[Lemma 6.13]{petrov2023non} and hence obtain a presentation of $Y^{\HT}$:
\begin{lemma}(\cite[Lemma 6.13]{petrov2023non})\label{lem.appl. petrov}
$Y^{\slashed{D}}\simeq \mathbb{G}_{a,\mathbb{Z}_p/p^m}^\sharp$, hence $Y^{\HT}\simeq \mathbb{G}_{a,\mathbb{Z}_p/p^m}^\sharp/\mathbb{G}_{m,\mathbb{Z}_p/p^m}^\sharp$, where the quotient is taken with respect to the scaling action.
\end{lemma}
\begin{remark}\label{rem.appl. crystal compatible}
    Based on Petrov's result, arguing as in the proof of \cref{prop.appl. base n=1 case}, we could give another presentation of quasi-coherent complexes on $Y^{\HT}$:
    \[\mathcal{D}(Y^{\HT})=\Mod_{\mathcal{O}_{\mathbb{G}_{a,\mathbb{Z}_p/p^m}^{\sharp}}}(\mathcal{D}_{\nil}(\MIC(\mathbb{Z}_p/p^m))),\]
    where the Sen operator $\Theta$ on $\mathcal{O}_{\mathbb{G}_{a,\mathbb{Z}_p/p^m}^{\sharp}}$ sends $\frac{t^i}{i!}$ to $\frac{t^i}{(i-1)!}$ for the coordinate of $\mathcal{O}_{\mathbb{G}_{a,\mathbb{Z}_p/p^m}^{\sharp}}$.
\end{remark}
 
%$(\mathbb{Z}_p/p^m)_{\Prism}$ (following from  implicitly).For the purpose of studying  quasi-coherent complexes on $Y_n^{\Prism}$, we would like to understand $Y_{1}^{\slashed{D}}$ first. 
Recall that by \cref{lem.appl description of points}, we have another presentation 
\[Y_{1}^{\slashed{D}}\simeq \Spf(\mathfrak{S}\{\frac{p^m}{\lambda}\}_{\delta}^{\wedge}/\lambda),\qquad Y^{HT}\simeq \Spf(\mathfrak{S}\{\frac{p^m}{\lambda}\}_{\delta}^{\wedge}/\lambda)/\mathbb{G}_{m,\mathbb{Z}_p/p^m}^\sharp\]

We would like to compare $Y_{1}^{\slashed{D}}$ with $Y^{\slashed{D}}$, once we show they are isomorphic, it would be obvious that the description of quasi-coherent complexes on $Y^{\HT}$ via \cref{prop.appl. base n=1 case} should be compatible with that given in \cref{rem.appl. crystal compatible} based on Petrov's result.
\begin{proposition}\label{prop.appl. compatible of diffracted}
    Assume $p\geq 3$, then there is an isomorphism between $Y_{1}^{\slashed{D}}$ and $Y^{\slashed{D}}$ as functors on $\Spf(\mathbb{Z}_p)$.
\end{proposition}
\begin{remark}
    Given \cref{lem.appl description of points} and \cref{lem.appl. petrov}, it's tempting to show that $\mathfrak{S}\{\frac{p^m}{\lambda}\}_{\delta}^{\wedge}/\lambda \cong \mathcal{O}_{\mathbb{G}_{a,\mathbb{Z}_p/p^m}^{\sharp}}$ directly, however, this turns out to be extremely difficult due to the complexity of the prismatic envelope $\mathfrak{S}\{\frac{p^m}{\lambda}\}_{\delta}^{\wedge}$, hence we take another indirect method. 
\end{remark}
\begin{proof}[Proof of \cref{prop.appl. compatible of diffracted}]
    For a test $\mathbb{Z}_p$-algebra $S$, \cref{lem.appl.technical} and the proof of \cref{lem.appl description of points} shows that
    \[Y_{1}^{\slashed{D}}(S)=\{x\in W(S)| ~p^m=\lambda x=V(F(x_{\lambda}x))\}.\]
    
    On the other hand, \cite[Example 5.15]{bhatt2022prismatization} shows that
    \[Y^{\slashed{D}}(S)=\{y\in W(S)| ~p^m=V(1) y=V(F(y))\}.\]

    As $x_{\lambda}$ is a unit, sending $x$ to $x_{\lambda} x$ induces an isomorphism of these two functors.
\end{proof}

The following lemma is used in the above proof.

\begin{lemma}\label{lem.appl.technical}
    Let $\iota: \mathfrak{S}\to W(\mathbb{Z}_p)$ be the unique $\delta$-ring map lifting the quotient map $\mathfrak{S}\to \mathbb{Z}_p=\mathfrak{S}/\lambda$. Suppose $p\geq 3$, then there exists a unit $x_{\lambda}\in W(\mathbb{Z}_p)$ such that $\iota(\lambda)=V(F(x_{\lambda}))$. 
\end{lemma}
\begin{proof}
    For simplicity, we will just write $\lambda$ for $\iota(\lambda)$ in the following proof. First notice that $\lambda$ maps to $0$ under the projection $W(\mathbb{Z}_p) \to \mathbb{Z}_p$, hence it lies in $V(W(\mathbb{Z}_p))$. It then suffices to show that $\lambda=V(F(x))$ has a solution $x_{\lambda}=(x_0,x_1,\cdots)$ in $W(\mathbb{Z}_p)$. As $\mathbb{Z}_p$ is $p$-torsion free, the ghost map is injective, hence this equation is equivalent to that 
\begin{align}\label{appl.ghost identity}
    \forall n\geq 0, w_n(\lambda)=w_n(V(F(x))).
\end{align}
    We will construct $x_{\lambda}=(x_0,x_1,\cdots)$ inductively on $n$ by showing that the solution exists in $W_n(\mathbb{Z}_p)$.

    For $n=0$, $w_0(\lambda)=0 \in \mathbb{Z}_p$, $w_0(V(F(x)))=0$, hence \cref{appl.ghost identity} always holds.

    For $n\geq 1$, we have that 
    \begin{equation*}
        w_n(\lambda)=w_0(\varphi^n(\lambda))=w_0((\lambda+p)^{p^n}-p)=p^{p^n}-p,
    \end{equation*}
    and that
    \begin{equation*}
        w_n(V(F(x)))=pw_{n-1}(F(x))=pw_n(x)=p(\sum_{i=0}^n x_i^{p^{n-i}}p^i).
    \end{equation*}
    
    Take $n=1$, then \cref{appl.ghost identity} is equivalent to that $x_0^p+px_1=p^{p-1}-1$, hence it suffices to pick $x_0=-1, x_1=p^{p-2}$.
    
    Next we do induction on $n$. Suppose $n\geq 2$ and we have determined $x_0,\cdots,x_{n-1}$ such that $x_0=-1$, $v_p(x_i)=p^{i-1}(p-2)-\frac{p^{i-1}-1}{p-1}$ for $1\leq i\leq n-1$ and that \cref{appl.ghost identity} holds for non-negative integers no larger than $n-1$. Then we claim that we could pick $x_n\in \mathbb{Z}_p$ such that \cref{appl.ghost identity} holds for $n$ as well. Actually, as $x_0=-1$, the previous calculation  implies that we just need a $x_n$ such that 
    \begin{equation*}
        p^{p^n-1}=\sum_{i=1}^n x_i^{p^{n-i}}p^i.
    \end{equation*}
    
    Our assumption on $x_i$ for $1\leq i\leq n-1$ guarantees that 
    \begin{equation}\label{appl.ghost identity sub}
        v_p(x_i^{p^{n-i}}p^i)=i+p^{n-i}(p^{i-1}(p-2)-\frac{p^{i-1}-1}{p-1})=p^{n-1}(p-2)+i-\frac{p^{n-1}-p^{n-i}}{p-1},
    \end{equation}
hence $v_p(x_i^{p^{n-i}}p^i)$ decreases as $i$ increases for $1\leq i\leq n-1$ and they are all bounded above by $v_p(x_1^{p^{n-1}}p)=p^n-2p^{n-1}+1<p^n-1$, bounded below by $v_p(p^{n-1}x_{n-1}^{p})=n-1+p^{n-1}(p-2)-\frac{p^{n-1}-p}{p-1}$. Consequently, \cref{appl.ghost identity sub} has a unique solution $x_n\in \mathbb{Z}_p$ with $$v_p(x_n)=v_p(p^{n-1}x_{n-1}^{p})-n=p^{n-1}(p-2)-\frac{p^{n-1}-1}{p-1},$$ we win. Clearly such $x_{\lambda}$ is a unit in $W(\mathbb{Z}_p)$ as $x_0=-1$ by construction.

\end{proof}

\bibliographystyle{amsalpha}
\bibliography{main,preprints}

\def\cfudot#1{\ifmmode\setbox7\hbox{$\accent"5E#1$}\else \setbox7\hbox{\accent"5E#1}\penalty 10000\relax\fi\raise 1\ht7 \hbox{\raise.1ex\hbox to 1\wd7{\hss.\hss}}\penalty 10000 \hskip-1\wd7\penalty 10000\box7}
\providecommand{\bysame}{\leavevmode\hbox to3em{\hrulefill}\thinspace}
\providecommand{\MR}{\relax\ifhmode\unskip\space\fi MR }
\providecommand{\MRhref}[2]{%
  \href{http://www.ams.org/mathscinet-getitem?mr=#1}{#2}
}
\providecommand{\href}[2]{#2}
\begin{thebibliography}{GMW23b}

\bibitem[AB21]{anschutz2021fourier}
Johannes Ansch{\"u}tz and Arthur-C{\'e}sar~Le Bras, \emph{A fourier transform for banach-colmez spaces}, arXiv preprint arXiv:2111.11116 (2021).

\bibitem[AHB22]{anschutz2022v}
Johannes Ansch{\"u}tz, Ben Heuer, and Arthur-C{\'e}sar~Le Bras, \emph{{v-vector bundles on p-adic fields and Sen theory via the Hodge-Tate stack}}, arXiv preprint arXiv:2211.08470 (2022).

\bibitem[AHLB23]{anschutz2023hodge}
Johannes Ansch{\"u}tz, Ben Heuer, and Arthur-C{\'e}sar Le~Bras, \emph{{Hodge-Tate stacks and non-abelian $ p $-adic Hodge theory of v-perfect complexes on smooth rigid spaces}}, arXiv e-prints (2023), arXiv--2302.

\bibitem[BL22a]{bhatt2022absolute}
Bhargav Bhatt and Jacob Lurie, \emph{Absolute prismatic cohomology}, arXiv preprint arXiv:2201.06120 (2022).

\bibitem[BL22b]{bhatt2022prismatization}
\bysame, \emph{The prismatization of $ p $-adic formal schemes}, arXiv preprint arXiv:2201.06124 (2022).

\bibitem[BMS19]{bhatt2019topological}
Bhargav Bhatt, Matthew Morrow, and Peter Scholze, \emph{Topological hochschild homology and integral p-adic hodge theory}, Publications math{\'e}matiques de l'IH{\'E}S \textbf{129} (2019), no.~1, 199--310.

\bibitem[BS19]{BhattScholzepPrismaticCohomology}
Bhargav {Bhatt} and Peter {Scholze}, \emph{{Prisms and Prismatic Cohomology}}, arXiv e-prints (2019), arXiv:1905.08229.

\bibitem[BS21]{bhatt2021prismatic}
Bhargav Bhatt and Peter Scholze, \emph{{Prismatic $F$-crystals and crystalline Galois representations}}, arXiv preprint arXiv:2106.14735 (2021).

\bibitem[Fon04]{fontaine2004arithmetique}
Jean-Marc Fontaine, \emph{{Arithm{\'e}tique des repr{\'e}sentations galoisiennes p-adiques}}, Ast{\'e}risque \textbf{295} (2004), 1--115.

\bibitem[GL23]{guo2023frobenius}
Haoyang Guo and Shizhang Li, \emph{Frobenius height of prismatic cohomology with coefficients}, arXiv preprint arXiv:2309.06663 (2023).

\bibitem[GMW23a]{gao2023hodge}
Hui Gao, Yu~Min, and Yupeng Wang, \emph{Hodge--tate prismatic crystals and sen theory}, arXiv preprint arXiv:2311.07024 (2023).

\bibitem[GMW23b]{gao2022rham}
\bysame, \emph{{Prismatic crystals over the de Rham period sheaf}}, arXiv preprint arXiv:2206.10276 (2023).

\bibitem[GSQ23]{gros2023absolute}
Michel Gros, Bernard~Le Stum, and Adolfo Quir{\'o}s, \emph{Absolute calculus and prismatic crystals on cyclotomic rings}, arXiv preprint arXiv:2310.13790 (2023).

\bibitem[Liu23]{liu2023rham}
Zeyu Liu, \emph{{De Rham prismatic crystals over $\mathcal{O}_K$}}, Mathematische Zeitschrift \textbf{305} (2023), no.~4, 53.

\bibitem[Liu24a]{Liu24a}
Zeyu Liu, \emph{{A $p$-adic Riemann Hilbert functor via prismatization}}, In preparation (2024).

\bibitem[Liu24b]{Liu24b}
\bysame, \emph{{On the log Cartier-Witt stack beyond the Hodge-Tate locus}}, In preparation (2024).

\bibitem[Pet23]{petrov2023non}
Alexander Petrov, \emph{Non-decomposability of the de rham complex and non-semisimplicity of the sen operator}, arXiv preprint arXiv:2302.11389 (2023).

\bibitem[Sch17]{scholze2017etale}
Peter Scholze, \emph{{{\'E}tale cohomology of diamonds}}, arXiv preprint arXiv:1709.07343 (2017).

\end{thebibliography}

\end{document}